\documentclass{tac}
\usepackage{graphicx} 
\usepackage{tikz-cd}
\usepackage{fullpage}
\usepackage{amssymb}
\usepackage[nointegrals]{wasysym}
\usepackage{url}
\usepackage{enumerate}
\usepackage[numbers]{natbib} 

\tikzcdset{
	arrow style=tikz, diagrams={>=stealth},
	row sep=large, column sep=large
}

\usepackage{xcolor}

\usepackage{mathtools}
\usepackage{suffix}
\usepackage{amsfonts}
\usepackage{mathpartir}
\usepackage{enumitem}
\usepackage{stmaryrd}
\usepackage[all]{xy}
\usepackage{twoopt}

\newcommand\op{\mathrm{op}}

\makeatother

\newcommand\Indexed{\mathsf{Indexed}}
\newcommand\Fibered{\mathsf{Fibred}}

\newcommand\StrictIndexed{\mathsf{StrictInd}}

\newcommand\xxxx{\chi}
\newcommand\yyyy{\omega}
\newcommand\llll{\ell}
\newcommand\mmmm{\mathsf{m}}

\newcommand\Sub{\mathsf{Sub}}

\newcommand{\Fam}[2][]{\mathbf{Fam}_{#1}(#2)}

\newcommand{\Dist}[1]{\mathbf{Dist}(#1)}

\newcommand{\sumfam}[2]{[#1\mid #2]}
\newcommand{\prodfam}[2]{\langle #1\mid #2\rangle}

\newcommand\Top{\mathbf{Top}}
\newcommand\Pos{\mathbf{Pos}}
\newcommand\pSet{\mathbf{pSet}}

\newcommand\Vect{\mathbf{Vect}}

\newcommand\CMon{\mathbf{CMon}}

\newcommand\Dial{\mathbf{Dial}}
\newcommand\Dill{\mathbf{Dill}}

\newcommand\self{\mathrm{self}
}

\newcommand\SET{\mathbf{SET}}
\newcommand\CAT{\mathbf{CAT}}
\newcommand\Cat{\mathbf{Cat}}
\newcommand\twoCat{\mathbf{2Cat}}
\newcommand\twoCAT{\mathbf{2CAT}}

\newcommand\explainr[1]{&\mbox{\pushright{\color{gray}\scriptsize\{\;\textnormal{#1}\;\}}}}

\definecolor{shade}{RGB}{223,223,223}
\definecolor{unshade}{RGB}{255,255,255}

\usepackage{tcolorbox}
\newtcbox{\shadebox}{on line,arc=1pt, outer arc=2pt,%
	colback=shade,colframe=shade,boxsep=0pt,%
	left=1pt,right=1pt,top=2pt,bottom=2pt,%
	boxrule=0pt,bottomrule=1pt,toprule=1pt}

\newtcbox{\unshadebox}{on line,arc=1pt, outer arc=2pt,%
	colback=unshade,colframe=shade,boxsep=0pt,%
	left=1pt,right=1pt,top=2pt,bottom=2pt,%
	boxrule=0pt,bottomrule=1pt,toprule=1pt}

\newcommand{\wCpo}{\mathbf{\boldsymbol\omega CPO}}

\newcommand{\CTop}{\mathbf{Top}}
\newcommand{\Set}{\mathbf{Set}}

\newcommand\inv[1]{#1^{-1}}

\WithSuffix\newcommand\inv+[1]{\parent{#1}^{-1}}

\newcommand\initial{\mathbf{0}}
\newcommand\terminal{\mathbf{1}}

\newcommand\isomorphic\cong

\newcommand\cat[1]{\mathcal{#1}}

\newcommand\catB{\cat{B}}
\newcommand\catC{\cat{C}}
\newcommand\catD{\cat{D}}
\newcommand\catE{\cat{E}}

\newcommand\catM{\cat{M}}
\newcommand\catL{\cat{L}}
\newcommand\catS{\cat{S}}

\newcommand\catV{\cat{V}}

\newcommand\depproj[2]{\mathbf{p}_{#1,#2}}

\newcommand\algebras[2]{#1\text{-}\mathbf{Alg}(#2)}

\newcommand\leftmultimap{\multimap}

\newcommand\ev[1][]{\mathrm{ev}^{#1}}
\newcommand\evf[1][]{\mathrm{ev1}^{#1}}

\newcommand\id[1][]{{\mathrm{id}_{#1}}}

\newcommand\xto\xrightarrow

\newcommand\colim{\mathrm{colim}\,}
\renewcommand\lim{\mathrm{lim}\,}

\newcommand\ob[1]{\mathrm{ob}\,#1}

\newcommand{\defeq}{\stackrel {\mathrm{def}}=}

\newcommand{\compldom}{\partial^{c}}

\newcommand{\abstrcompldom}{\mathfrak{d}^{c}}

\newtheorem{counterexample}[theorem]{Counterexample}

\title{\bf Monoidal closure of Grothendieck constructions via
$\Sigma$-tractable monoidal structures and Dialectica formulas}
\author{Fernando Lucatelli Nunes and Matthijs V\'ak\'ar}
\address{Departement Informatica, Universiteit Utrecht, Nederland}
 \eaddress{(1):fernandolucatellinunes@gmail.com, and (2):m.i.l.vakar@uu.nl}
\date{\today}

\allowdisplaybreaks
\begin{document}

\maketitle

\begin{abstract}
	We examine the categorical structure of the Grothendieck construction 
	$\Sigma_{\mathsf{C}}\mathsf{L}$ of an indexed category 
	$\mathsf{L}\colon \mathsf{C}^{op}\to\mathsf{CAT}$. 
	Our analysis begins with characterisations of fibred limits, colimits, and 
	monoidal (closed) structures. 
	The study of fibred colimits leads naturally to a generalisation of the notion 
	of \emph{extensive indexed category} introduced in 
	\emph{CHAD for Expressive Total Languages}, 
	and gives rise to the concept of \emph{left Kan extensivity}, 
	which provides a uniform framework for computing colimits 
	in Grothendieck constructions. 
	
	We then establish sufficient conditions for the (non-fibred) monoidal closure 
	of the total category $\Sigma_{\mathsf{C}}\mathsf{L}$. 
	This extends G\"odel’s Dialectica interpretation, and rests upon a new 
	notion of \emph{$\Sigma$-tractable monoidal structure}. 
	Under this notion, $\Sigma$-tractable coproducts unify and extend 
	cocartesian coclosed structures, biproducts, and extensive coproducts. 
	Finally, we consider when the induced closed structure is fibred, 
	showing that this need not hold in general, even in the presence 
	of a fibred monoidal structure.
\end{abstract}

\setcounter{tocdepth}{1}

\tableofcontents
\newpage
\section*{Introduction}

It is a familiar fact that an indexed family of sets
$S(-) \colon I \to \Set$
can be represented equivalently by a pair of sets
$\sum_{i \in I} S_i$ and $I$, together with a projection
$\pi_1 \colon \sum_{i \in I} S_i \to I$.
In direct analogy, an indexed category
$L \colon \catC^{op} \to \CAT$
may be represented by a category $\Sigma_{\catC} L$ equipped with a projection
$\pi_1 \colon \Sigma_{\catC} L \to \catC$
which is a cloven fibration.
This correspondence between indexed categories and fibrations is the
\emph{Grothendieck construction}, also known as the
\emph{$\Sigma$-type of categories}.
It is a fundamental device, providing a categorical framework for
dependent structure, and it reappears throughout mathematics,
logic, and theoretical computer science.

One of our main objectives is to understand the structural properties
of categories of the form $\Sigma_{\catC} L$,
including the existence of limits, colimits, and closed structures.
While the behaviour of limits in Grothendieck constructions is well understood,
the situation for colimits is more subtle.
Several relevant results are part of the established folklore, yet,
to our knowledge, they have not been explicitly formulated in the literature,
in particular regarding their connection to (op)lax limits.
We provide necessary and sufficient conditions for the existence of
fibred limits and colimits, and relate these to suitable generalisations
of classical notions such as extensivity.

From these considerations we derive a general principle that affords
a uniform treatment of colimits in Grothendieck constructions.
This leads naturally to the notion of \emph{left Kan extensivity},
which generalises the concept of extensive indexed category introduced in~\cite{nunes2023chad},
and characterises when colimits in the total category arise coherently
from those in the base and in the fibres.

We then turn to the question of when the total category
$\Sigma_{\catC} L$ carries a monoidal or cartesian closed structure.
Our analysis extends ideas originating in G\"odel’s Dialectica interpretation
and is formulated in terms of a new notion of
\emph{$\Sigma$-tractable monoidal structure}.
This provides a unified account of the existence of closed structures
across a wide range of examples, and clarifies how the general properties of
Grothendieck constructions account for the behaviour of these
more intricate instances.

\paragraph*{\textbf{Contributions}}
In summary, this paper makes the following contributions:
\begin{enumerate}
	\item a proof of necessary and sufficient conditions for the existence of fibred limits in a Grothendieck construction $\Sigma_\catC \catL \to \catC$ (Theorem~\ref{thm:limits-groth});
	while similar results are known~\cite{MR0213413},
	our formulation appears to be new;
	\item a proof of necessary and sufficient conditions for the existence of \emph{fibred colimits} in a Grothendieck construction (Theorem~\ref{thm:grothendieck-colim});
	unlike the case of limits, this result seems to be novel in the literature;
	in particular, our characterisation of colimits introduces natural generalisations of the classical notion of extensive categories and the notion of extensive indexed categories from~\cite{nunes2023chad};
	\item necessary and sufficient conditions for fibred monoidal closure of a Grothendieck construction $\Sigma_\catC \catL \to \catC$ (Theorem~\ref{thm:fibred-monoidal-closed});
	\item the introduction of the notion of \emph{$\Sigma$-(co)tractable monoidal structure}, together with examples showing that many monoidal categories arising in practice satisfy this property (Section~\ref{sec:sigma-tractable-monoidal});
	\item a proof that $\Sigma$-cotractability, combined with the existence of $\Pi$-types (fibred products), yields sufficient conditions for a Grothendieck construction over a suitable\footnote{A model of dependent type theory with $\Pi$-types and strong $\Sigma$-types.} base category to be (non-fibred) monoidal closed, via a generalised Dialectica formula for exponentials (Theorem~\ref{theo:grothendieck-closed});
	\item several new examples of non-fibred monoidal closed and cartesian closed structures arising naturally from Grothendieck constructions (Section~\ref{sec:exponentials-in-grothendieck}).
\end{enumerate}

\paragraph*{\textbf{A remark on size and set-theoretic foundations}}
We work within a standard von Neumann–Bernays–Gödel set theory as a basis for our constructions.
In particular, we assume a predicative hierarchy of universes, of which we employ only the first three levels.
We refer to sets at these levels as \emph{small}, \emph{large}, and \emph{very large}.
Unless stated otherwise, all categories are assumed to be large but locally small, meaning that each hom-set is small.
We write $\Set$, $\Cat$, and $\twoCat$ for the large, locally small $(2$-)categories of small sets, small categories, and small $2$-categories, respectively.
We also use $\SET$, $\CAT$, and $\twoCAT$ for the corresponding very large, locally large $(2$-)categories of large sets, large categories, and large $2$-categories.

\section{Two-dimensional category theory}
We begin by fixing some conventions for terminology in $2$-dimensional category theory, with particular attention to pseudofunctors, oplax natural transformations, and oplax (co)limits. Our notation and conventions align with the standard usage in the modern literature on $2$-dimensional category theory; see, for example, \cite{zbMATH05659661, zbMATH05770299, arXiv:1711.02051, 2021arXiv211210072P}.  
We start by recalling the definition of a pseudofunctor, \textit{e.g.}~\cite{zbMATH0681654, MR3491845}.

\begin{definition}[Pseudofunctor]\label{def:pseudofunctor}
Let \(\mathcal C, \mathcal D\) be 2-categories.  A pseudofunctor 
\[
L \colon \mathcal C \to \mathcal D
\]
is a pair \( L = (L, \llll )\) consisting of the following data:
\begin{enumerate}[label=\textbf{(\alph*)}, leftmargin=2em, itemsep=0.3em]
  \item A function on objects \(L : \mathrm{Ob}(\mathcal C)\to \mathrm{Ob}(\mathcal D)\).
  \item For each pair of objects \(X,Y\in \mathcal C\), a functor
    \[
      L_{X,Y}\colon \mathcal C(X,Y) \;\to\; \mathcal D(LX,\,LY).
    \]
  \item For each composable pair 
    \[
      X \xrightarrow{g} Y \xrightarrow{h} Z
    \]
    of 1-cells 
    in \(\mathcal C\), an invertible 2-cell in \(\mathcal D\)
    \[
     L^{hg} = \llll _{hg} \colon L(h)\circ L(g)\;\Longrightarrow\; L(h\circ g).
    \]
  \item For each object \(X\in \mathcal C\), an invertible 2-cell in \(\mathcal D\)
    \[
      L^X = \llll_X \colon \mathrm{id}_{LX}\;\Longrightarrow\; L(\mathrm{id}_X).
    \]
\end{enumerate}
These data are required to satisfy the following three coherence axioms below.
\begin{itemize}
\item \textbf{Associativity:} Equation \eqref{Eq:Associativity}  holds for any composable triple of 1-cells 
$$ W \xrightarrow{f} X \xrightarrow{g} Y \xrightarrow{h} Z$$
in $\mathcal C$.

\item \textbf{Identity:} Equation \eqref{Eq:Identity} holds for any morphism $f: W\to X$ in \(\mathcal C\).

\normalsize
\item \textbf{Naturality:} Equation \eqref{Eq:Naturality}  holds for any pair of $2$-cells	$\xxxx : g\Rightarrow \hat{g} $ and $\yyyy : f\Rightarrow \hat{f} $ in $\mathcal C $. 	
\end{itemize}	
\begin{equation}\label{Eq:Associativity} \tag{A-PS}
\xymatrix{ 
LW\ar[rr]^{L(f)}\ar[dd]_{L(hgf)}\ar[ddrr]|{L(gf)}
&&
LX\ar[dd]^{L(g)}\ar@{}[dl]|{\xLeftarrow{\llll _ {{}_{gf}}}} 
&&
LW\ar[rr] ^{L(f) }\ar[dd]_{L(hgf)}\ar@{}[dr]|{\xLeftarrow{\llll_ {{}_{(hg)f}}}}
&&
LX\ar[dd]^{L(g)}\ar[ddll]|{L(hg)}
\\
&&&=&&&
\\
LZ\ar@{}[ru]|{\xLeftarrow{\llll _ {{}_{h(gf)}}}}
&&
LY\ar[ll]^{L(h)}
&&
LZ 
&&
LY\ar[ll]^{L(h)} \ar@{}[ul]|{\xLeftarrow{\llll _ {{}_{hg}}}} 							
}
\end{equation}
\begin{equation}\label{Eq:Identity}\tag{Id-PS} 
\xymatrix{  LW     \ar[rr]^{L(f)}\ar[dd]_{L(\id _ {{}_X}f)}&&
LX\ar@/_6ex/[dd]|{L(\id _ {{}_X}) }
                    \ar@{}[dd]|{\xLeftarrow{\llll _ {{}_X}} }
                    \ar@/^6ex/[dd]|{\id _ {{}_{LX}} }
&&
LW\ar[dd]_{L(f\id _ {{}_W})}
&& 
LW\ar@{=}[ll]\ar@/_5.5ex/[dd]|{L(\id _ {{}_W}) }
                    \ar@{}[dd]|{\xLeftarrow{\llll _ {{}_W}} }
                    \ar@/^5.5ex/[dd]|{\id _ {{}_{LW}} }
										&&
LW\ar@/_4ex/[dd]|{L(f) }
                    \ar@{}[dd]|{=}
                    \ar@/^4ex/[dd]|{L(f) }										&
\\
&\ar@{}[l]|{\xLeftarrow{\llll _ {{}_{\id _ {{}_X} f }}}} &
&=&
&\ar@{}[l]|{\xLeftarrow{\llll _ {{}_{f\id _ {{}_W}  }}}} &
\phantom{A}\ar@{}[rr]|{=}&&& \\										
LX\ar@{=}[rr]&&
LX
&&
LX &&LW\ar[ll]^{L(f)}																				
&& LX	&						}
\end{equation}
\begin{equation}\label{Eq:Naturality}\tag{Nat-PS}
\xymatrix{  LX\ar[dddd]_{L(\hat{h}\hat{g})}\ar@{=}[r] &LX     \ar@{=}[r] 
                    \ar[dd]_{L(\hat{g})} 
                       & 
              {LX}  \ar[dd]^{L(g) }      && LX\ar[dddd]_{L(\hat{h}\hat{g})}\ar@{=}[r] &LX     \ar[rr]^{L(g)} 
                    \ar[dddd]^{L(hg)} 
                       && 
              {LY}  \ar[dddd]_{L(h)}\\
							&&{}\ar@{}[l]|{\xLeftarrow{L(\xxxx )}} && &&&
															\\
              &{LY}\ar@{}[l]|{\xLeftarrow{\llll _ {{}_{\hat{h}\hat{g}}}}} \ar[dd]_{L(\hat{h})} \ar@{=}[r] &
              {LY}\ar[dd]^{L(h)}   &=&  &{}\ar@{}[l]|{\xLeftarrow{L(\yyyy\ast\xxxx ) }} &\xLeftarrow{\hskip .5em \llll _ {{}_{hg}} \hskip .5em  }&
              \\
							&&{}\ar@{}[l]|{\xLeftarrow{L(\yyyy )}} && &&&
															\\
								LZ\ar@{=}[r] &LZ\ar@{=}[r]& LZ && LZ\ar@{=}[r] &LZ\ar@{=}[rr]&& LZ }
\end{equation}

\end{definition}

\begin{definition}[Unitors and Compositors]
	Let $L = (L, \llll) : \catC \to \catD$ be a pseudofunctor. As already suggested in Definition~\ref{def:pseudofunctor}, for each object $X$ of $\catC$ and each composable pair of morphisms
	\[
	X \xrightarrow{g} Y \xrightarrow{h} Z,
	\]
	we adopt the shorthand notation
	\begin{equation}
		L^X \defeq \llll_X,
		\qquad
		L^{hg} \defeq \llll_{hg}.
	\end{equation}
		This terminology allows us, when convenient, to suppress explicit mention of~$\llll$ and to speak simply of the pseudofunctor
	\[
	L : \catC \to \catD.
	\]
		The invertible $2$-cell $L^X$ is referred to as the \emph{unitor} at $X$, while $L^{hg}$ is called the \emph{compositor} corresponding to the pair $(h,g)$. 
\end{definition}
\noindent
One may define $2$-functors as $\CAT$-enriched functors between $2$-categories.  
However, with the definition of pseudofunctor in place, we may state the following equivalent formulation.

\begin{definition}[$2$-Functor]
	A \emph{$2$-functor} between $2$-categories $\catC$ and $\catD$ is a pseudofunctor 
	\[
	L \colon \catC \to \catD
	\]
	for which the unitor $L^X$ and the compositor $L^{hg}$ are identities, for every object $X$ of $\catC$ and every composable pair of morphisms
	\[
	X \xrightarrow{g} Y \xrightarrow{h} Z
	\]
	in $\catC$.
\end{definition}

\begin{definition}[Co-opposite]
Every $2$-category $\catC $ has a \emph{co-opposite} (or \emph{dual}), denoted $\catC^{\mathrm{co}}$. The objects of 
$\catC^{\mathrm{co}}$ are the same as the objects of $\catC$, and for each pair of objects $A,B$ in $\catC$,
	\begin{equation} 
	\catC^{\mathrm{co}}\left( A, B \right)\defeq  \left( \catC\left( A, B \right)\right) ^\op. 
	\end{equation}  
The composition is, then,  defined by $\circ_{\catC} ^\op $, where $\circ _{\catC} $ is the composition functor of $\catC$.
More precisely, for each triple $(A,B,C)$  of objects, we define
\begin{equation} 
\left( \circ _{\catC ^{\mathrm{co}}}\right)_{ABC} \defeq \left( \circ_{\catC}\right) _{ABC} ^\op  :
\catC^{\mathrm{co}}\left( B, C \right) \times \catC^{\mathrm{co}}\left( A, B \right)\to 	\catC^{\mathrm{co}}\left( A, C \right)
\end{equation} 	
\noindent
With this definition, each pseudofunctor $L = (L, \llll ) : \catC\to \catD $  induces a 
\textit{co-opposite} $$L ^\mathrm{co} = (L ^\mathrm{co}, \llll ^\mathrm{co} ) : \catC^\mathrm{co}\to \catD^\mathrm{co} $$ where
\begin{enumerate}[label=\textbf{(\alph*)}, leftmargin=2em, itemsep=0.3em]
	\item we define \(L ^\mathrm{co}\defeq L : \mathrm{Ob}(\mathcal C^\mathrm{co} ) = \mathrm{Ob}(\mathcal C ) \to  \mathrm{Ob}(\mathcal D  ) = \mathrm{Ob}(\mathcal D ^\mathrm{co} )\);
    \item For each pair of objects \(X,Y\in \mathcal C\), set
\[
L^\mathrm{co}_{X,Y}\defeq (L_{X,Y})^{\op} : \mathcal C^\mathrm{co}(X,Y) \;\to\; \mathcal D ^\mathrm{co}(LX,\,LY).
\]
\item For each composable pair \(X \xrightarrow{g} Y \xrightarrow{h} Z\) in \(\mathcal C^\mathrm{co}\), define the compositor
\[
\llll^\mathrm{co}_{hg} \defeq \big(\llll_{hg}\big)^{-1} : L(h)\circ L(g)\Longrightarrow L(h\circ g)
\]
(viewed in \(\mathcal D^{\mathrm{co}}\)).
\item For each \(X\in \mathcal C\), define the unitor
\[
\llll^\mathrm{co}_X \defeq \big(\llll_X\big)^{-1} : \mathrm{id}_{LX}\Longrightarrow L(\mathrm{id}_X)
\]
(viewed in \(\mathcal D^{\mathrm{co}}\)).

\end{enumerate}
\end{definition} 	
We now introduce our conventions concerning (op)lax natural transformations.  
Classical treatments may be found in \cite{zbMATH0681654, zbMATH03447118}, while a succinct exposition is provided in \cite{leinster1998basicbicategories}.

\begin{definition}[Oplax Natural Transformation]
Let $ L =  (L, \llll ), M = (M , \mmmm ) \colon \catC \to \catD$ be pseudofunctors between 2-categories.  
An \emph{oplax natural transformation}
\[
\theta \colon L \longrightarrow M
\]
consists of the following data:
\begin{enumerate}[label=\textbf{(\alph*)}, leftmargin=2em, itemsep=0.3em]
  \item for each object $X \in \mathcal C$, a $1$-cell 
    \[
      \theta_X \colon L(X) \;\to\; M(X);
    \]
    in $\mathcal D$;
  \item For each $1$-cell $f \colon X \to Y$ in $\mathcal C$, a $2$-cell 
    \[
      \theta_f \colon  \theta_Y \circ L(f)\;\Longrightarrow\; M(f)\circ \theta_X.
    \]
    in $\mathcal D$.
\end{enumerate}
These data are required to satisfy the following three coherence conditions.
\begin{enumerate}
\item \textbf{Identity (Unit Coherence):} For every object $X\in\mathcal C$, Equation \eqref{eq:identity-oplaxnatural} holds.
\item \textbf{Associativity (Composition Coherence):} For any composable pair of $1$-cells 
\[
X \xrightarrow{f} Y \xrightarrow{g} Z,
\]
in $\mathcal C$, Equation \eqref{eq:associativity-oplaxnatural} holds.
\item \textbf{Naturality:} For any $2$-cell $\alpha \colon f \Rightarrow f'$ in $\mathcal C$, Equation \eqref{eq:natural-oplaxnatural} holds.
\end{enumerate}
\begin{equation}\label{eq:identity-oplaxnatural}\tag{Id-OL}
	\begin{tikzcd}[row sep=large, column sep=huge]
		L X \arrow[d, "\theta_X"'] \arrow[rr, "\id _{LX}"] &
		&
		L X \arrow[d, "\theta_X"] \\
		M X \arrow[rr,swap, "\id_{MX}"'] \arrow[rr, bend right=50, "M(\id_{X})"'] &
		\phantom{A}  &
		M X\\
		&\phantom{A}&
		\arrow[Rightarrow, from=1-3, to=2-1, "\id_{\theta_X}" description]
		\arrow[Rightarrow, from=2-2, to=3-2, "\mmmm_X" description]
	\end{tikzcd}
	\;=\;
\begin{tikzcd}[row sep=large, column sep=huge]
		&\phantom{A}& \\ 
	L X \arrow[rr, swap, bend left=60, "\id _{LX}"'] \arrow[d, "\theta_X"'] \arrow[rr,swap, "L(\id_{X})"] & \phantom{A}
	&
	L X \arrow[d, "\theta_X"] \\
	M X \arrow[rr, "M (\id_{X} )"']  &
	  &
	M X
		\arrow[Rightarrow, from=1-2, to=2-2, "\llll_{X}" description]
	\arrow[Rightarrow, from=2-3, to=3-1, "\theta_{\id _ X}" description]
\end{tikzcd}
\end{equation} 

\begin{equation}\label{eq:associativity-oplaxnatural}\tag{A-OL}
\begin{tikzcd}[row sep=large, column sep=huge]
L X \arrow[d, "\theta_X"'] \arrow[r, "L (f)"] &
L Y \arrow[d, "\theta_Y"] \arrow[r, "L (g)"] &
L Z \arrow[d, "\theta_Z"] \\
M X \arrow[r, "M (f)"'] \arrow[rr, bend right=50, "M(g\circ f)"'] &
M Y \arrow[r, "M (g)"'] &
M Z\\
&\phantom{A}&
\arrow[Rightarrow, from=1-2, to=2-1, "\theta_f" description]
\arrow[Rightarrow, from=1-3, to=2-2, "\theta_g" description]
\arrow[Rightarrow, from=2-2, to=3-2, "\mmmm_{gf}" description]
\end{tikzcd}
\;=\;
\begin{tikzcd}[row sep=large, column sep=huge]
&LY \arrow[dr, swap, "L(g)"'] &\\
L X \arrow[d, "\theta_X"'] \arrow[rr, swap, "L(g\circ f)"]\arrow[ur, swap, "L(f)"']  &\phantom{A}&
L Z \arrow[d, "\theta_Z"] \\
M X \arrow[rr, "M(g\circ f)"'] &&
M Z
\arrow[Rightarrow, from=2-3, to=3-1, "\theta_{gf}" description]
\arrow[Rightarrow, from=1-2, to=2-2, "\llll_{gf}" description]
\end{tikzcd}
\end{equation} 

\begin{equation}\label{eq:natural-oplaxnatural}\tag{Nat-OL}
	\begin{tikzcd}[row sep=large, column sep=huge]
		L X \arrow[d, "\theta_X"'] \arrow[rr, "L(f)"] &
		&
		L Y \arrow[d, "\theta_Y"] \\
		M X \arrow[rr,swap, "M(f)"'] \arrow[rr, bend right=50, "M(f')"'] &
		\phantom{A}  &
		M Y\\
		&\phantom{A}&
		\arrow[Rightarrow, from=1-3, to=2-1, "\theta _f" description]
		\arrow[Rightarrow, from=2-2, to=3-2, "M(\alpha)" description]
	\end{tikzcd}
	\;=\;
	\begin{tikzcd}[row sep=large, column sep=huge]
		&\phantom{A}& \\ 
		L X \arrow[rr, swap, bend left=60, "\id _{LX}"'] \arrow[d, "\theta_X"'] \arrow[rr,swap, "L(f')"] & \phantom{A}
		&
		L Y \arrow[d, "\theta_Y"] \\
		M X \arrow[rr, "M (f' )"']  &
		&
		M Y
		\arrow[Rightarrow, from=1-2, to=2-2, "L(\alpha )" description]
		\arrow[Rightarrow, from=2-3, to=3-1, "\theta_{f'}" description]
	\end{tikzcd}
\end{equation} 
\noindent
A \emph{lax natural transformation} 
$\gamma \colon L \longrightarrow M$
consists of an oplax natural transformation $\widehat{\gamma} \colon L^{\mathrm{co}} \longrightarrow M^{\mathrm{co}}$.
A \emph{pseudonatural transformation}
$\beta \colon L \longrightarrow M $
is an oplax natural transformation such that each component $\beta_f$ is invertible.  
Finally, a \emph{$2$-natural transformation} $ \omega \colon L \longrightarrow M $ is a pseudonatural transformation whose components $\omega_f$ are identities for all $1$-cells $f$ in~$\catC$.
\end{definition}

\begin{remark}[Colax natural transformations]
	The terms \emph{oplax} and \emph{colax} are commonly used interchangeably in the literature, as they denote the same notion.  
	In particular, \emph{colax natural transformations} are the same as \emph{oplax natural transformations}.  
	Throughout this work, we treat the two terms as synonymous.
\end{remark}

\begin{definition}[Modification]\label{def:Modification-2-category}
	Let $L, M \colon \catC \to \catD$ be pseudofunctors between $2$-categories, and let 
	$ \gamma , \theta \colon L \Rightarrow M$ be \emph{lax natural transformations}. 
	A \emph{modification}
	\[
	\xi \colon \gamma\Rrightarrow \theta 
	\]
	consists of a family of $2$-cells in $\catD$
	\[
	(\,\xi_X \colon \theta_X \Rightarrow \gamma_X\,)_{X \in \mathrm{Ob}(\catC)}
	\]
	such that, for every $1$-cell $f \colon X \to Y$ in $\catC$, Equation \eqref{eq:natural-oplaxnatural-modification} holds.
\begin{equation}\label{eq:natural-oplaxnatural-modification}\tag{Nat-M}
	\begin{tikzcd}[row sep=large, column sep=huge]
		L X \arrow[d,swap, "L(f)"] \arrow[rr, "\theta_X"] &
		&
		MX \arrow[d, "M(f)"] \\
		LY\arrow[rr,swap, "\theta_Y"'] \arrow[rr, bend right=60, "\gamma _Y"'] &
		\phantom{A}  &
		M Y\\
		&\phantom{A}&
		\arrow[Rightarrow, from=2-1, to=1-3, "\theta _f" description]
		\arrow[Rightarrow, from=3-2, to=2-2, "\xi _Y" description]
	\end{tikzcd}
	\;=\;
	\begin{tikzcd}[row sep=large, column sep=huge]
		&\phantom{A}& \\ 
		L X \arrow[rr, swap, bend left=60, "\theta _X"'] \arrow[d, "L(f)"'] \arrow[rr,swap, "\gamma _X"] & \phantom{A}
		&
		M X \arrow[d, "M(f)"] \\
		L Y \arrow[rr, "\gamma_Y"']  &
		&
		M Y
		\arrow[Rightarrow, from=2-2, to=1-2, "\xi _X" description]
		\arrow[Rightarrow, from=3-1, to=2-3, "\gamma_{f}" description]
	\end{tikzcd}
\end{equation} 	
\end{definition}

\begin{definition}[$2$-categories of pseudofunctors]
Let $\mathcal C$ and $\mathcal D$ be $2$-categories.  
We write
\[
[\mathcal C,\mathcal D]_{\mathrm{oplax}}
\]
for the $2$-category of pseudofunctors and oplax natural transformations, defined as follows:

	\begin{enumerate}[label=\textbf{(\alph*)}, leftmargin=2em, itemsep=0.3em]
		\item \emph{objects:} pseudofunctors $L\colon \mathcal C \to \mathcal D$;
		\item \emph{$1$-cells:} in $[\mathcal C,\mathcal D]_{\mathrm{oplax}}$, the oplax natural transformations 
		\(\theta\colon L\Rightarrow M\);
		\item \emph{$2$-cells:} modifications $\xi\colon \theta \Rrightarrow \theta'$ between (op)lax natural transformations.
	\end{enumerate}
	
	Composition and identities are defined objectwise.  
	For oplax transformations $\theta\colon L\Rightarrow M$ and $\phi\colon M\Rightarrow N$, their composite
	\(\phi\circ\theta\colon L\Rightarrow N\) is given by
\begin{equation} 
	(\phi\circ\theta)_X \defeq \phi_X\circ \theta_X,
	\qquad
	(\phi\circ\theta)_f \defeq
	(\phi_Y \circ \theta_f)\,\cdot\,(\phi_f \circ L(f)),
\end{equation} 
	where $\circ$ denotes horizontal composition and $\cdot$ vertical composition of $2$-cells in~$\mathcal D$.
	
\textit{	The $2$-category $[\mathcal C,\mathcal D]_{\mathrm{lax}}$ of pseudofunctors and lax natural transformations is defined by}
	\[
	[\mathcal C,\mathcal D]_{\mathrm{lax}}
	\defeq
	[\mathcal C^{\mathrm{co}},\mathcal D^{\mathrm{co}}]_{\mathrm{oplax}}.
	\]
	Finally, the $2$-category of pseudofunctors and pseudonatural transformations 	\[
	[\catC,\catD]_{\mathrm{PS}}
	\] is the wide and locally full sub-$2$-category of
$	[\mathcal C^{\mathrm{co}},\mathcal D^{\mathrm{co}}]_{\mathrm{oplax}}$
	whose $1$-cells are the pseudonatural transformations.
\end{definition}

\begin{remark}[Convention and Terminology]
	The terminology distinguishing the directions of the $2$-cells in \emph{lax} and \emph{oplax} natural transformations is not entirely uniform across the literature.  
	Here we adopt the convention that appears to be the most prevalent in two-dimensional universal algebra, namely, the one aligning with the standard treatment of pseudomonads and their pseudoalgebras; see, for example, \cite{zbMATH04105188, zbMATH06284339, 2018arXiv180201767L, zbMATH06970806}.
	
	In this setting, given a pseudomonad, one considers the $2$-categories of pseudoalgebras and their morphisms, pseudo, lax, and oplax.  
	Let $\catC$ and $\catD$ be $2$-categories, and let $\mathsf{disc}(\catC)$ denote the wide discrete sub-$2$-category of $\catC$.  
	If $\catD$ admits sufficient (weighted) bilimits (see, e.g., \cite{MR3491845, zbMATH06881682}), then the restriction functor
	\begin{equation}
		[\catC,\catD]_{\mathrm{PS}}
		\;\longrightarrow\;
		[\mathsf{disc}(\catC),\catD]_{\mathrm{PS}}
	\end{equation}
	is \emph{pseudomonadic}; see, for instance, \cite{zbMATH04105188, zbMATH01787496} or \cite[Sections~7 and~9]{MR3491845} for details.
	
	From this perspective, the $2$-category $[\catC,\catD]_{\mathrm{PS}}$ is (bi)equivalent to the $2$-category of pseudoalgebras and pseudomorphisms, while $[\catC,\catD]_{\mathrm{oplax}}$ (as defined above) corresponds to the $2$-category of pseudoalgebras and oplax morphisms.  
	Dually, $[\catC,\catD]_{\mathrm{lax}}$ corresponds to that of pseudoalgebras and lax morphisms; see again \cite{zbMATH04105188, zbMATH06284339, 2018arXiv180201767L, zbMATH06970806}.
		This convention is therefore fully consistent with the one adopted throughout the present work.
\end{remark}

We now recall the definition of an (op)lax (co)limit.  
There is an extensive and rich literature in two-dimensional limits, including interesting aspects on its variants. For basic aspects of two-dimensional limits, we refer, for instance, to 
\cite{zbMATH03447118, zbMATH05659661, zbMATH06881682, zbMATH07629358,  zbMATH03523837, zbMATH03680046, zbMATH04008629}.  
When $2$-dimensional strict (co)limits exist, (op)lax limits and colimits can be constructed from them;  
for coherence results relating (op)lax (co)limits to other forms of (co)limits, we refer to 
\cite{zbMATH04125649, zbMATH05994242, MR3491845, zbMATH06970806}.  
In our setting, however, it is more convenient to study oplax (co)limits without relying extensively on these results.  
We recall the basic definition below.

\begin{definition}[(Op)lax (Co)limits]
	Let $\catS$ and $\catC$ be $2$-categories.  
	Given pseudofunctors $D \colon \catS \to \catC$ and $\mathcal{W} \colon \catS \to \CAT$, the \emph{oplax limit of $D$ with weight $\mathcal{W}$}, if it exists, is an object
	\[
	\mathrm{oplax}\lim(\mathcal{W}, D)
	\]
	of $\catC$ such that there is an equivalence
	\begin{equation}
		\label{eq:oplax-limit}
		\catC\left(X,\, \mathrm{oplax}\lim(\mathcal{W}, D)\right)
		\;\simeq\;
		[\catC,\CAT ]_{\mathrm{oplax}}\left(\mathcal{W},\, \catC(X, D-)\right)
	\end{equation}
	that is pseudonatural in $X \in \catC$.
	In this case, we say that $\mathrm{oplax}\lim(\mathcal{W}, D)$ is the \emph{oplax $\mathcal{W}$-limit of $D$}.
	
\noindent
	Dually, given a pseudofunctor $\mathfrak{W} \colon \catS^{\op} \to \CAT$, the \emph{oplax $\mathfrak{W}$-colimit of $D \colon \catS \to \catC$}, if it exists, is an object
	\[
	\mathrm{oplax}\colim(\mathfrak{W}, D)
	\]
	of $\catC$ such that there is an equivalence
	\begin{equation}
		\label{eq:oplax-colimit}
		\catC\left(\mathrm{oplax}\colim(\mathfrak{W}, D),\, X\right)
		\;\simeq\;
		[\catC ^\op ,\CAT ]_{\mathrm{oplax}}\left(\mathfrak{W},\, \catC(D-, X)\right)
	\end{equation}
	that is pseudonatural in $X \in \catC$.

\noindent	
Codually, we define the \emph{lax $\mathcal{W}$-limit} and \emph{lax $\mathfrak{W}$-colimit} of a $2$-functor $D$
by the following.
\small
\begin{equation}
\mathrm{lax}\lim(\mathcal{W}, D)\ \defeq\ \mathrm{oplax}\lim\big(\mathcal{W}^{\mathrm{co}},\, D^{\mathrm{co}}\big),
\qquad
\mathrm{lax}\colim(\mathfrak{W}, D)\ \defeq\ \mathrm{oplax}\colim\big(\mathfrak{W}^{\mathrm{co}},\, D^{\mathrm{co}}\big).
\end{equation}

\end{definition}
\normalsize
Analogously, one can define bilimits and strict $2$-dimensional limits. 
For instance, given pseudofunctors
\[
D \colon \catS \to \catC, 
\qquad 
\mathfrak{W} \colon \catS^{\op} \to \CAT, 
\qquad 
\mathcal{W} \colon \catS \to \CAT,
\]
if they exist, the $\mathcal{W}$\emph{-bilimit} $\mathrm{bi}\lim(\mathcal{W}, D)$ 
and the $\mathfrak{W}$\emph{-bicolimit} $\mathrm{bi}\colim(\mathfrak{W}, D)$ 
are objects of $\catC$ such that there are equivalences
\small
\begin{equation}
	\label{eq:bilimit-limit}
	\catC\left(X,\, \mathrm{bi}\lim(\mathcal{W}, D)\right)
	\;\simeq\;
	[\catC, \CAT]_{\mathrm{PS}}\left(\mathcal{W},\, \catC(X, D-)\right),
\end{equation} 
\begin{equation} 	\label{eq:bicolimit-limit}
	\catC\left(\mathrm{bi}\colim(\mathfrak{W}, D),\, X\right)
	\;\simeq\;
	[\catC^{\op}, \CAT]_{\mathrm{PS}}\left(\mathfrak{W},\, \catC(D-, X)\right),
\end{equation}
\normalsize
which are pseudonatural in $X \in \catC$.

\medskip

In what follows, we shall mainly work with \emph{conical} (oplax/lax) (co)limits.  
A weighted (oplax/lax) (co)limit is said to be \emph{conical} when its weight is the constant pseudofunctor at the terminal category in $\CAT$.  
More precisely:

\begin{definition}[Conical (op)lax (co)limits]
	Let $D \colon \catS \to \catC$ be a pseudofunctor.  
	The \emph{conical oplax colimit} $\colim_{\mathrm{oplax}} D$ 
	and the \emph{conical oplax limit} $\lim_{\mathrm{oplax}} D$, if they exist, are respectively defined by
	\begin{equation}
		\mathrm{oplax}\colim D \defeq \mathrm{oplax}\colim(\terminal, D),
		\qquad
		\mathrm{oplax}\lim D \defeq \mathrm{oplax}\lim(\terminal, D),
	\end{equation}
	where, by abuse of notation, $\terminal$ denotes the constant weight at the terminal category $\terminal \in \CAT$.  
	Codually, the \emph{conical lax (co)limits} are defined by
	\begin{equation}
		\mathrm{lax}\colim D \defeq \mathrm{lax}\colim(\terminal, D),
		\qquad
		\mathrm{lax}\lim D \defeq \mathrm{lax}\lim(\terminal, D).
	\end{equation}
\end{definition}

\noindent
Analogously, the \emph{conical bicolimit} $\mathrm{bi}\colim D$ 
and the \emph{conical bilimit} $\mathrm{bi}\lim D$ 
are respectively defined by
\begin{equation}
	\mathrm{bi}\colim D \defeq \mathrm{bi}\colim(\terminal, D),
	\qquad
	\mathrm{bi}\lim D \defeq \mathrm{bi}\lim(\terminal, D).
\end{equation}

If the reader is familiar with the usual definition of (ordinary) conical limits, as presented for instance in \cite{zbMATH03367095},  
the following formulation of conical oplax (co)limits, expressed in terms of oplax cones, may be appreciated.

\begin{lemma}[Conical oplax (co)limits]
	Let $D \colon \catS \to \catC$ be a pseudofunctor.  
	We define
	\[
	\Delta \colon \catC \to [\catS, \catC]_{\mathrm{oplax}}
	\]
	to be the $2$-functor sending each object $X \in \catC$ to the constant $2$-functor at $X$.
	
	The conical oplax colimit of $D$ exists and is equivalent to an object $\mathfrak{C}$ whenever there is an equivalence
	\[
	\catC(\mathfrak{C}, X) \;\simeq\; [\catS, \catC]_{\mathrm{oplax}}(D, \Delta X)
	\]
	pseudonatural in $X$.  
	Dually, the conical oplax limit of $D$ exists and is equivalent to an object $\mathfrak{L}$ whenever there is an equivalence
	\[
	\catC(X, \mathfrak{L}) \;\simeq\; [\catS, \catC]_{\mathrm{oplax}}( \Delta X, D)
	\]
	pseudonatural in $X$.
\end{lemma}

\section{Indexed categories}

The study of indexed categories, fibrations, and Grothendieck constructions has a long and intricate history, marked by several distinct stages of development.  
Foundational ideas were introduced in the seminal works \cite{MR1603475, grothendieck1971revetements} and subsequently systematized in the classical expositions \cite{zbMATH0681654, MR0213413, johnstone2002sketches, zbMATH03512374}, which collectively established the categorical foundations of the subject.  
More recent developments have considerably deepened and broadened these foundations, as illustrated in \cite{zbMATH06810410, zbMATH06881682, nunes2025unravelingiterativechad, nunes2023chad, zbMATH07229468, arXiv:2410.22876, shulman2008framed, zbMATH07005873}.

In this section we recall the principal definitions, review some basic results, and fix the notation and conventions adopted throughout.

\begin{definition}[Indexed Category]
	Let $\catC$ be a category.  
	An \emph{indexed category} consists of a pair $(\catC, \catL)$, where
	\[
	\catL \colon \catC^{\mathrm{op}} \to \CAT
	\]
	is a pseudofunctor, with $\catC$ regarded as a locally discrete $2$-category.  
	Thus an indexed category is determined by the following data:
\begin{enumerate}[label=\textbf{(\alph*)}, leftmargin=2em, itemsep=0.3em]
		\item for each object $A$ of $\catC$, a (possibly large) category $\catL(A)$;
		\item for each morphism $f \colon A \to B$ in $\catC$, a functor $\catL(f) \colon \catL(B) \to \catL(A)$;
		\item natural isomorphisms
		\[
		\eta^A \defeq \catL^A \colon \id_{\catL(A)} \rightarrow \catL(\id_A)
		\quad \text{(unitors)}
		\]
		for each $A \in \catC$, and
		\[
		\mu^{f,g} \defeq \catL^{fg} \colon \catL(f) \circ \catL(g) \rightarrow \catL(g \circ f)
		\quad \text{(compositors)}
		\]
		for all $A \xrightarrow{f} B \xrightarrow{g} C$ in $\catC$;
		\item coherence conditions: for every triple of composable morphisms
		\[
		A\xrightarrow{f} B \xrightarrow{g} C \xrightarrow{h} D
		\]
		in $\catC$, Diag.~\eqref{diag:unitor-indexed-category}   and Diag.~\eqref{diag:compositor-indexed-category}   hold.	
\begin{equation}\label{diag:unitor-indexed-category}  
			\begin{tikzcd}
				&& \catL(f) \arrow[d, equal] \arrow[lld, Rightarrow, "\eta^{A}\ast\id_{\catL(f) }"'] \arrow[rrd, Rightarrow, "\id_{\catL(f) }\ast \eta^{B}"] && \\
				\catL(\id_{A}) \circ \catL(f) \arrow[rr, Rightarrow, "{\mu^{\id_A f}}"'] && \catL(f) && \catL(f) \circ \catL(\id_{B}) \arrow[ll, Rightarrow, "{\mu^{f\id_{B}}}"]
			\end{tikzcd}
\end{equation} 		
\begin{equation}\label{diag:compositor-indexed-category}  
			\begin{tikzcd}
				\catL(f) \circ \catL(g) \circ \catL(h)
				\arrow[rrr, Rightarrow, "{\id_{{}_{\catL(f)}}\ast \mu^{gh}}"]
				\arrow[d, Rightarrow, "{\mu^{fg}\ast \id_{\catL(h)}}"']
				&&&
				\catL(f) \circ \catL(h \circ g) \arrow[d, Rightarrow, "{\mu^{f (h\circ g)}}"]
				\\
				\catL(g \circ f) \circ \catL(h)
				\arrow[rrr, Rightarrow, "{\mu^{(g\circ f)h }}"']
				&&&
				\catL( h\circ g \circ f)
			\end{tikzcd}
\end{equation} 
	\end{enumerate} 

\noindent	
	In this context, we say that $\catL$ is a \emph{$\catC$-indexed category}.  
	If $\catL$ is a $2$-functor (that is, a functor), we say that $\catL$ is a \emph{strict} $\catC$-indexed category.
\end{definition}

We can now define \emph{morphisms of indexed categories}.  
As in the case of morphisms between $2$-categories, several variants may be considered: lax, oplax, pseudo, and strict.  
For our purposes, we regard the \emph{oplax} version as canonical.

\begin{definition}[Morphisms of Indexed Categories]\label{def:Morphisms-Indexed-Categories}
	Let $$(\catC, \catL \colon \catC^{\mathrm{op}} \to \CAT)\quad\mbox{ and }\quad(\catC', \catL' \colon \catC'{}^{\mathrm{op}} \to \CAT)$$ be indexed categories.  
	The following notions of morphisms may be distinguished.
\begin{enumerate}[label=\textbf{(\alph*)}, leftmargin=2em, itemsep=0.3em]
		\item \textit{Lax morphism (lax indexed functor):} a \textit{lax morphism} $$\left( \catC, \catL\right) \to \left( \catC ', \catL'\right)  $$ consists of a pair $(F, \theta)$ where $F \colon \catC \to \catC'$ is a functor and 
		\[
		\theta \colon \catL \longrightarrow \catL' \circ F^{\mathrm{op}}
		\]
		is a lax natural transformation.
		
		\item \textit{Oplax morphism (oplax indexed functor):} an \textit{oplax morphism} $\left( \catC, \catL\right) \to \left( \catC ', \catL'\right)  $ consists of a pair $(F, \gamma)$ where $F \colon \catC \to \catC'$ is a functor and 
		\[
		\gamma \colon \catL \longrightarrow \catL' \circ F^{\mathrm{op}}
		\]
		is an oplax natural transformation.
		
		\item \textit{Pseudomorphism (pseudo-indexed functor):} a \textit{pseudomorphism} $$\left( \catC, \catL\right) \to \left( \catC ', \catL'\right)   $$ consists of an oplax morphism $(F, \gamma): \left( \catC, \catL\right) \to \left( \catC ', \catL'\right) $ in which $\gamma$ is pseudonatural.
		
		\item \textit{Strict morphism (strictly indexed functor):} a \textit{strict morphism} $$\left( \catC, \catL\right) \to \left( \catC ', \catL'\right)   $$ consists of an oplax morphism $(F, \gamma): \left( \catC, \catL\right) \to \left( \catC ', \catL'\right) $ in which $\gamma$ is $2$-natural.

	\end{enumerate} 
\end{definition}

Henceforth, for indexed categories $\catL \colon \catC^{\mathrm{op}} \to \CAT$, we shall, unless stated otherwise, work with \emph{oplax} morphisms.  
In other words, we consider the category whose objects are indexed categories and whose morphisms are oplax.

\begin{definition}[Indexed modification]\label{def:Modification-Indexed-categories}
	Let 
$\left( F, \theta \right),\ \left( G, \gamma \right) \colon 
	\left( \catC, \catL \right) \longrightarrow \left( \catC', \catL' \right)$
	be oplax morphisms of indexed categories.  
	An \emph{indexed modification}
	\[
	\chi \colon \left( F, \theta \right) \Rrightarrow \left( G, \gamma \right)
	\]
	consists of a pair $\chi = \left( \hat{\chi}, \chi \right)$, where:\\
	\begin{minipage}{0.45\textwidth}
		\begin{equation}\label{eq:natural-transformation-modification}
			\hat{\chi} \colon F^{\op} \Rightarrow G^{\op} ,
		\end{equation}
	\end{minipage}
	\hfill
	\begin{minipage}{0.45\textwidth}
		\begin{equation}\label{eq:modification-modification}
			\chi \colon 
			\left( \id_{\catL} \ast \hat{\chi} \right) \cdot \theta 
			\Rrightarrow 
			\gamma ,
		\end{equation}
	\end{minipage}\\ \\
	in which~\eqref{eq:natural-transformation-modification} is a natural transformation
	and~\eqref{eq:modification-modification} is a modification.
\end{definition}

\begin{remark}[$2$-categories]
	Up to size considerations, the foregoing structures assemble into a $2$-category, denoted $\Indexed$, whose objects are indexed categories, whose $1$-cells are oplax morphisms, and whose $2$-cells are modifications. 
	We likewise write $\StrictIndexed$ for the $2$-category of strict indexed categories, strict morphisms, and modifications.
\end{remark}

\begin{remark}[Strict indexed categories]
	In many, though not all, interesting cases, $\catL$ is a strict indexed category.  
	By the classical strictification theorems of Giraud and Bénabou, every indexed category is equivalent to a strict one, in two distinct ways. 
	
	These two ways correspond to the left and right $2$-adjoints to the inclusion $\StrictIndexed\to \Indexed$ of the $2$-category of strict indexed categories into that of all indexed categories.  
	
	This result is a special instance of the general coherence theorem for strict algebras and strict algebras for a $2$-monad. We refer the reader to \cite{zbMATH04105188, zbMATH06154005, zbMATH01839069, MR3491845, zbMATH06970806} for the $2$-(co)monadic approach to coherence.
\end{remark}

Let $(\catC, \catL)$ be an indexed category.  
The \emph{Grothendieck construction}, also referred to as the \emph{$\Sigma$-type} of $\catL$ and denoted by
\[
\Sigma_\catC \catL 
\quad \text{or equivalently} \quad 
\int_\catC \catL,
\]
is the \emph{conical oplax colimit} of $\catL$.  
Dually, the category of sections $\Pi_\catC \catL$ is the \emph{conical oplax limit} of $\catL$.  
Explicitly, we have the following definition.

\begin{definition}[Grothendieck construction]\label{def:Grothendieck-construction}
	Let $\catL \colon \catC^{\mathrm{op}} \to \CAT$ be an $\catC$-indexed category.
	\begin{enumerate}
		\item[(Gr)] The \emph{Grothendieck construction} of $(\catC, \catL)$, denoted
		$\Sigma_\catC \catL$ or $\int_\catC \catL$, is defined by
		\[
		\Sigma_\catC \catL \;\defeq\; \mathrm{oplax}\colim\, \catL.
		\]
		
		\item[(Sect)] The \emph{category of sections} of $(\catC, \catL)$, denoted
		$\Pi_\catC \catL$, is defined by
		\[
		\Pi_\catC \catL \;\defeq\; \mathrm{oplax}\lim\, \catL.
		\]
	\end{enumerate}
\end{definition}

To justify Definition~\ref{def:Grothendieck-construction}, we must first verify that such a conical oplax colimit and limit indeed exist.

\subsection{Grothendieck construction} 
It is well known that the following explicit construction provides the oplax colimit, that is, the Grothendieck construction of an indexed category $\left( \catC , \catL \right)$ as in Definition~\ref{def:Grothendieck-construction}. This is, in fact, the original formulation of the Grothendieck construction, \textit{e.g.} \cite[Def.~1.3.1]{johnstone2002sketches} or \cite{nunes2023chad, zbMATH07229468,   arXiv:2410.22876}.

\begin{proposition}[Explicit Grothendieck construction]]\label{prop:Grothendieck-construction-basic} 
	Let $(\catC, \catL)$ be an indexed category.  
	The Grothendieck construction $\Sigma_\catC \catL$ maybe equivalently described as follows.
	\begin{enumerate}
		\item[(O)] \emph{Objects:} dependent pairs $(A, X)$ consisting of an object $A$ of $\catC$ and an object $X$ of $\catL(A)$, that is,
		\[
		\mathrm{Ob}(\Sigma_\catC \catL) \;\defeq\; \sum_{A \in \mathrm{Ob}(\catC)} \mathrm{Ob}(\catL(A)).
		\]
		
		\item[(M)]  \emph{Morphisms:} for $(A, X)$ and $(B, Y)$, a morphism 
		\[
		(A, X) \longrightarrow (B, Y)
		\]
		is a dependent pair $(f, f')$ where $f \colon A \to B$ in $\catC$ and $f' \colon X \to \catL(f)(Y)$ in $\catL(A)$; equivalently,
		\[
		\Sigma_\catC \catL((A, X), (B, Y)) 
		\;\defeq\; 
		\sum_{f \in \catC(A, B)} \catL(A)\big(X, \catL(f)(Y)\big).
		\]

	\item[(C)]  \emph{Composition:} given morphisms
	\[
	(A, X) \xrightarrow{(f, f')} (B, Y) \xrightarrow{(g, g')} (C, Z),
	\]
	their composite is defined by
	\[
	(g, g') \circ (f, f') 
	\;\defeq\; 
	(g \circ f,\, \catL^{f g}_Z \circ \catL(f)(g') \circ f').
	\]
	\end{enumerate}
\end{proposition}
\noindent
In the setting above, the identities of the Grothendieck construction $\Sigma_\catC \catL$  of an indexed category $(\catC, \catL)$ are given by:
\begin{enumerate} 
\item[(I)]  for each object $(A, X)$ of $\Sigma_\catC \catL$, the morphism
\[
\id_{(A, X)} \;\defeq\; (\id_A,\, \catL^A_X) : (A, X)\to (A, X)
\]
is the identity on $(A, X)$ in $\Sigma_\catC \catL$.
\end{enumerate} 

\begin{remark}[Size of Grothendieck constructions]
	Let $\catL$ be a $\catC$-indexed category.  
	The category $\Sigma_\catC \catL$ is (locally) small, respectively large, whenever $\catC$ and all fibres $\catL(C)$ are (locally) small, respectively large.
\end{remark}

It remains to establish that the \textit{category of sections}, 
as defined in Definition~\ref{def:Grothendieck-construction}, 
also exists. We proceed to do so below.

\subsection{Categories of sections}\label{subsect:categ-sections}

Let $\catL \colon \catC^{\mathrm{op}} \to \CAT$ be a pseudofunctor.  
Its \emph{category of  sections}, also referred to as its \emph{$\Pi$-type}, as defined in Definition~\ref{def:Grothendieck-construction}, is described explicitly below.

\begin{proposition}[Explicit category of sections]
	For an indexed category $\catL \colon \catC^{\mathrm{op}} \to \CAT$, the category of sections
	$\Pi_{\catC}\catL$ may be equivalently described as follows.
	\begin{enumerate}
		\item[(O)] \emph{Objects} are pairs $(X, \xi)$, where $X = (X_C)_{C \in \ob \catC}$ with each $X_C \in \ob \catL(C)$, together with, for every morphism $f \colon A \to B$ in $\catC$, a morphism
		\[
		\xi_f \colon X_A \longrightarrow \catL(f)(X_{B}) \quad \text{in } \catL(A),
		\]
		satisfying, for all $A \xrightarrow{f} B \xrightarrow{g} C$,
		\[
		\xi_{\id_A} = \eta^A_{X_A}
		\qquad\text{and}\qquad
		\xi_{g \circ f} \;=\;
		\mu^{g,f}_{X_{C}} \circ \catL(f)(\xi_{g}) \circ \xi_f,
		\]
		where $\eta$ and $\mu$ are respectively the unitor and compositor of $\catL$.
		\item[(M)] \emph{Morphisms} $\alpha \colon (X, \xi) \to (Y, \zeta)$ are families of morphisms $\alpha_C \colon X_C \to Y_C$ in $\catL(C)$ such that, for every $f \colon A\to B$,
		\[
		\catL(f)(\alpha_{B}) \circ \xi_f \;=\; \zeta_f \circ \alpha_A.
		\]
		\item[(C)] \emph{Composition} is defined componentwise: that is to say, for $$\alpha \colon (X,\xi) \to (Y,\zeta)\quad\mbox{ and }\quad\beta \colon (Y,\zeta) \to (Z,\theta),$$
		\[
		(\beta \circ \alpha)_C \;\defeq\; \beta_C \circ \alpha_C \quad \text{for each } C.
		\]
	\end{enumerate}
\end{proposition}
\noindent
In the setting above, the identities of the category of sections 	$\Pi_{\catC}\catL$   of an indexed category $(\catC, \catL)$ are given by:
\begin{enumerate} 
	\item[(I)]  for each object $(X, \xi )$ of $\Pi_{\catC}\catL$,  the morphism
	\[
	\id_{(X,\xi)} \;\defeq\; (\id_{X_C})_{C \in \ob \catC} : (X, \xi)\to (X, \xi).
	\]
	is the identity on $(X, \xi )$ in  $\Pi_{\catC}\catL$.
\end{enumerate}

It will be particularly useful in Section~\ref{sect:Limits-Colimits-Grothendieck-Constructions} 
to consider the following alternative explicit construction.

\begin{proposition}[Alternative characterization]
	Let $\catL \colon \catC^{\mathrm{op}} \to \CAT$ be an indexed category.  
	The category of sections $\Pi_{\catC}\catL$ may equivalently be described as the full subcategory
	\[
	\Pi_{\catC}\catL
	\;\subseteq\;
	\CAT(\catC, \Sigma_{\catC}\catL)
	\]
	spanned by those functors $F \colon \catC \to \Sigma_{\catC}\catL$ for which
	$\pi_1 \circ F = \id_{\catC}$.
\end{proposition}

\section{Grothendieck fibrations}

Let $(\catC, \catL)$ be an indexed category.  
Consider the canonical projection
\begin{equation}\label{eq:Grothendieck-construction-fibration}
	\pi_1 \colon \Sigma_\catC \catL \longrightarrow \catC,
\end{equation}
which assigns to each dependent pair its first component, that is, $\pi_1(f, f') = f$.  
It is well known that the projection \eqref{eq:Grothendieck-construction-fibration} 
is a Grothendieck fibration (see, for instance, Proposition~\ref{prop:cloven-fibration-from-indexed-category} below and Section~\ref{sec:Grothendieck-construction-revisited}, where the biequivalence induced by the Grothendieck construction is recalled in detail).

It is well known that, given an appropriate choice of cleavage, fibrations and indexed categories coincide as notions in the $2$-categorical sense.  
Accordingly, the fibration of~\eqref{eq:Grothendieck-construction-fibration} embodies precisely the same data as the indexed category from which it arises.  
To recall and substantiate this correspondence, we briefly review the pertinent definitions of fibrations for subsequent use.

\begin{definition}[Cartesian lift]\label{def:cartesian-lift}
	Let $P \colon \catE \to \catB$ be a functor, and $f \colon A \to B$ a morphism in $\catB$.  
	
	A \emph{cartesian lift of $f$ with codomain $E \in \catE$} (where $P(E) = B$) is a morphism  
	$e \colon D \to E$ in $\catE$ such that $P(e) = f$, and such that the induced functor  
	\[
	P_E \colon \catE / E \longrightarrow \catB / B,
	\qquad (D', e') \longmapsto (P(D'), P(e')),
	\]
	has the following property: for every object $(D', e')$ of $\catE / E$, the induced function
	\[
	P_E \colon \catE / E \bigl( (D', e'), (D, e) \bigr)
	\longrightarrow
	\catB / B \bigl( (P(D'), P(e')), (A, f) \bigr)
	\]
	is a bijection.  
	Equivalently, $P$ induces a natural bijection
	\[
	P \colon \catE / E \bigl( -, (D, e) \bigr)
	\longrightarrow
	\catB / B \bigl( -, (A, f) \bigr).
	\]
	In this context, we write $e_{f,E}$ to indicate that $e$ is the chosen cartesian lift of $f$ with codomain $E$.
\end{definition}

Let $P \colon \catE \to \catB$ be a functor, $f \colon A \to B$ a morphism in $\catB$, and $E$ an object of $\catE$ with $P(E) = B$.  
By the Yoneda lemma, any two cartesian lifts of $f$ with codomain $E$ are uniquely isomorphic over $B$.  
Thus, while the choice of a cartesian lift is not canonical, it is determined uniquely up to a unique isomorphism.  
The existence of such lifts, however, is an additional condition on $P$, and is not automatic.

\begin{definition}[Grothendieck fibration]\label{def:Grothendieck-fibration}
	A functor $P \colon \catE \to \catB$ is a Grothendieck fibration if, for each $f \colon A \to B$ in $\catB$ and each $E$ with $P(E) = B$, there is a cartesian lift $e_{f,E}$ of $f$ with codomain $E$.
\end{definition}	

\begin{definition}[Cloven fibration and cleavage]\label{def:Grothendieck-fibration-cloven}
	Given a fibration $P \colon \catE \to \catB$, a choice $e_{f,E}$ of representatives of the isomorphism class of cartesian liftings of $f$ with codomain $E$ is called a \emph{cleavage} of $p$.  
	A fibration equipped with a chosen cleavage is called a \emph{cloven fibration}.
\end{definition}

Under a suitable form of the axiom of choice, every Grothendieck fibration may be endowed with a cleavage.  
For completeness, we recall the classical characterization of cloven fibrations.
Recall that a \textit{rari}, \textit{e.g.}~\cite[Definition~1.2]{zbMATH07844805},  for a functor $F$ 
is a right adjoint $G$ with an adjunction $\left( F\dashv G, \varepsilon , \eta \right) $ such that the counit $\varepsilon $ is the identity.

\begin{proposition}
	A functor \( P \colon \catE \to \catB \) is a cloven fibration if and only if, for each object 
	\( E \in \catE \), the induced functor
	\[
	P_E \colon \catE / E \longrightarrow \catB / P(E)
	\]
	admits a (chosen) rari (and hence a fully faithful right adjoint).
		We denote the right adjoint by
	\[
	\widehat{P_E} \colon \catB / P(E) \longrightarrow \catE / E.
	\]
\end{proposition}
\begin{proof}
	See, for instance, \cite[Theorem~2.10]{MR0213413}.
\end{proof}

\begin{proposition}[Grothendieck construction]\label{prop:cloven-fibration-from-indexed-category}
	Given an indexed category $\left( \catC , \catL \right) $, the functor given by the first projection 
	\begin{equation}
		\pi_1 :  \Sigma_\catC \catL\to \catC, \qquad  (f,f')\mapsto f , 
	\end{equation}
which assigns to each dependent pair its first component,
is a cloven fibration.
\end{proposition}
\begin{proof}
	More details are given in Section \ref{sec:Grothendieck-construction-revisited}. 
\end{proof}

As any two-dimensional notion, there are four dual notions corresponding to fibrations.  
Among these, we focus on the following.

\begin{definition}[Opfibration and bifibration]
	A functor $P \colon \catE \to \catB$ for which  
	$P^{\mathrm{op}} \colon \catE^{\mathrm{op}} \to \catB^{\mathrm{op}}$ is a fibration  
	is called an \emph{opfibration}.  
	A functor that is both a fibration and an opfibration is called a \emph{bifibration}.
\end{definition}

\begin{definition}[Split fibration]
	A cloven fibration $P: \catE \to \catB $ is said to be \emph{split} if its cleavage satisfies $e_{\id[B], E} = \id[E]$ and  
	$e_{g, E} \circ e_{f, E'} = e_{g \circ f, E}$,  
	where $E'$ denotes the domain of $e_{g, E}$.
\end{definition}

\begin{definition}[Morphisms of fibrations]\label{def:morphisms-of-fibrations} 
	Given two fibrations $P \colon \catE \to \catB$ and $P' \colon \catE' \to \catB'$, we distinguish the following notions of morphism:
	\begin{itemize}
		\item \emph{oplax morphisms:} commutative squares of functors
		\[
		\begin{tikzcd}
			\catE \arrow[rr, "F_1"] \arrow[dd, "P"'] && \catE' \arrow[dd, "P'"] \\
			&&\\
			\catB \arrow[rr, "F_0"'] && \catB'
		\end{tikzcd}
		\]
		\item \emph{pseudomorphisms} (sometimes called \emph{fibred functors}): oplax morphisms for which $F_1$ sends cartesian liftings along $P$ of morphisms $f$ in $\catB$ to cartesian liftings along $P'$ of $F_0(r)$.
	\end{itemize}
	If cleavages $e_{-,-}$ and $e'_{-,-}$ are fixed for both $p$ and $p'$, we further define:
	\begin{itemize}
		\item \emph{strict morphisms} (sometimes called \emph{split functors}): pseudomorphisms that respect the chosen cleavages, in the sense that  
		$F_1(e_{f,E}) = e'_{F_0(f), F_1(E)}$.
	\end{itemize}
\end{definition}

\begin{definition}[Fibred natural transformation]
	Given two (oplax) morphisms of fibrations $$(F_1, F_0), (G_1, G_0) \colon (P \colon \catE \to \catB) \to (P' \colon \catE' \to \catB'),$$  
	a \emph{fibred natural transformation} between them is a pair of natural transformations  
	\begin{equation} 
	\left( \alpha_0 \colon F_0 \to G_0 , \alpha_1 \colon F_1 \to G_1 \right)  
	\end{equation} 
	such that Equation \eqref{eq:coherence-fibred-natural-transformation} holds.
\begin{equation}\label{eq:coherence-fibred-natural-transformation}\tag{FibNat}
	\begin{tikzcd}[row sep=large, column sep=huge]
		\catE \arrow[d,swap, "P"] \arrow[rr, "G_1"] &
		&
		\catE ' \arrow[d, "P ' "] \\
		\catB \arrow[rr,swap, "G_0"'] \arrow[rr, bend right=60, "F_0"'] &
		\phantom{A}  &
		\catB '\\
		&\phantom{A}&
		\arrow[Rightarrow, from=3-2, to=2-2, "\alpha _ 0" description]
		\arrow[phantom, Rightarrow, from=2-1, to=1-3, "=" description]
	\end{tikzcd}
	\;=\;
	\begin{tikzcd}[row sep=large, column sep=huge]
		&\phantom{A}& \\ 
		\catE \arrow[rr, swap, bend left=60, "G_1"'] \arrow[d, "P"'] \arrow[rr,swap, "F_1"] & \phantom{A}
		&
		\catE  ' \arrow[d, "P ' "] \\
	\catB \arrow[rr, "F_0"']  &
		&
		\catB ' 
		\arrow[Rightarrow, from=2-2, to=1-2, "\alpha _ 1 " description]
		\arrow[Rightarrow, phantom,  from=3-1, to=2-3, "=" description]
	\end{tikzcd}
\end{equation} 			
	
\end{definition}

This gives rise to the $2$-categories of general, cloven, and split fibrations, 
together with their respective oplax, pseudo, and strict morphisms, 
and the corresponding fibred natural transformations.

\section{Grothendieck construction revisited}\label{sec:Grothendieck-construction-revisited}

Since we shall freely move between cloven fibrations and indexed categories, it is convenient to recall the classical correspondence between the two; namely, we describe below how the Grothendieck construction establishes various biequivalences.

\begin{proposition}[Grothendieck construction]\label{prop:2equivalence}
	The Grothendieck construction induces a \emph{biequivalence} of $2$-categories between
	\begin{enumerate}[label=$\star$]
		\item the $2$-category of indexed categories, together with their oplax morphisms and modifications; and
		\item the $2$-category of cloven fibrations, together with their oplax morphisms and fibred natural transformations.
	\end{enumerate}
	Furthermore, this equivalence restricts in the following ways:
\begin{enumerate}[label=\textbf{(\alph*)}, leftmargin=2em, itemsep=0.3em]
		\item pseudomorphisms of indexed categories correspond to pseudomorphisms of fibrations;
		\item strict morphisms of indexed categories correspond to strict morphisms of cloven fibrations;
		\item strict indexed categories correspond to split fibrations;
\item cloven bifibrations correspond precisely to indexed categories that factor (pseudofunctorially) through the $2$-category $\Cat_{Adj}$ of categories and adjunctions, i.e.\ to indexed categories equipped with chosen left adjoints $\catL_!(f)\dashv\catL(f)$ that satisfy the usual Beck--Chevalley/coherence conditions; see \cite[Lemma~9.1.2]{jacobs1999categorical}.
	\end{enumerate}
\end{proposition}
\begin{proof} 
We construct mutually pseudoinverse $2$-functors between the $2$-category of indexed categories, oplax morphisms, and modifications and the $2$-category of cloven fibrations, oplax morphisms, and fibred natural transformations. 

\smallskip 
\noindent
\emph{From indexed categories to cloven fibrations.}
Let $(\catC,\catL\colon \catC^{op}\to\CAT)$ be a $\catC$-indexed category.
Its Grothendieck construction
\[
\pi_1\colon \Sigma_{\catC}\catL \longrightarrow \catC,\qquad \pi_1(f,f')=f,
\]
is a cloven fibration (Definition~\ref{def:Grothendieck-construction}, Proposition~\ref{prop:Grothendieck-construction-basic}, and  Proposition~\ref{prop:cloven-fibration-from-indexed-category}).

We take the \emph{canonical cleavage} by declaring, for $f\colon A\to B$ in $\catC$ and $(B,Y)\in \Sigma_{\catC}\catL$, \begin{equation}\label{eq:canonical-cleavage} e_{f,(B,Y)}\;\defeq\; (\,f,\ \id_{\catL(f)(Y)}\,)\;:\;(A,\catL(f)(Y))\longrightarrow (B,Y), 
\end{equation} 
and verify that $e_{f,(B,Y)}$ is cartesian in the sense of the universal property in the slice (see Def.~\ref{def:cartesian-lift}): given $g=f\circ h$ and a morphism \( (h,s')\colon (C,Z)\to (A,\catL(f)Y) \) with $s'\colon Z\to \catL(h)(\catL(f)Y)$, the composite
$$ e_{f,(B,Y)}\circ (h,s') = \bigl(f\circ h,\ \catL^{hf}_Y\circ \catL(h)(\id)\circ s'\bigr) = \bigl(g,\ \catL^{hf}_Y\circ s'\bigr). $$
Conversely, given any morphism \((g,t)\colon (C,Z)\to (B,Y)\) with \(g=f\circ h\), there is a unique
\(s'\colon Z\to \catL(h)(\catL(f)Y)\) such that \(t=\catL^{hf}_Y\circ s'\), namely \(s'=(\catL^{hf}_Y)^{-1}\circ t\).
Hence the induced map on slices is bijective, and $e_{f,(B,Y)}$ is cartesian. This proves that $\pi_1$ is a cloven fibration.
\smallskip

Now, let $(F,\gamma)\colon (\catC,\catL)\to (\catC',\catL')$ be an \emph{oplax} morphism of indexed categories in our sense (Definition~\ref{def:Morphisms-Indexed-Categories}).
We write the components of $\gamma $ as \[ \gamma_C\colon \catL(C)\to \catL'(FC) \qquad\text{and}\qquad \gamma_f\colon \gamma_C\circ \catL(f)\Longrightarrow \catL'(Ff)\circ \gamma_{C'} \quad (f\colon C\to C'). \] 
We define a functor \[ \Sigma(F,\gamma)_1\colon \Sigma_{\catC}\catL\longrightarrow \Sigma_{\catC'}\catL' \quad\text{by}\quad \begin{cases} (C,X)\mapsto\ (FC,\ \gamma_C(X)),\\[.3em] \bigl(f,\ u\colon X\to \catL(f)(Y)\bigr)\mapsto \bigl(Ff,\ (\gamma_f)_Y\circ \gamma_C(u)\bigr), \end{cases} \] and set $\Sigma(F,\gamma)_0\defeq F$. Then $\pi_1'\circ \Sigma(F,\gamma)_1 = \Sigma(F,\gamma)_0\circ\pi_1$ by construction, so $(\Sigma(F,\gamma)_1,\Sigma(F,\gamma)_0)$ is an \emph{oplax morphism of fibrations} (commutative square of functors, see Definition \ref{def:morphisms-of-fibrations}). The axioms of a oplax natural transformation $\gamma$ (identity, composition, naturality; see~\eqref{eq:identity-oplaxnatural}, \eqref{eq:associativity-oplaxnatural}, \eqref{eq:natural-oplaxnatural}) translate exactly into functoriality of $\Sigma(F,\gamma)_1$. 

Finally, a(n indexed) modification $\chi\colon (F,\gamma)\Rrightarrow (G,\delta)$ between oplax morphisms of indexed categories (Definition \ref{def:Modification-Indexed-categories}) consists of a natural transformation $\hat{\chi}\colon F^{op}\rightarrow G^{op}$ and a modification 	$\chi \colon 
\left( \id_{\catL} \ast \hat{\chi} \right) \cdot \gamma 
\Rrightarrow 
\delta $. 

It should be noted that the modification
\[
\chi \colon 
\left( \id_{\catL} \ast \hat{\chi} \right) \cdot \gamma 
\Rrightarrow 
\delta
\]
is a family of natural transformations
\[
\left(
\chi_C \colon 
\catL\left( \hat{\chi}_C \right) \cdot \gamma_C 
\longrightarrow 
\delta_C
\right)_{C \in \catC},
\]
satisfying the modification conditions 
(Definition~\ref{def:Modification-2-category}).  
We then denote by
\[
\chi_{C,X} \colon 
\catL\left( \hat{\chi}_C \right)\bigl(\gamma_C(X)\bigr)
\longrightarrow 
\delta_C(X)
\]
the component of the natural transformation $\chi_C$ at $X \in \catL(C)$.

This induces a \emph{fibred} natural transformation \[ \Sigma(\chi)_1\colon \Sigma(F,\gamma)_1\Longrightarrow \Sigma(G,\delta)_1, \qquad \Sigma(\chi)_1{}_{(C,X)}\defeq \bigl(\hat{\chi}_C,\ \chi_{C,X}\bigr), \] over $\Sigma(\chi)_0\defeq \hat{\chi}$, and the fibred naturality square commutes by  Equation~\eqref{eq:natural-oplaxnatural-modification}. Thus we have a $2$-functor \[ \mathbf{Gr}\colon \Indexed\longrightarrow \Fibered^{\mathrm{clov}}, \qquad (\catC,\catL)\longmapsto \bigl(\pi_1\colon \Sigma_{\catC}\catL\to \catC \text{ with the cleavage } \eqref{eq:canonical-cleavage}\bigr). \] 
\smallskip 
\noindent\emph{From cloven fibrations to indexed categories.} Conversely, let $P \colon \catE \to \catB$ be a cloven fibration with a fixed cleavage $e_{f,E}$ (Definitions~\ref{def:Grothendieck-fibration} and \ref{def:Grothendieck-fibration-cloven}). 

We define a pseudofunctor \[ \catL_P\colon \catB^{op}\longrightarrow \CAT \] by 
\[ \catL_P(B)\;\defeq\; \catE_B \quad(\text{the fibre over } B),\qquad \catL_P(f)\;\defeq\; f^{*}\colon \catE_B\to \catE_A, \] where $f^{*}(E)$ is the domain of the chosen cartesian lift $e_{f,E}\colon f^{*}E\to E$ and $f^{*}(u)$ is defined by cartesianness: $e_{f,E'}\circ f^{*}(u)=u\circ e_{f,E}$. The unitors and compositors \[ \eta^B\colon \id_{\catE_B}\Longrightarrow (\id_B)^{*}, \qquad \mu^{fg}\colon f^{*}\circ g^{*}\Longrightarrow (g\circ f)^{*} \] are the unique isomorphisms obtained by comparing the two evident cartesian liftings (by the uniqueness part of the cartesian universal property); they satisfy the pseudofunctor coherence by the same uniqueness. Thus $\catL_P$ is a $\catB$-indexed category. If $(F_1,F_0)$ is an \emph{oplax morphism of fibrations} 
\[ \begin{tikzcd} \catE \ar[r,"F_1"] \ar[d,"P"'] & \catE' \ar[d,"P'"] \\ \catB \ar[r,"F_0"'] & \catB' \end{tikzcd} \qquad (\;P'\circ F_1 = F_0\circ P\;) \] 
between cloven fibrations (no cartesian preservation required), we define \[ (F_0,\gamma)\colon (\catB,\catL_P)\longrightarrow (\catB',\catL_{P'}) \] on objects by $\gamma_B\defeq F_1|_{\catE_B}\colon \catE_B\to \catE'_{F_0(B)}$ and on morphisms by the natural transformations \[ \gamma_f\colon \gamma_A\circ f^{*}\Longrightarrow (F_0 f)^{*}\circ \gamma_B \] whose component at $E\in \catE_B$ is the \emph{unique} morphism in $\catE'$ over $\,\id_{F_0(A)}$ factoring $F_1(e_{f,E})$ through the chosen cartesian lift $e'_{F_0(f),\,F_1(E)}$. \[ e'_{F_0(f),\,F_1(E)}\circ (\gamma_f)_E \;=\; F_1(e_{f,E}) . \] The naturality and oplax axioms for $\gamma$ are immediate from cartesianness and functoriality. If $\xi\colon (F_1,F_0)\rightarrow (G_1,G_0)$ is a \emph{fibred} natural transformation, its components define a modification $\widehat{\xi}\colon (F_0,\gamma)\Rightarrow (G_0,\delta)$ in the evident way. Altogether we obtain a $2$-functor \[ \mathbf{Idx}\colon \Fibered^{\mathrm{clov}}\longrightarrow \Indexed, \qquad (\;P\colon \catE\to \catB\;)\longmapsto (\;\catB,\catL_P\;). \] \smallskip \noindent\emph{The unit and counit of the biequivalence.} First, for any cloven fibration $P\colon \catE\to \catB$, there is a \emph{canonical isomorphism of fibrations} over $\catB$ \[ \varepsilon_P\colon \Sigma_{\catB}\catL_P \xrightarrow{\ \cong\ } \catE \] defined by \[ \varepsilon_P(B,E)\defeq E,\qquad \varepsilon_P\bigl(f,\ u\colon X\to f^{*}Y\bigr)\defeq e_{f,Y}\circ u\;:\; X\to Y. \] Its inverse sends $t\colon X\to Y$ over $f$ to \( (f,\ \widehat{t}\colon X\to f^{*}Y) \), the unique factor through the chosen cartesian lift $e_{f,Y}$ (Definition of cartesian lift and cleavages). It is routine to check that $\varepsilon_P$ is natural in $P$ with respect to oplax morphisms of fibrations, hence a $2$-natural isomorphism \[ \varepsilon\colon \mathbf{Gr}\circ \mathbf{Idx}\Rightarrow \id_{\Fibered^{\mathrm{clov}}}. \] Second, for any indexed category $(\catC,\catL)$, there is a \emph{pseudonatural} equivalence of indexed categories (over $\catC$) \[ \eta_{\catL}\colon \catL \xRightarrow{\ \simeq\ } \catL_{\pi_1} \] whose component at $C\in \catC$ is the equivalence \[ \eta_{\catL,C}\colon \catL(C)\xrightarrow{\ \simeq\ } (\Sigma_{\catC}\catL)_C, \qquad X\longmapsto (C,X),\quad u\colon X\to Y\longmapsto (\id_C,\ \eta^C_Y\circ u), \] with pseudo-inverse $(C,X)\mapsto X$ and the coherence $2$-cells built out of the unitors $\eta^C$ and compositors $\mu^{fg}$ of $\catL$ (Definition of indexed category). Compatibility with reindexing is exactly the definition of $\catL_{\pi_1}$; pseudonaturality follows from the axioms for the unitors/compositors. Thus we obtain a $2$-natural equivalence \[ \eta\colon \id_{\Indexed}\Rightarrow \mathbf{Idx}\circ \mathbf{Gr}. \] Altogether, $(\mathbf{Gr},\mathbf{Idx},\eta,\varepsilon)$ exhibit a biequivalence between the two $2$-categories in the first part of the statement. 
\smallskip 
\noindent
\emph{The listed restrictions.} We now verify the four refinements. 
	
\begin{enumerate} 
\item[(a)] \emph{Pseudomorphisms $\Leftrightarrow$ preservation of cartesian liftings.} If $(F,\gamma)$ is a pseudomorphism of indexed categories, each $\gamma_f$ is invertible. Then $\Sigma(F,\gamma)_1$ sends the canonical cartesian lift $e_{f,(B,Y)}=(f,\id)$ to the arrow \[ \bigl(Ff,\ (\gamma_f)_Y\bigr)\;:\; (FA,\gamma_A X)\longrightarrow (FB,\gamma_B Y), \] which is cartesian over $Ff$ because $(\gamma_f)_Y$ is an isomorphism in the fibre.\footnote{In the Grothendieck construction with the canonical cleavage, a morphism $(f,u)$ is cartesian iff $u$ is an isomorphism; in particular the chosen lift is $(f,\id)$.} Hence $(\Sigma(F,\gamma)_1,\Sigma(F,\gamma)_0)$ preserves cartesian arrows, i.e. it is a \emph{pseudomorphism of fibrations}. Conversely, given a pseudomorphism of fibrations $(F_1,F_0)$, the mate defining $\gamma_f$ above is invertible because $F_1$ preserves cartesian arrows: both $F_1(e_{f,E})$ and the chosen $e'_{F_0(f),\,F_1(E)}$ are cartesian over $F_0(f)$ with the same codomain, so the induced comparison in the fibre is an isomorphism. Hence $(F_0,\gamma)$ is a pseudomorphism of indexed categories. 
		
\item[(b)] \emph{Strict morphisms $\Leftrightarrow$ split (strict) morphisms.} If $(F,\gamma)$ is strict ($2$-natural), then each $\gamma_f$ is an identity; it follows that $\Sigma(F,\gamma)_1$ preserves the chosen cleavage \eqref{eq:canonical-cleavage} \emph{on the nose}, hence is a split (strict) morphism of cloven fibrations. The converse is proved by the same mate argument as in~(a): strict preservation of cartesian morphisms forces each $\gamma_f$ to be an identity, i.e.\ $(F,\gamma)$ is strict.

\item[(c)] \emph{Strict indexed categories $\Leftrightarrow$ split fibrations.} If $\catL$ is strict (a $2$-functor), then the cleavage \eqref{eq:canonical-cleavage} on $\pi_1$ satisfies $e_{\id,(C,X)}=\id$ and $e_{g,(C',Y)}\circ e_{f,(C,X)}=e_{g\circ f,(C',Y)}$ strictly, i.e.\ $\pi_1$ is split. Conversely, if $P$ is split, then the composites of the reindexings $f^*$ assemble strictly to a functor $\catL_P\colon \catB^{op}\to \CAT$. 

\item[(d)] \emph{Bifibrations $\Leftrightarrow$ factorisation through $\Cat_{Adj}$.} By the standard characterisation (proved, e.g., in~\cite[Lemma~9.1.2]{jacobs1999categorical}), a cloven fibration is a bifibration iff each reindexing $f^*$ admits a left adjoint $f_!$. Transporting along the biequivalence, this says exactly that the associated indexed category sends $f$ to a functor with a chosen left adjoint, i.e.\ that it factors through the $2$-category $\Cat_{Adj}$ of categories and adjunctions. 
\end{enumerate} 
This completes the proof of the biequivalence and its listed refinements. 
\end{proof}

\begin{remark}[Lax vs.\ oplax under the Grothendieck correspondence]\label{rem:lax-indexed-vs-fib}
	Proposition~\ref{prop:2equivalence} establishes a biequivalence between
	\emph{oplax} morphisms of indexed categories and \emph{oplax} morphisms of cloven fibrations (i.e.\ strictly commuting squares of functors).  
	By contrast, \emph{lax} morphisms of indexed categories do \emph{not} in general correspond to any obvious notion of morphisms of (cloven) fibrations.
	
	Indeed,
	let $\theta\colon \catL \Rightarrow \catL'\circ F^{\op}$ be a \emph{lax} natural transformation of pseudofunctors $\catC^{\op}\to\CAT$.  
	For each $f\colon C\to C'$ in $\catC$, the laxity constraint has the direction
\[
\theta_f\;:\; 
\catL'(Ff)\circ \theta_{C'} \;\Longrightarrow\; \theta_{C}\circ \catL(f)
\quad\text{(so }(\theta_f)_Y:\ \catL'(Ff)(\theta_{C'}Y)\Rightarrow \theta_C(\catL(f)Y)\text{)}.
\]
	To define a functor $\Sigma_{\catC}\catL \to \Sigma_{\catC'}\catL'$ on morphisms $(f,u)$ in the total category, one needs, functorially in $u\colon X\to \catL(f)Y$, a canonical arrow
	\[
	\theta_C(X) \longrightarrow \catL'(Ff)\bigl(\theta_{C'}(Y)\bigr)
	\]
	in the \emph{opposite} direction to $\theta_f$; without invertibility of $\theta_f$ this is unavailable.  
	Thus, in general a lax indexed morphism does \emph{not} induce a functor between Grothendieck constructions over $F$.
	
	\smallskip
	
	\noindent\emph{Concrete counterexample.}
Take $\catC$ to be the category $0\xrightarrow{\,f\,}1$.  
Let $\catL$ and $\catL'$ have fibres $\catL(0)=\catL(1)=\catL'(0)=\catL'(1)=\Set$, and set the reindexing functors on $f$ to be
\[
\catL(f)=\Delta 1:\Set\to\Set,
\qquad
\catL'(f)=\Delta 2:\Set\to\Set.
\]
Let $F=\id_{\catC}$ and let $\theta$ be the lax transformation with $\theta_0=\theta_1=\id_{\Set}$ and with
\[
\theta_f\colon \catL'(f)\circ\theta_1=\Delta 2\ \Longrightarrow\ \theta_0\circ\catL(f)=\Delta 1
\]
the unique “fold” natural transformation (induced by $2\to 1$), which is \emph{not} invertible.
Consider a morphism $(f, {!}\,)\colon (0,X)\to (1,Y)$ in $\Sigma_{\catC}\catL$ (necessarily $!\colon X\to 1$).  
A total functor determined by $\theta$ would have to assign a morphism
\[
(0,\theta_0 X)=(0,X)\;\longrightarrow\;(1,Y)=(1,\theta_1 Y),
\]
i.e.\ a map
\[
X\longrightarrow \catL'(f)\bigl(\theta_1 Y\bigr)=\Delta 2(Y),
\]
but from the given $u: X\to \Delta 1(Y)$ there is no canonical way to obtain such a map $X\to \Delta 2(Y)$, and the laxity 2\textendash cell $\theta_f$ points in the opposite direction needed to transport $u$.  
Hence no functor $\Sigma_{\catC}\catL\to\Sigma_{\catC}\catL'$ can be induced by this lax $\theta$.
\end{remark}

Henceforth we shall work with indexed categories (equivalently, cloven fibrations) and pseudomorphisms, except where stated otherwise.  
We conclude this section with a note on terminology.

\begin{remark}[Terminology: fibred structure]
	Let $\left( \catC, \catL \right)$ be an indexed category.
	We say that a structure on $\Sigma_\catC \catL$ is \emph{fibred} when the projection
	\begin{equation}\label{eq:gGrothendieck-construction-fibration}
		\pi_1 \colon \Sigma_\catC \catL \longrightarrow \catC
	\end{equation}
	preserves it.
	This usage will be maintained throughout, particularly when the \emph{structure} concerned is that of a monoidal structure, or of limits, colimits, or exponentials.
\end{remark}

\section{Examples of Indexed Categories}

Before turning to structural aspects of the Grothendieck construction, we illustrate the generality of our setting through a series of standard examples.

\begin{example}[Pullback]\label{ex:fibration-pullback}
	Given an indexed category \( \catL \colon \catC^{\mathrm{op}} \to \CAT \) and a (pseudo)functor \( F \colon \catD \to \catC \), composition yields a new indexed category \( \catL \circ F^{\mathrm{op}} \colon \catD^{\mathrm{op}} \to \CAT \).  
	Choosing a cleavage, this corresponds to the familiar fact that for a fibration \( p \colon \catE \to \catC \), the pullback \( F^{*}p \) in \( \CAT \) is again a fibration.
\end{example}

\begin{example}[Composition]\label{ex:fibration-composition}
	The composite of two fibrations is again a fibration.  
	In terms of indexed categories, given \( \catL \colon \catC^{\mathrm{op}} \to \CAT \) and \( \catL' \colon (\Sigma_\catC \catL)^{\mathrm{op}} \to \CAT \),  
	the indexed Grothendieck construction yields
	\[
	(\Sigma_\catL \catL') \colon \catC^{\mathrm{op}} \longrightarrow \CAT
	\]
	with fibre over \( C \) given by \( (\Sigma_\catL \catL')(C) = \Sigma_{\catL(C)} \catL'(C, -) \).
\end{example}

\begin{example}[Dual]
	Postcomposition with \( \mathrm{op} \colon \CAT \to \CAT \) sends any indexed category \( \catL \colon \catC^{\mathrm{op}} \to \CAT \)  
	to its dual \( \catL^{\mathrm{op}} \colon \catC^{\mathrm{op}} \to \CAT \).
\end{example}

\begin{example}[Domain fibration]
	For any category \( \catC \), composition of morphisms defines an indexed category \( \catL \colon \catC^{\mathrm{op}} \to \CAT \) by \( \catL(C) = C / \catC \).  
	The corresponding fibration is the domain functor \( \mathrm{dom} \colon \catC^{\to} \to \catC \).  
	(The dual construction using overcategories yields an opfibration.)
\end{example}

\begin{example}[Codomain fibration]\label{ex:codomain-fibration}
	If \( \catC \) admits pullbacks, the codomain functor \( \mathrm{cod} \colon \catC^{\to} \to \catC \) is a fibration.  
	Using the axiom of choice, one may select a cleavage, obtaining an indexed category whose fibres are the overcategories of \( \catC \).
\end{example}

\begin{definition}[Locally indexed category]
	Following \cite{levy2012call}, a \( \CAT(\catC^{\mathrm{op}}, \Set) \)-enriched category is called a \emph{locally indexed category}.  
	Equivalently, these are indexed categories \( \catL \colon \catC^{\mathrm{op}} \to \CAT \) whose objects are independent of \( C \), and for each \( c \colon C' \to C \), the functor \( \catL(c) \) acts as the identity on objects.
\end{definition}

\begin{example}[\( \catC \)-enriched category]\label{ex:enriched-category}
	Any \( \catC \)-enriched category \( \catD \), for a cartesian monoidal category \( \catC \), is \( \CAT(\catC^{\mathrm{op}}, \Set) \)-enriched via the Yoneda embedding, and hence determines a locally \( \catC \)-indexed category.
\end{example}

\begin{example}[Product self-indexing]\label{ex:locally-indexed}
	Let \( \catC \) be a category with chosen finite products.  
	Define a locally indexed category \( \self(\catC) \colon \catC^{\mathrm{op}} \to \CAT \) by
	\[
	\ob \self(\catC)(C) = \ob \catC, \qquad
	\self(\catC)(C)(C', C'') = \catC(C \times C', C'').
	\]
	For each \( c \colon C' \to C \), the induced functor \( \self(\catC)(C) \to \self(\catC)(C') \) acts as the identity on objects and sends \( f \colon C \times C_1 \to C_2 \) to \( f \circ (c \times \id[C_1]) \).
\end{example}

\begin{example}[Scone]\label{ex:scone}
	Combining Examples~\ref{ex:fibration-pullback} and~\ref{ex:codomain-fibration}, a functor \( F \colon \catD \to \catC \) into a category \( \catC \) with pullbacks induces a \( \catD \)-indexed category \( \catL \) with fibres \( \catL(D) = \catC / F D \).  
	This indexed category is the \emph{Scone} or \emph{Artin gluing} of \( F \).  
	It plays a central rôle in the semantics of programming languages, where logical-relations arguments over a denotational semantics in \( \catD \) are organised as one valued in \( \Sigma_\catD \catL \); see \cite{mitchell1992notes}.
\end{example}

\begin{example}[Lax comma]\label{ex:lax-comma}
	Given a \( 2 \)-category \( \catC \), the \( \CAT \)-enriched Yoneda embedding yields a strict indexed category
	\[
	\twoCAT(-^{\mathrm{op}}, \catC) \colon \Cat^{\mathrm{op}} \longrightarrow \CAT,
	\]
	whose morphisms are oplax natural transformations.  
	The Grothendieck construction $$ \Sigma_\Cat \twoCAT(-^{\mathrm{op}}, \catC)$$ is the \emph{lax comma category} over \( \catC \); see, for instance, \cite{zbMATH07844805, zbMATH07766161, clementino2024lax}.  
	Taking \( \catC = \Cat \) yields the large indexed category  
	\( \twoCAT(-^{\mathrm{op}}, \Cat) \colon \Cat^{\mathrm{op}} \to \CAT \)  
	of small strict indexed categories and oplax natural transformations.
\end{example}

\begin{example}[Families construction]\label{ex:families}
	Precomposing the indexed category of Example~\ref{ex:lax-comma} with the embedding \( \Set \hookrightarrow \Cat \) of sets as discrete categories yields  
	\( \Cat(-, \catC) \colon \Set^{\mathrm{op}} \to \CAT \).  
	Its Grothendieck construction \( \Sigma_\Set \Cat(-, \catC) \) is the familiar category \( \Fam{\catC} \), the free coproduct completion of \( \catC \).
\end{example}

\begin{example}[Subobject fibration]\label{ex:subobject-fibration}
	Assume \( \catC \) has pullbacks.  
	For each \( C \in \catC \), let \( \Sub_{\catC}(C) \) be the poset of subobjects of \( C \) (isomorphism classes of monos into \( C \)), and for \( f \colon C' \to C \), let
	\[
	f^{*} \colon \Sub_{\catC}(C) \longrightarrow \Sub_{\catC}(C')
	\]
	be inverse image along \( f \) (pullback of monos).  
	This defines an indexed category
	\[
	\Sub \colon \catC^{\mathrm{op}} \longrightarrow \CAT, \qquad
	\Sub(C) = \Sub_{\catC}(C), \quad \Sub(f) = f^{*},
	\]
	whose Grothendieck construction
	\( \pi_1 \colon \Sigma_{\catC} \Sub \to \catC \)
	is the \emph{subobject fibration}, a cloven \emph{posetal} fibration (cleaved by chosen pullbacks).  
	If \( \catC \) is \emph{regular} (images exist and are pullback-stable), each \( f^{*} \) admits a left adjoint
	$$
	\Sigma_f \colon \Sub(C') \to \Sub(C)
	$$
	given by (regular) image along \( f \), so that \( \Sub \) is a bifibration; see, e.g.~\cite[§A1.3]{johnstone2002sketches}.
\end{example}

Example~\ref{ex:subobject-fibration} may be generalised by considering the right class \( \catM \) of an orthogonal factorization system \( (\catE, \catM) \) on a category \( \catC \) with pullbacks, in place of monos.  
A related and widely used construction is the \emph{subscone}, which may be viewed as a composite of Examples~\ref{ex:scone} and~\ref{ex:subobject-fibration} (or their generalisation to orthogonal factorization systems); see, for instance, \cite{goubault2002logical}.

\section{Limits and colimits in Grothendieck constructions}
\label{sect:Limits-Colimits-Grothendieck-Constructions} 
We begin our study of the structural properties of Grothendieck constructions with a systematic discussion of limits and colimits.  
While several of the results presented here are part of the folklore, or appear implicitly in the literature, a comprehensive and detailed account seems not to have been written down.  
In particular, we emphasize our introduction of slight generalizations of the notion of \emph{extensive indexed categories} originally formulated in~\cite{nunes2023chad}.

For both technical and practical purposes, we shall make use of the category of sections of indexed categories, including the alternative descriptions presented in Subsection~\ref{subsect:categ-sections}.  
Moreover, we observe the following.

\begin{proposition}
	Given a category $\catE$ and a functor $J \colon \catE \to \Sigma_{\catC}\catL$,  
	there exists a unique decomposition of $J$ as $(J_1, J_2)$, where 
	$J_1 \colon \catE \to \catC$ is a functor and  
	$J_2 \in \Pi_{\catE}(\catL \circ J_1^{\mathrm{op}})$ is a section of $\catL \circ J_1^{\mathrm{op}}$.  
	In particular, there is a canonical isomorphism of categories
	\[
	\CAT(\catE, \Sigma_{\catC}\catL)
	\;\cong\;
	\Sigma_{J_1 \in \CAT(\catE, \catC)}
	\Pi_{\catE}(\catL \circ J_1^{\mathrm{op}}),
	\]
	where the right-hand side is the Grothendieck construction of the indexed category
	\[
	\CAT(\catE, \catC)^{\mathrm{op}} \longrightarrow \CAT,
	\qquad
	J_1 \longmapsto \Pi_{\catE}(\catL \circ J_1^{\mathrm{op}}).
	\]
\end{proposition}

\subsection{Limits in Grothendieck constructions}

We now turn to a general characterization of fibred limits in a Grothendieck construction 
$\pi_1 \colon \Sigma_\catC \catL \to \catC$,  
that is, limits preserved by the projection~$\pi_1$.  
A related formulation appears in~\cite[Theorem~4.2]{MR0213413},  
although the precise statement given here does not seem to appear explicitly in the literature.

\begin{theorem}[Fibred limits in a Grothendieck construction]\label{thm:limits-groth}
	Let $$J = (J_1, J_2) \colon \catE \to \Sigma_\catC \catL$$ be a functor.  
	Then $J$ admits a fibred limit if and only if the following conditions hold:
	\begin{enumerate}
		\item the functor $J_1 \colon \catE \to \catC$ admits a limit $(L, \lambda)$ in $\catC$; and
		\item the induced functor 
		\[
		\catL_{*}(\lambda)(J_2) \colon \catE \longrightarrow \catL(L)
		\]
		admits a limit that is preserved by every $\catL(u) \colon \catL(L) \to \catL(K)$ for $u \colon K \to L$ in $\catC$.
	\end{enumerate}
	In this case, the limit of $J$ is given by
	\[
	\bigl(L,\, \lim_{\catE}\, \catL_{*}(\lambda)(J_2)(E)\bigr).
	\]
\end{theorem}

\begin{proof}
First, given a limit cone $(L,\lambda:\Delta_L\Rightarrow J_1)$ in $\catC$,
reindexing the section $$J_2\in \Pi_{\catE}(\catL\circ J_1^{op})$$ along $\lambda$
yields $\catL_*(\lambda)(J_2)\in \Pi_{\catE}(\catL\circ \Delta_L^{op})$,
which under $$\Pi_{\catE}(\catL\circ \Delta_L^{op})\cong \CAT(\catE,\catL(L))$$
we view as a functor $\catE\to\catL(L)$.

Then we have natural isomorphisms
\begin{align*}
&\CAT(\catE,\Sigma_\catC \catL)\big(\Delta_{(C_1, C_2)},(J_1, J_2)\big)\\
&\cong
\Sigma_{f:\,\Delta_{C_1}\Rightarrow J_1}\ 
\Pi_{\catE}(\catL\circ \Delta_{C_1}^{op})\big(\Delta_{C_2},\,\catL_*(f)(J_2)\big)\\
&\cong
\Sigma_{u\in \catC(C_1, L)}\ 
\Pi_{\catE}(\catL\circ \Delta_{C_1}^{op})\big(\Delta_{C_2},\,\catL_*(\lambda\circ u)(J_2)\big)
\explainr{limit in $\catC$, $f=\lambda\circ u$}\\
&\cong
\Sigma_{u\in \catC(C_1, L)}\ 
\Pi_{\catE}(\catL\circ \Delta_{C_1}^{op})\big(\Delta_{C_2},\,\catL_*(u)( \catL_*(\lambda)(J_2))\big)
\explainr{pseudofunctoriality of $\catL$}\\
&\cong
\Sigma_{u\in \catC(C_1, L)}\ 
\CAT(\catE,\catL(C_1))\big(\Delta_{C_2},\catL(u)\circ \catL_*(\lambda)(J_2)\big)\\
&\cong
\Sigma_{u\in \catC(C_1, L)}\ 
\catL(C_1)\Big(C_2,\ \lim_E \catL(u)\big( \catL_*(\lambda)(J_2)(E)\big)\Big)
\explainr{limit in $\catL(C_1)$}\\
&\cong
\Sigma_{u\in \catC(C_1, L)}\ 
\catL(C_1)\Big(C_2,\ \catL(u)\big(\lim_E  \catL_*(\lambda)(J_2)(E)\big)\Big)
\explainr{$\catL(u)$ preserves the limit}\\
&\cong\ \Sigma_\catC\catL\big((C_1, C_2),(L, \lim_E  \catL_*(\lambda)(J_2)(E))\big).
\end{align*}
Conversely, if $J$ has a fibred limit $(L,X)$ with limiting cone $(\lambda_E,\chi_E):(L,X)\to (J_1(E),J_2(E))$, then since $\pi_1$ preserves fibred limits we get that $(L,\lambda)$ is a limit of $J_1$ in $\catC$, while unpacking the universal property in $\Sigma_\catC\catL$ shows $X\cong \lim_E \catL_*(\lambda)(J_2)(E)$ in $\catL(L)$; moreover, for each $u:C_1\to L$ the cone
\[
\big(\lambda_E\circ u,\; (\mu^{u,\lambda_E})_{J_2(E)}\circ \catL(u)(\chi_E)\big)
\]
exhibits $\catL(u)X$ as the limit of $E\mapsto \catL(\lambda_E\circ u)(J_2(E))$, so each $\catL(u)$ preserves this limit.

\end{proof}

\cite[Theorem 52]{nunes2023chad} shows that a similar construction relates fibred terminal coalgebras of fibred endofunctors on $\Sigma_\catC \catL\to \catC$
to pairs of a terminal coalgebra $L$ in $\catC$ and a terminal coalgebra in $\catL(L)$ that is preserved by change of base.

\subsection{Colimits in Grothendieck Constructions} \label{ssec:colimits-grothendieck}
By duality (and as noted in \cite[Theorem 4.2]{MR0213413}), Theorem \ref{thm:limits-groth} also tells us how to compute fibred colimits in an opfibration.
In particular, it applies to colimits in
bifibrations $\pi_1:\Sigma_\catC\catL\to\catC$ in the sense of a fibration such that all $\catL(f)$ have a left adjoint $\catL_!(f)$.
Indeed, in that case, $(\Sigma_\catC\catL)^{op}\cong \Sigma_{\catC^{op}}\catL_!^{op}\to \catC^{op}$,
which lets us construct fibred colimits in $\Sigma_\catC\catL$ out of colimits in $\catC$ and colimits in $\catL$ that are preserved by change of base in $\catL_!$.
However, there are also other cases of colimits in $\Sigma_\catC\catL$ that we are interested in.
In general, we have the following result.
(We imagine that it is known, but have not found a reference to it in the literature.)

\begin{theorem}[Fibred colimits in a Grothendieck construction] \label{thm:grothendieck-colim}
A functor $$J=(J_1,J_2):\catE\to\Sigma_\catC\catL$$ has a fibred colimit iff
$J_1:\catE\to\catC$ has a colimit $(L,\lambda)$ (i.e.\ a cocone $\lambda:J_1\Rightarrow \Delta_L$) and the functor 
\[
\catL(\lambda):\ \catL(L)\longrightarrow \Pi_\catE(\catL\circ J_1^{op}),\qquad
C_2\longmapsto \big(E\mapsto \catL(\lambda_E)(C_2)\big)
\]
has a left adjoint $\catL_!(\lambda)$. The colimit of $J$ is then given by $(L, \catL_!(\lambda)(J_2))$.
\end{theorem}
\begin{proof}
Write $L\defeq\colim_E J_1(E)$. We have natural isomorphisms
\begin{align*}
&\CAT(\catE,\Sigma_\catC\catL)\big((J_1, J_2),\Delta_{(C_1, C_2)}\big)\\
&\cong\;
\Sigma_{\,f:\,J_1\Rightarrow \Delta_{C_1}}\ 
\Pi_\catE(\catL\circ J_1^{op})\big(J_2,\ E\mapsto \catL(f_E)(C_2)\big)\\
&\cong\;
\Sigma_{\,g\in \catC(L, C_1)}\ 
\Pi_\catE(\catL \circ J_1^{op})\big(J_2,\ E\mapsto \catL(g \circ \lambda_E)(C_2)\big)
\explainr{colimit in $\catC$: $f_E=g\circ\lambda_E$}\\
&\cong\;
\Sigma_{\,g\in \catC(L, C_1)}\ 
\Pi_\catE(\catL \circ J_1^{op})\big(J_2,\ E\mapsto \catL(\lambda_E)\big(\catL(g)(C_2)\big)\big)
\explainr{pseudofunctoriality of $\catL$}\\
&\cong\;
\Sigma_{\,g\in \catC(L, C_1)}\ 
\catL(L)\big(\catL_!(\lambda)(J_2),\, \catL(g)(C_2)\big)
\explainr{$\catL_!(\lambda)\dashv \catL(\lambda)$}\\
&\cong\;
\Sigma_\catC\catL\big((L, \catL_!(\lambda)(J_2)),\,(C_1, C_2)\big).
\end{align*}
Conversely, if $J$ has a fibred colimit $(L,X)$ with cocone $(\lambda_E,\chi_E)$, then $\pi_1$ preserves fibred colimits, so $(L,\lambda)$ is a colimit of $J_1$ in $\catC$. For each $C_2\in\catL(L)$, the universal property of the colimit yields natural bijections
\[
\catL(L)(X,C_2)\;\cong\;\Pi_{\catE}(\catL\circ J_1^{op})\big(J_2,\ E\mapsto \catL(\lambda_E)(C_2)\big),
\]
exhibiting $X$ as the value at $J_2$ of a left adjoint to $\catL(\lambda)$. If such fibred colimits exist for all $J_2$ (naturally in $J_2$), these representatives assemble to a functor $\catL_!(\lambda)$.
\end{proof}

\cite[Theorem 48]{nunes2023chad} shows that a similar construction relates fibred initial algebras of fibred endofunctors on $\Sigma_\catC \catL\to \catC$
to pairs of a initial algebra $L$ in $\catC$ and an initial algebra in $\catL(L)$.
Next, it may be instructive to specialize the colimit construction above to the special cases of coproducts and equalisers.

\noindent {\textbf{Coproducts.}}
A coproduct diagram is a functor $J_1:I\to \catC$ with $I$ discrete, i.e.\ an $I$-indexed family $\{X_i\}_{i\in I}$.
Then Theorem~\ref{thm:grothendieck-colim} expresses coproducts in $\Sigma_\catC\catL$ in terms of the coproduct $\bigsqcup_{i\in I} X_i$ in $\catC$ and a left adjoint to the canonical comparison functor
\[
\catL(\lambda):\ \catL\Big(\bigsqcup_{i\in I} X_i\Big)\ \longrightarrow\ 
\Pi_I(\catL\circ J_1^{op})\ \cong\ \bigsqcap_{i\in I}\catL(X_i).
\]
(Here $\Pi_I$ is the category of sections over the discrete $I$, hence the strict product of categories.)
For $J_2\in\Pi_I(\catL\circ J_1^{op})$ given by a family $(A_i)_{i\in I}$, the coproduct in $\Sigma_\catC\catL$ is
\[
\Big(\ \bigsqcup_{i\in I} X_i,\ \ \catL_!(\lambda)\big((A_i)_{i\in I}\big)\ \Big).
\]

\noindent {\textbf{Coequalizers.}}
Let us now specialise the colimit construction above to coequalizers.
Let $\catE$ be the category with a single parallel pair $r,s:0\rightrightarrows 1$.
A functor $J=(J_1,J_2):\catE\to\Sigma_\catC\catL$ is precisely a parallel pair
$$(f,\alpha),(g,\beta):(X,A)\rightrightarrows (Y,B)$$
with $$f,g:X\rightrightarrows Y\ \text{in }\catC,\;
\alpha:A\to\catL(f)(B),\ \beta:A\to\catL(g)(B).$$
Let $q:Y\to Q$ be the coequalizer of $f$ and $g$ in $\catC$, so $q\circ f=q\circ g$.
Writing $\lambda$ for the colimiting cocone, with components $\lambda_1=q$ and $\lambda_0=q\circ f=q\circ g$, Theorem~\ref{thm:grothendieck-colim} says: if
$\catL(\lambda):\catL(Q)\to\Pi_{\catE}(\catL\circ J_1^{op})$ has a left adjoint
$\catL_!(\lambda)$, then the colimit of $J$ in $\Sigma_\catC\catL$ is
\[
(Q,\ \catL_!(\lambda)(J_2)).
\]
This adjoint acts as follows.
An object $J_2$ of \(\Pi_{\catE}(\catL\circ J_1^{op})\) is a quadruple \((A,B,\alpha,\beta)\) with \(A\in\catL(X)\) and \(B\in\catL(Y)\), \(\alpha:A\to\catL(f)(B)\), \(\beta:A\to\catL(g)(B)\).
The left adjoint \(\catL_!(\lambda)\) sends \((A,B,\alpha,\beta)\) to \(C^\ast\in\catL(Q)\) equipped with maps \(\eta_0:A\to \catL(q\circ f)(C^\ast)\) and \(\eta_1:B\to \catL(q)(C^\ast)\) satisfying
\[
\mu^{f,q}_{C^\ast}\circ \catL(f)(\eta_1)\circ \alpha \;=\; \eta_0,
\qquad
\mu^{g,q}_{C^\ast}\circ \catL(g)(\eta_1)\circ \beta \;=\; \eta_0,
\]
(with compositors $\mu$ of $\catL$) and is universal: for any \(C\in\catL(Q)\), giving \(C^\ast\to C\) in \(\catL(Q)\) is the same as giving maps \(\eta_0':A\to\catL(q\circ f)(C)\), \(\eta_1':B\to\catL(q)(C)\) satisfying these equations.

\begin{corollary}[Coequalizers in a Grothendieck construction from coequalizers in the fibres] \label{cor:coequalizers:extensivity}
{Assume that the reindexings $\catL(\lambda_0),\catL(\lambda_1)$ admit left adjoints
\[
\catL_!(\lambda_0)\dashv \catL(\lambda_0)
\qquad\text{and}\qquad
\catL_!(\lambda_1)=\catL_!(q)\dashv \catL(q).
\]
}
Then, the fibre component $C^*$ of the coequalizer in the Grothendieck construction is precisely a coequalizer in $\catL(Q)$ of two canonical morphisms derived from $\alpha$ and $\beta$.
\end{corollary}
\begin{proof}
We claim that $C^*$ above is computed as the coequalizer in $\catL(Q)$ of the two canonical maps
\[
\widehat\alpha,\widehat\beta:\;\catL_!(\lambda_0)A \rightrightarrows \catL_!(q)B
\]
that we define next.
Let $\eta^q: \id \Rightarrow \catL(q)\circ \catL_!(q)$ 
be the unit of $\catL_!(q)\dashv \catL(q)$.
Define $\widehat\alpha$ to be the mate (along $\catL_!(\lambda_0)\dashv \catL(\lambda_0)$) of the composite
\[
A \xrightarrow{\ \alpha\ } \catL(f)(B)
 \xrightarrow{\ \catL(f)(\eta^q_B)\ } \catL(f)(\catL(q)(\catL_!(q)B))
 \xrightarrow{\ \mu^{f,q}_{\catL_!(q)B}\ } \catL(\lambda_0)(\catL_!(q)B),
\]
and define $\widehat\beta$ analogously using $g$ and $\beta$ as the mate of
\[
A \xrightarrow{\ \beta\ } \catL(g)(B)
 \xrightarrow{\ \catL(g)(\eta^q_B)\ } \catL(g)(\catL(q)(\catL_!(q)B))
 \xrightarrow{\ \mu^{g,q}_{\catL_!(q)B}\ } \catL(\lambda_0)(\catL_!(q)B).
\]

If we then write down the universal property of the coequalizer of $\widehat\alpha$
and $\widehat\beta$, we see that it is precisely the universal property of $C^*$ above.
Indeed, let
\[
\begin{tikzcd}
\catL_!(\lambda_0)A \arrow[r, shift left=0.3em, "\widehat\alpha"] \arrow[r, shift right=0.3em, swap, "\widehat\beta"] & \catL_!(q)B \arrow[r, "e"] & C^\ast
\end{tikzcd}
\]
be a coequalizer in $\catL(Q)$.
Then the cocone maps
\[
\eta_1 \defeq \catL(q)(e)\circ \eta^q_B : B\to \catL(q)(C^\ast),
\qquad
\eta_0 \defeq \catL(\lambda_0)(e)\circ \mu^{f,q}_{\catL_!(q)B}\circ \catL(f)(\eta^q_B)\circ \alpha : A\to \catL(\lambda_0)(C^\ast)
\]
(equivalently, using $g,\beta$) satisfy
\[
\mu^{f,q}_{C^\ast}\circ \catL(f)(\eta_1)\circ \alpha \;=\; \eta_0
\qquad\text{and}\qquad
\mu^{g,q}_{C^\ast}\circ \catL(g)(\eta_1)\circ \beta \;=\; \eta_0,
\]
and are universal with this property: for any $C\in\catL(Q)$, postcomposition with
$e:\catL_!(q)B\to C^\ast$ induces a bijection between morphisms $C^\ast\to C$ in $\catL(Q)$ and pairs
$(\eta_0',\eta_1')$ with  
$$\eta_0':A\to \catL(\lambda_0)(C),\ \eta_1':B\to \catL(q)(C)$$
satisfying
$$\mu^{f,q}_{C}\circ \catL(f)(\eta_1')\circ \alpha=\eta_0'
=\mu^{g,q}_{C}\circ \catL(g)(\eta_1')\circ \beta.$$
\end{proof}

So, in this case, the coequalizer of $(f,\alpha), (g,\beta)$ in $\Sigma_\catC \catL$ is constructed precisely as $(Q, C^*)$ where $Q$ is the coequalizer of $f$ and $g$ in $\catC$ and $C^*$ is the coequalizer of $\widehat{\alpha}$ and $\widehat{\beta}$ in $\catL(Q)$.

\subsection{(Left Kan) $\catE$-Extensivity of Indexed Categories}
\label{ssec:left-kan-extensive}

We now introduce a generalisation of the notion of \emph{extensive indexed category} developed in our earlier work~\cite{nunes2023chad} and in \cite{nunes2025unravelingiterativechad}.
There, we observed that the classical notion of an \emph{extensive category} admits a fibrational reformulation: a category is extensive precisely when its basic fibration is extensive.
This perspective leads naturally to the notion of an extensive indexed category, capturing the fibrational essence of extensivity for dependent structure.

In the present work we take a further step, introducing the concept of a \emph{left Kan $\catE$-extensive indexed category}.
Intuitively, this relaxes the preservation of colimits (of shape~$\catE$) from preservation up to equivalence to requiring the existence of a left adjoint.
This refinement provides the natural categorical setting in which colimits in Grothendieck constructions can be described and computed uniformly, clarifying how colimits in the total category arise coherently from those in the base.

\begin{definition}[$\catE$-extensivity and left Kan $\catE$-extensivity]
	Let $\catC$ admit all colimits of shape~$\catE$.
	An indexed category $\catL \colon \catC^{op}\to\CAT$ is said to be \emph{$\catE$-extensive} if, for every diagram $J_1\colon \catE \to \catC$ with colimit $(L,\lambda)$, the canonical comparison functor
	\[
	\catL(\lambda)\colon \catL(L) \longrightarrow \Pi_{\catE}\big(\catL \circ J_1^{op}\big)
	\]
	is an equivalence.
	If, instead, this comparison functor merely admits a left adjoint, then $\catL$ is called \emph{left Kan $\catE$-extensive}.
	
	Since $\catC$ is a $1$-category, oplax limits in~$\catC$ coincide with strict limits; thus, $\catE$-extensivity amounts to the preservation of oplax limits of shape~$\catE$ up to equivalence, while left Kan extensivity weakens this to preservation up to an adjunction.
\end{definition}

The following corollary, immediate from Theorem~\ref{thm:grothendieck-colim}, provides the main motivation for the definition.

\begin{corollary}[Colimits in (Left Kan) extensive fibrations]
	If $\catC$ has colimits of shape~$\catE$ and $\catL\colon \catC^{op}\to\CAT$ is (left Kan) $\catE$-extensive, then the Grothendieck construction $\Sigma_\catC\catL$ admits fibred colimits of shape~$\catE$.
\end{corollary}

Evidently, an indexed category is (left Kan) $\catE$-extensive for all small~$\catE$ if and only if it is (left Kan) extensive for small coproduct and coequalizer diagrams.

When $\catE$ ranges over (finite) discrete categories $I$, one obtains the notion of a (finite) \emph{coproduct-extensive indexed category}~\cite[§6.5]{nunes2023chad}, \cite[§4.5]{nunes2025unravelingiterativechad}: a pseudofunctor $\catL\colon \catC^{op}\to\CAT$ that preserves (finite) products, that is, for which the canonical comparison
\[
\catL\Big(\bigsqcup_{i\in I} C_i\Big)\ \longrightarrow\ \bigsqcap_{i\in I}\catL(C_i)
\]
is an equivalence.
We refer to this special case (for $\catE = I$ finite) simply as an \emph{extensive indexed category}.
In particular, this recovers~\cite[§6.5]{nunes2023chad}: extensive indexed categories have (finite) coproducts in their Grothendieck construction.\footnote{See also the iteration-extensive notion in~\cite{nunes2025unravelingiterativechad}.}

If $\catL$ is the codomain fibration of~$\catC$ of Example~\ref{ex:codomain-fibration}, it is extensive for (finite) coproducts if and only if~$\catC$ is extensive in the classical sense (see, e.g.~\cite{MR1201048, zbMATH01685986, prezado2024extensive}).
Other examples include the lax comma and families fibrations of Examples~\ref{ex:lax-comma} and~\ref{ex:families}, since representable $2$-functors preserve products.
From the discussion at the beginning of Subsection~\ref{ssec:colimits-grothendieck}, the following observation is immediate.

\begin{corollary}[Left Kan $\catE$-extensivity of bifibrations]
	Let $\catL\colon \catC^{op}\to \CAT$ be an indexed category over a base $\catC$ such that both $\catC$ and each fibre $\catL(C)$ admit colimits of shape~$\catE$.
	Suppose further that each reindexing functor $\catL(f)$ admits a left adjoint $\catL_!(f)$, for all $f\colon C\to C'$ in~$\catC$; equivalently, that the projection $\Sigma_\catC \catL\to \catC$ is a bifibration.
	Then $\catL$ is left Kan $\catE$-extensive, and $\Sigma_\catC\catL\to\catC$ has fibred colimits of shape~$\catE$.
\end{corollary}

\begin{proof}
	Since $(\Sigma_\catC \catL)^{op}\cong \Sigma_{\catC^{op}}(\catL_!^{op})$, the result follows by applying Theorem~\ref{thm:limits-groth} to the indexed category $\catL_!^{op}\colon (\catC^{op})^{op}\to \CAT$, observing that each left adjoint $\catL_!(f)$ preserves colimits.
\end{proof}

In particular, the codomain fibration $\catC^{\to}\to\catC$ of Example~\ref{ex:codomain-fibration} is left Kan $\catE$-extensive whenever~$\catC$ admits colimits of shape~$\catE$.

In general, \emph{left Kan} extensivity for coequalizer diagrams is far more common than extensivity. Indeed:
\begin{lemma}[Coequalizer extensivity forces fibres to be groupoids]
\label{lem:fibres-are-groupoids}
Let $\catL:\catC^{op}\to\CAT$ be an indexed category that is extensive for coequalizer diagrams (i.e.\ $\catE$-extensive for the $\catE=\{0\rightrightarrows 1\}$). Then, for every $X\in\catC$, the fibre $\catL(X)$ is a groupoid.
\end{lemma}

\begin{proof}
We fix $X\in\catC$ and take the parallel pair $f=g=\id_X$, whose coequalizer is $$q=\id_X:X\to X .$$ By extensivity, the comparison
\[
\catL(\lambda):\ \catL(X)\ \xrightarrow{\ \simeq\ }\ \Pi_{\catE}(\catL\circ J_1^{op})
\]
is an equivalence, where objects on the right are quadruples $(A,B,\alpha,\beta)$ with $A,B\in\catL(X)$ and
$\alpha,\beta:A\to \catL(\id_X)(B)\cong B$.
The essential image of $\catL(\lambda)$ consists of objects $(B,B,\theta,\theta)$ where $\theta$ is the canonical isomorphism (built from the unitors/compositors of the pseudofunctor $\catL$).

Given any morphism $u:A\to B$ in $\catL(X)$, we consider $(A,B,u,u)$. Since $\catL(\lambda)$ is essentially surjective, there exist $C$ in $\catL \left( X\right) $, $a:A\xrightarrow{\cong}B'$ and $b:B \xrightarrow{\cong} B'$ such that
$$b\circ u=\theta\circ a. $$ 
Hence $u=b^{-1}\circ\theta\circ a$ is invertible. This proves that every arrow in $\catL(X)$ is invertible, that is to say,  $\catL(X)$ is a groupoid.
\end{proof}
In particular, the families (and lax comma) fibrations tend not to be extensive for coequalizer diagrams.
However, observe that, by Corollary \ref{cor:coequalizers:extensivity}, the families indexed category $\Cat(-,\catD)|_{\Set}$ is left Kan extensive for coequalizer diagrams as long as $\catD$ has enough coequalizers.
There do exist non-trivial examples of extensivity for coequalizers, though, as the following example shows.

\begin{example}[Representables are extensive for all colimits]
Let $C\in \catC$. Then the representable functor $\catC(-, C):\catC^{op}\to \CAT$, where we consider a set as a discrete category, preserves all limits, so it defines an indexed category (with discrete fibres) that is extensive for all colimits, including coequalisers.
\end{example} 

\section{Monoidal structures in Grothendieck constructions}

Theorem~\ref{thm:limits-groth} entails, in particular, that for a cartesian monoidal category~$\catC$, 
there is an equivalence between 
indexed cartesian monoidal structures on $\catL \colon \catC^{\mathrm{op}} \to \CAT$ 
(that is, monoidal structures on the fibres $\catL(C)$ preserved by the reindexing functors $\catL(f)$) 
and fibred cartesian monoidal structures on $\Sigma_\catC \catL$ 
(in the sense of fibred functors, or pseudomorphisms of fibrations).  

This correspondence is realized by taking
\[
\terminal_{\Sigma_\catC\catL}
= (\terminal_\catC,\, \terminal_{\catL(\terminal_\catC)}),
\qquad
(C, L) \times_{\Sigma_\catC\catL} (C', L')
= \bigl(C \times_\catC C',\, \catL(\pi_1)(L) \times_{\catL(C \times_\catC C')} \catL(\pi_2)(L')\bigr),
\]
where $\pi_1$ and $\pi_2$ denote the product projections
$C \xleftarrow{\;\pi_1\;} C \times C' \xrightarrow{\;\pi_2\;} C'$ in~$\catC$.  

As shown by Shulman~\cite{shulman2008framed}, this correspondence extends beyond the cartesian case, encompassing
monoidal, braided monoidal, and symmetric monoidal structures on $\Sigma_\catC \catL$, 
provided that the monoidal structure on~$\catC$ itself is cartesian.

\begin{theorem}[Monoidal structures on a Grothendieck construction {\cite[Theorem~12.7]{shulman2008framed}}]\label{thm:fibred-monoidal}
	Assume that $\catC$ is a cartesian monoidal category.  
	Then the following definitions of the monoidal unit~$I$ and tensor product~$\otimes$
	determine an equivalence between fibred monoidal structures on $\Sigma_\catC \catL$
	and indexed monoidal structures on~$\catL$:
	\small 
	\[
	\begin{array}{lll}
		I_{\Sigma_\catC\catL} = (\terminal_\catC,\, I_{\catL(\terminal_\catC)}) &
		\text{and} &
		(C, L) \otimes_{\Sigma_\catC\catL} (C', L')
		= \bigl(C \times_{\catC} C',\, \catL(\pi_1)(L) \otimes_{\catL(C \times_{\catC} C')} \catL(\pi_2)(L')\bigr), \\[1em]
		I_{\catL(C)} = \catL(!_C)(\pi_2(I_{\Sigma_\catC \catL})) &
		\text{and} &
		L \otimes_{\catL(C)} L'
		= \catL(\langle \id[C], \id[C] \rangle)
		\bigl(\pi_2((C, L) \otimes_{\Sigma_\catC \catL} (C, L'))\bigr).
	\end{array}
	\] \normalsize
	Moreover, fibred braidings for $\otimes_{\Sigma_\catC \catL}$ correspond bijectively 
	to indexed braidings for $\otimes_\catL$, 
	and a braiding for $\otimes_{\Sigma_\catC \catL}$ is symmetric if and only if 
	the corresponding braiding for $\otimes_\catL$ is symmetric.
\end{theorem}

A detailed analysis of monoidal structures on Grothendieck constructions lies beyond the scope of this work.  
Our interest here concerns the circumstances under which the above monoidal structure is \emph{closed}.  
For the general results underlying Theorem~\ref{thm:fibred-monoidal}, we refer the reader to~\cite{shulman2008framed},  
and for a comprehensive treatment of monoidal Grothendieck constructions, to~\cite{zbMATH07229468}.

\subsection{Monoidal closed structures on Grothendieck constructions}

A natural question is how monoidal closure of the projection 
$\pi_1 \colon \Sigma_\catC \catL \to \catC$ 
relates to the monoidal closure of the fibre categories of~$\catL$.  
The remainder of this paper is devoted primarily to this question.

The closest result in this direction that we are aware of is the following, due to Shulman.

\begin{lemma}[{\cite[Proposition~13.25]{shulman2008framed}}]\label{lem:shulman-monoidal-closure}
	Suppose that $\catL$ is an indexed monoidal category over a cartesian monoidal category~$\catC$,  
	and that for every morphism $f \colon C \to C'$ in~$\catC$,  
	the reindexing functor $\catL(f) \colon \catL(C') \to \catL(C)$ admits a right adjoint 
	$\catL_*(f)$ satisfying the right Beck--Chevalley condition.  
	Then the following are equivalent:
	\begin{itemize}
		\item the functors 
		$(-) \otimes_{\Sigma_\catC\catL} (C_2, L_2) \colon \catL(C_1) \to \catL(C_1 \times C_2)$ 
		have right adjoints if and only if the functors 
		$(-) \otimes_{\catL(C)} L \colon \catL(C) \to \catL(C)$ 
		have right adjoints;
		\item the functors 
		$(C_1, L_1) \otimes_{\Sigma_\catC\catL} (-) \colon \catL(C_2) \to \catL(C_1 \times C_2)$ 
		have right adjoints if and only if the functors 
		$L \otimes_{\catL(C)} (-) \colon \catL(C) \to \catL(C)$ 
		have right adjoints.
	\end{itemize}
\end{lemma}

As observed in~\cite[Remark~13.12]{shulman2008framed},  
Lemma~\ref{lem:shulman-monoidal-closure} does not, in itself, imply that $\Sigma_\catC \catL$ is monoidal closed.  
Indeed, many naturally arising monoidal closed structures on total categories $\Sigma_\catC \catL$ of fibrations 
are not fibred, and the sufficient conditions ensuring their existence are typically subtle.  
In the present work, we shall introduce one such sufficient condition on the monoidal structure, 
which we call \emph{$(\Sigma, \catC)$-cotractability}.  

This condition is implied, in particular, by indexed monoidal closure of~$\catL$,  
but it is strictly more general.  
Before turning to it, we record the following result, which provides a fibred analogue of Shulman’s theorem.

\begin{theorem}[Fibred $\Longleftrightarrow$ indexed monoidal closure]\label{thm:fibred-monoidal-closed}
	Let $\catC$ be a cartesian monoidal category, and let $\catL$ be an indexed monoidal category over~$\catC$.  
	By Theorem~\ref{thm:fibred-monoidal}, the projection 
	$\Sigma_\catC \catL \to \catC$ 
	is canonically a fibred monoidal category.  
	Then the following conditions are equivalent:
	\begin{enumerate}
		\item[\textnormal{(i)}] \textit{(Indexed side)}:  
		$\catC$ is cartesian closed,  
		$\catL$ is \emph{indexed monoidal left-closed} (resp.\ right-closed),  
		meaning that each fibre $\catL(C)$ is left-closed (resp.\ right-closed)  
		and that every reindexing functor preserves the closure;  
		furthermore, for every projection $\pi_2 \colon C \times C' \to C'$,  
		the reindexing functor 
		$\catL(\pi_2) \colon \catL(C') \to \catL(C \times C')$  
		admits a right adjoint $\catL_*(\pi_2)$ satisfying the right Beck--Chevalley condition 
		for pullback squares along projections~$\pi_2$.
		\item[\textnormal{(ii)}] \textit{(Fibred side)}:  
		The Grothendieck construction 
		$\Sigma_\catC \catL \to \catC$ 
		is a fibred monoidal left-closed (resp.\ right-closed) fibration.
	\end{enumerate}
	
	Moreover, when these equivalent conditions hold,  
	the closures determine one another via the explicit formulas
	\begin{equation}\tag{$\dagger$}\label{eq:Sigma-internal-hom}
		(C,L)\multimap_{\Sigma_{\catC}\catL} (C',L')
		\;=\;
		\Bigl(C\Rightarrow_{\catC} C',\;
		\catL_*(\pi_2)\bigl(\,\catL(\pi_1)(L)\multimap_{\catL(C\times(C\Rightarrow_{\catC} C'))}\catL(\ev)(L')\,\bigr)\Bigr),
	\end{equation}
	and
	\begin{equation}\tag{$\ddagger$}\label{eq:fibre-internal-hom}
		L\multimap_{\catL(C)} L'
		\;=\;
		\catL\bigl(\Lambda(\pi_2)\bigr)\Bigl(\,
		\pi_2\bigl((C,L)\multimap_{\Sigma_{\catC}\catL}(C,L')\bigr)\Bigr),
	\end{equation}
	where $\pi_1,\pi_2$ are the product projections,  
	$\ev \colon C\times (C\Rightarrow_{\catC}C')\to C'$ is the evaluation map,  
and $\Lambda(\pi_2) \colon C \to C\Rightarrow_{\catC}C$ is the transpose of 
$\pi_2 \colon C\times C \to C$ under the exponential adjunction in~$\catC$.  
In both~\eqref{eq:Pi-along-projection} and~\eqref{eq:fibre-internal-hom},  
$\pi_2(-)$ denotes the second (fibre) component of an object of $\Sigma_{\catC}\catL$;  
in~\eqref{eq:Pi-along-projection} we reindex along $\Lambda(\id_{C\times C'}) \colon C'\to C\Rightarrow_{\catC}(C\times C')$,  
and in~\eqref{eq:fibre-internal-hom} along $\Lambda(\pi_2) \colon C\to C\Rightarrow_{\catC}C$.

	In particular, when \textnormal{(ii)} holds, the right adjoints along projections are determined uniquely by
\begin{equation}\tag{$\star$}\label{eq:Pi-along-projection}
\catL_*(\pi_2)(M)
\;=\;
\catL\big(\Lambda(\id_{C\times C'})\big)
\left(
\pi_2\Bigl((C,I)\multimap_{\Sigma_{\catC}\catL}(C\times C',M)\Bigr)
\right),
\end{equation}
	and they satisfy the right Beck--Chevalley condition.  
	The right-closed case is entirely analogous.
\end{theorem}

\begin{proof}
Suppose that $\catC$ is cartesian closed and that $\catL$ is indexed monoidal left-closed (so each $\multimap$ is preserved by base change) and that $\catL(\pi_2)\dashv\catL_*(\pi_2)$ satisfies right Beck--Chevalley for projections.
Then we have the following natural isomorphisms\\
\resizebox{\linewidth}{!}{
\parbox{\linewidth}{
\begin{align*}
&\Sigma_\catC \catL\big((C_1, L_1) \otimes (C_2, L_2), (C_3, L_3)\big) =\\
&\Sigma_{f\in \catC(C_1\times C_2, C_3)} \catL(C_1\times C_2)\big(\catL(\pi_1)(L_1)\otimes \catL(\pi_2)(L_2), \catL(f)(L_3)\big)\cong\explainr{left-closure in $\catL(C_1\times C_2)$}\\
&\Sigma_{f} \catL(C_1\times C_2)\big(\catL(\pi_2)(L_2), \,\catL(\pi_1)(L_1)\multimap_{\catL(C_1\times C_2)} \catL(f)(L_3)\big)\cong\explainr{$\catL(\pi_2)\dashv\catL_*(\pi_2)$}\\
&\Sigma_{f} \catL(C_2)\big(L_2, \,\catL_*(\pi_2)\big(\catL(\pi_1)(L_1)\multimap_{\catL(C_1\times C_2)} \catL(f)(L_3)\big)\big)\cong\explainr{exponential adjunction in $\catC$}\\
&\Sigma_{g \in \catC(C_2, C_1\Rightarrow_{\catC} C_3)} \catL(C_2)\big(L_2, \,\catL_*(\pi_2)\big(\catL(\pi_1)(L_1)\multimap \catL(\ev \circ (C_1\times g))(L_3)\big)\big)\cong\explainr{$\catL$ pseudofunctor}\\
&\Sigma_{g} \catL(C_2)\big(L_2, \,\catL_*(\pi_2)\big(\catL(C_1\times g)(\catL(\pi_1)(L_1)\multimap \catL(\ev)(L_3))\big)\big)\cong\\
&\Sigma_{g} \catL(C_2)\big(L_2, \,\catL_*(\pi_2)\big(\catL(\pi_1)(L_1)\multimap \catL(C_1\times g)(\catL(\ev)(L_3))\big)\big)\cong\explainr{reindexing preserves $\multimap$ and $\pi_1\circ (C_1\times g)=\pi_1$}\\
&\Sigma_{g} \catL(C_2)\big(L_2, \,\catL(g)\big(\catL_*(\pi_2)(\catL(\pi_1)(L_1)\multimap \catL(\ev)(L_3))\big)\big)\cong\explainr{Beck--Chevalley for $\pi_2$ vs.\ $C_1\times g$}\\
&\Sigma_\catC \catL\big((C_2,L_2),(C_1\Rightarrow_{\catC} C_3,\catL_*(\pi_2)(\catL(\pi_1)(L_1)\multimap \catL(\ev)(L_3)))\big)\\
&=\,\Sigma_\catC \catL\big((C_2,L_2),(C_1,L_1)\multimap_{\Sigma_\catC\catL}(C_3,L_3)\big).
\end{align*}}}
Thus \eqref{eq:Sigma-internal-hom} defines a left-exponential in $\Sigma_\catC\catL$, which is clearly fibred by construction.
This shows \textnormal{(i)}$\Rightarrow$\textnormal{(ii)}.

\medskip

Conversely, if $\Sigma_{\catC}\catL\to\catC$ is fibred monoidal left-closed, then for every projection
$\pi_2:C\times C'\to C'$ the change of base
$\catL(\pi_2):\catL(C')\to\catL(C\times C')$ has a right adjoint given by the (corrected) formula
\[
\catL_*(\pi_2)(M)\;:=\;
\catL\big(\Lambda(\id_{C\times C'})\big)
\left(
\pi_2\Big((C,I)\multimap_{\Sigma_{\catC}\catL}(C\times C',M)\Big)\right)\in \catL(C'),
\]
and this adjunction satisfies right Beck--Chevalley for pullback squares along projections. Indeed, using $(C,I)\otimes(C',X)\cong (C\times C',\,\catL(\pi_2)(X))$ and monoidal closure of $\Sigma_\catC \catL$ yields natural isomorphisms:
\small 
\begin{align*}
&\catL(C\times C')\big(\catL(\pi_2)X,M\big)
\cong\explainr{by \eqref{eq:fixed-base} with $c'=\id_{C\times C'}$}\\
&\Big\{(h,m)\in (\Sigma_{\catC}\catL)\big((C\times C',\catL(\pi_2)X),(C\times C',M)\big)\mid h=\id\Big\}
\cong\explainr{$(C,I)\otimes(C',X)\cong(C\times C',\catL(\pi_2)X)$}\\
&\Big\{(k,n)\in (\Sigma_{\catC}\catL)\big((C,I)\otimes(C',X),(C\times C',M)\big)\mid k=\id\Big\}
\cong\explainr{closure in $\Sigma_{\catC}\catL$}\\
&\Big\{(g,p)\in (\Sigma_{\catC}\catL)\big((C',X),(C,I)\multimap_{\Sigma_{\catC}\catL}(C\times C',M)\big)
\mid g=\Lambda(\id_{C\times C'})\Big\}
\cong\explainr{apply \eqref{eq:fixed-base} with $c'=\Lambda(\id_{C\times C'})$}\\
&\catL(C')\Big(X,\ \catL(\Lambda(\id_{C\times C'}))\big(\pi_2((C,I)\multimap_{\Sigma_{\catC}\catL}(C\times C',M))\big)\Big)
\cong\explainr{definition of $\catL_*(\pi_2)$ via \eqref{eq:Pi-along-projection}}\\
&\catL(C')\big(X,\ \catL_*(\pi_2)(M)\big).
\end{align*}
\normalsize \noindent exhibiting $\catL(\pi_2)\dashv \catL_*(\pi_2)$.

\smallskip

\emph{Right Beck--Chevalley for these adjoints.}
For any $g:C'\to C''$ consider the pullback square
\[
\begin{tikzcd}
C\times C' \arrow[r,"C\times g"] \arrow[d,"\pi_2"'] & C\times C'' \arrow[d,"\pi_2"]\\
C' \arrow[r,"g"] & C''
\end{tikzcd}
\]
The pseudofunctoriality of $\catL$ gives an invertible $2$-cell
$\catL(C\times g)\circ \catL(\pi_2)\;\cong\; \catL(\pi_2)\circ \catL(g)$.
Taking mates under the adjunction $\catL(\pi_2)\dashv\catL_*(\pi_2)$ (whose unit/counit are defined via the fibred internal hom in $\Sigma_\catC\catL$ as above) yields the Beck--Chevalley isomorphism
\[
\catL(g)\,\catL_*(\pi_2)\ \cong\ \catL_*(\pi_2)\,\catL(C\times g),
\]
natural in $M\in\catL(C\times C'')$.

\smallskip

Let
\[
\bigl(C\Rightarrow_{\catC}C',\,E\bigr)\;:=\;(C,I)\multimap_{\Sigma_{\catC}\catL}(C',I).
\]
Define $\ev:C\times (C\Rightarrow_{\catC}C')\to C'$ to be the \emph{base component} of the counit
\[
(C,I)\otimes (C\Rightarrow_{\catC}C',E)\;\longrightarrow\; (C',I)
\]
in $\Sigma_{\catC}\catL$. The transpose bijection
$\catC(X\times C,C')\cong \catC\bigl(X,\,C\Rightarrow_{\catC}C'\bigr)$
is induced by the internal hom adjunction in $\Sigma_\catC \catL$; hence $\catC$ is cartesian closed.

For any $c':C\to C'$ we have
\begin{equation}
\tag{$*$}\label{eq:fixed-base}
\catL(C)(L,\catL(c')(L')) \;\cong\; \{(c, m) \in (\Sigma_\catC \catL)((C, L), (C', L')) \mid c = c'\}.
\end{equation}
Define, for $L,L'\in\catL(C)$,
\[
L\multimap_{\catL(C)} L'
\;:=\;
\catL\big(\Lambda(\pi_2)\big)\Big(\,\pi_2\big((C, L)\multimap_{\Sigma_\catC\catL}(C, L')\big)\Big),
\]
where $\Lambda(\pi_2):C\to C\Rightarrow_{\catC} C$ is the transpose of $\pi_2:C\times C\to C$.
Then, naturally in $X,L,L'$,
\small 
\begin{align*}
&\catL(C)\big(X,\,L\multimap L'\big)
\cong\\
&\big\{(g,m)\in (\Sigma_\catC\catL)\big((C,X),(C\Rightarrow_{\catC} C,\pi_2((C,L)\multimap(C,L')))\big)\,\big|\,g=\Lambda(\pi_2)\big\}\cong\hspace{-6pt}\\
&\big\{(h,n)\in (\Sigma_\catC\catL)\big((C,X)\otimes(C,L),(C,L')\big)\,\big|\,h=\ev\circ(\id_C\times \Lambda(\pi_2))\big\}=\explainr{closure in $\Sigma_\catC \catL$}\\
& \big\{(h,n)\in (\Sigma_\catC\catL)\big((C,X)\otimes(C,L),(C,L')\big)\,\big|\,h=\pi_2\big\}\cong\explainr{$\ev\circ(\id\times \Lambda(\pi_2))=\pi_2$}\\
& \catL(C\times C)\big(\catL(\pi_1)X\otimes \catL(\pi_2)L,\;\catL(\pi_2)L'\big)\cong\explainr{by \eqref{eq:fixed-base} with $c'=\pi_2$}\\
& \catL(C)\big(\catL_*(\pi_2)(\catL(\pi_1)X\otimes \catL(\pi_2)L),\,L'\big)\cong\explainr{$\catL(\pi_2)\dashv\catL_*(\pi_2)$}\\
& \catL(C)\big(X\otimes L,\,L'\big).\explainr{right Beck--Chevalley}
\end{align*} \normalsize 
So the formula \eqref{eq:fibre-internal-hom} equips each fibre $\catL(C)$ with a left internal hom; its indexedness follows from the fibred closure and Beck--Chevalley. This shows \textnormal{(ii)}$\Rightarrow$\textnormal{(i)}, with $\catL_*(\pi_2)$ given by \eqref{eq:Pi-along-projection}.

The constructions \eqref{eq:Sigma-internal-hom} and \eqref{eq:fibre-internal-hom} are mutually inverse, yielding the desired equivalence. The right-closed case is analogous.
\end{proof}

We now turn to the more general situation of closed structures on $\Sigma_\catC \catL$ which are not necessarily fibred.

\section{$\Sigma$-(Co)tractable Monoidal Structures}\label{sec:sigma-tractable-monoidal}

We now seek more general sufficient conditions ensuring that a fibred monoidal structure
on $\Sigma_\catC \catL$ is closed.  
To this end, we focus on a particularly well-behaved class of monoidal structures on~$\catL$,  
which we call \emph{$\Sigma$-tractable}.  
Intuitively, these are monoidal structures that admit a canonical decomposition of morphisms 
$A \to B \otimes C$ into a component that depends only on~$A$ and~$B$, 
together with a residual morphism into~$C$.

\begin{definition}[$\Sigma,\catC$-tractable monoidal structure]
	Let $\catC$ be a category with a terminal object.  
	A monoidal category $(\catL, I, \otimes, a, l, r)$ is said to be 
	\emph{$\Sigma,\catC$-tractable} if it is equipped with:
	\begin{itemize}
		\item a functor $(-) \multimap T(-) \colon \catL^{\mathrm{op}} \times \catL \to \catC$;
		\item a functor 
		\[
		\abstrcompldom \colon \terminal \downarrow \bigl((-) \multimap T(-)\bigr) \longrightarrow \catL,
		\]
		from the comma category of $\terminal \colon \terminal \to \catC$ and 
		$(-) \multimap T(-) \colon \catL^{\mathrm{op}} \times \catL \to \catC$;
		\item a natural isomorphism between functors 
		$\catL^{\mathrm{op}} \times \catL \times \catL \to \Set$:
		\[
		\catL(A, B \otimes C)
		\;\cong\;
		\Sigma_{f \in \catC(\terminal, A \multimap T B)}
		\catL\bigl(\abstrcompldom(A, B, f),\, C\bigr).
		\]
	\end{itemize}
	Dually, $(\catL, I, \otimes, a, l, r)$ is said to be 
	\emph{$\Sigma,\catC$-cotractable} if 
	$(\catL^{\mathrm{op}}, I, \otimes^{\mathrm{op}}, a^{-1}, l^{-1}, r^{-1})$ 
	is $\Sigma,\catC$-tractable.
\end{definition}

In most cases, $\Sigma,\catC$-tractable monoidal categories arise from a more intrinsic notion,  
which we call simply \emph{$\Sigma$-tractable}.  

\begin{definition}[$\Sigma$-tractable monoidal structure]
	A monoidal category $(\catL, I, \otimes, a, l, r)$ is said to be 
	\emph{$\Sigma$-tractable} if it is equipped with:
	\begin{itemize}
		\item a functor $T \colon \catL \to \catL$;
		\item a functor 
		\[
		\abstrcompldom \colon \catL \downarrow T \longrightarrow \catL,
		\]
		from the comma category $\catL \downarrow T$ of $\id[\catL] \colon \catL \to \catL$ and $T \colon \catL \to \catL$;
		\item a natural isomorphism between functors 
		$\catL^{\mathrm{op}} \times \catL \times \catL \to \Set$:
		\[
		\catL(A, B \otimes C)
		\;\cong\;
		\Sigma_{f \in \catL(A, T B)} \catL\bigl(\abstrcompldom(A, B, f),\, C\bigr).
		\]
	\end{itemize}
	Dually, $(\catL, I, \otimes, a, l, r)$ is said to be 
	\emph{$\Sigma$-cotractable} if 
	$(\catL^{\mathrm{op}}, I, \otimes^{\mathrm{op}}, a^{-1}, l^{-1}, r^{-1})$ 
	is $\Sigma$-tractable.
\end{definition}

\begin{lemma}[$\Sigma$-tractable, $\catC$-enriched $\Rightarrow$ $\Sigma,\catC$-tractable]
Suppose that
\begin{itemize}
    \item $\catL$ is a $\Sigma$-tractable monoidal category;
    \item $\catL(-,-):\catL^{op}\times \catL\to \Set$ factors over $\catC(\terminal,-):\catC\to \Set$ for some category $\catC$ with a terminal object $\terminal$ (for example because $\catL$ is enriched over $\catC$).
\end{itemize}
Then, $\catL$ is $\Sigma,\catC$-tractable.
\end{lemma}
\begin{proof}
Write $(-)\multimap (-)$ for the functor $\catL^{op}\times \catL\to \catC$, such that $\catC(\terminal, A\multimap B)\cong \catL(A,B)$.
Then, using $T:\catL\to\catL$ and $(-)\multimap (-)$, we have a composite functor $(-)\multimap T(-):\catL^{op}\times \catL\to\catC$.
Observe that we have an equivalence of comma categories $\terminal\downarrow ((-)\multimap T(-)) \simeq \catL\downarrow T$ to get our desired $\abstrcompldom{}{}{}$.
We get the desired natural isomorphism by definition of $(-)\multimap (-)$.
\end{proof}
\begin{corollary}
Any $\Sigma$-tractable monoidal category $\catL$ is, in particular, $\Sigma,\Set$-tractable. 
\end{corollary}
We will be most interested in such $\Sigma$-tractable monoidal structures as they give us the vast majority of our examples of $\Sigma,\catC$-tractable monoidal structures.

Given a $\Sigma$-tractable monoidal structure, observe that we have a natural transformation $\catL(A, B\otimes C)\to \catL(A, TB)$, hence, by the Yoneda lemma, a natural transformation $B\otimes C\to TB$.
\begin{lemma}
In case $\catL$ has a terminal object $\terminal$ and a $\Sigma$-tractable monoidal structure, then $T\cong(-)\otimes \terminal$.
\end{lemma}
\begin{proof}
Take $C=\terminal$ in the definition of $\Sigma$-tractable monoidal structure: $$\catL(A, B\otimes \terminal)\cong \Sigma f\in \catL(A,TB).\catL(\abstrcompldom (A, B, f), \terminal)\cong  \Sigma f\in \catL(A,TB).\terminal\cong \catL(A,TB).$$ Then use that the Yoneda embedding is fully faithful.
\end{proof}

The basic example of $\Sigma$-(co)tractable monoidal structures are given by products and coproducts.
\begin{example}[Products and coproducts]\label{ex:prodcoprod}
A cartesian (resp., cocartesian) monoidal structure on $\catL$ is always $\Sigma$-tractable (resp., $\Sigma$-cotractable) with $T=\id_\catL$ and $\abstrcompldom (A,B,f)=A$.
\end{example}
In Section \ref{sec:sigma-tractable-coproducts}, we study examples of categories for which the coproducts are $\Sigma$-tractable or the products are $\Sigma$-cotractable, which is not always guaranteed.

\begin{proposition}[Monoidal closure]\label{ex:monoidal-closure}
	Suppose that $\catL$ carries a left closed monoidal structure.
	Then $\catL$ is $\Sigma,\catC$-cotractable for any category $\catC$ with a terminal object.
	Indeed:
	\begin{itemize}
		\item we take $T(-)\multimap (-)\colon \catL^{op}\times \catL \to \catC$ to be the constant functor $\Delta_\terminal$;
		\item we define $\abstrcompldom=(\pi_1\leftmultimap \pi_2)\colon \terminal\downarrow \Delta_\terminal \cong (\catL^{op}\times\catL)\to \catL$ by
		\[
		(B,A,\id_\terminal)\longmapsto (B\leftmultimap A);
		\]
		\item there is a natural isomorphism
		\[
		\catL(B\otimes C, A)
		\cong \catL(C, B\leftmultimap A)
		\cong \Sigma_{f\in \catC(\terminal,\terminal)}\, \catL(C, (\pi_1\leftmultimap\pi_2)(B,A,f)),
		\]
		since $B\leftmultimap (-)$ is right adjoint to $B\otimes (-)$.
	\end{itemize}
	
	If, in addition, $\catL$ has an initial object $\initial$, we may take $T\colon \catL\to \catL$ to be the constant functor $\Delta_\initial$, so that $T(-)\multimap (-)=\Delta_\terminal$.
	In this case $\catL$ is $\Sigma$-cotractable.
\end{proposition}

\begin{example}
	Typical examples of categories $\catL$ satisfying Proposition~\ref{ex:monoidal-closure} include any cartesian closed category, such as $\Set$, $\Pos$, or $\Cat$.
	
	A further broad class of examples arises from Eilenberg–Moore categories $\algebras{S}{\catC}$ for a commutative monad $S$ on a symmetric monoidal closed category $\catC$ with equalisers and coequalizers.
	The induced symmetric monoidal structure on $\algebras{S}{\catC}$ is described in~\cite{keigher1978symmetric}, while the closed structure is due to~\cite{kock1971closed};
	a concise modern account is given in~\cite[Theorem~2.3.3]{vakar2017search}.
	
	Other standard sources of monoidal closed categories—hence of $\Sigma$-tractable monoidal structures—are categories of $\catV$-enriched presheaves $\catV\text{-}\CAT(\catC^{op},\catV)$, equipped with the Day convolution~\cite{day1970closed}.
\end{example}

\begin{example}[Product categories]\label{ex:product-categories}
Observe that any product $\bigsqcap_{i\in I}\catV_i$ of categories $\catV_i$ with $\Sigma$-tractable (resp., $\Sigma$-cotractable) monoidal structures has a $\Sigma$-tractable  (resp., $\Sigma$-cotractable) monoidal structure.
Similarly, if $\catC$ has products, then any product $\bigsqcap_{i\in I}\catV_i$ of categories $\catV_i$ with $\Sigma,\catC$-tractable (resp., $\Sigma,\catC$-cotractable) monoidal structures has a $\Sigma,\catC$-tractable  (resp., $\Sigma,\catC$-cotractable) monoidal structure.
\end{example}

\subsection{$\Sigma$-tractable Coproducts and $\Sigma$-Cotractable Products}\label{sec:sigma-tractable-coproducts}
Recall that coproducts $B\sqcup C$ in a category $\catL$ (like all colimits and left adjoints) are defined via a mapping-out
property: morphisms $B\sqcup C\xto{a} A$ \emph{out of} the coproduct are easy to analyse.
Such $a$ always correspond precisely to pairs of $B\xto{a_B}A$ and $C\xto{a_C}A$.
Put differently, we have a natural isomorphism
$$
\catL(B\sqcup C,A)\cong \catL(B,A)\times\catL(C,A).
$$
We can convert between both representations using coprojections and copairing.
Indeed, that is precisely the universal property of coproducts.
In particular, coproducts are always $\Sigma$-cotractable.

Problematically, however, we might not have any tools for analysing morphisms $A\to B\sqcup C$ \emph{into} a coproduct.
To be able to say anything about such morphisms, we need to impose extra axioms.
The same goes for analysing morphisms $B\times C\to A$ out of a product.

Interestingly, large classes of coproducts (resp., products) we encounter in practice are $\Sigma$-tractable (resp., $\Sigma$-cotractable). 
We give some important classes of examples.
\begin{example}[Biproducts]\label{ex:biprod}
Suppose that $\catL$ has binary products and binary coproducts that coincide (for example, because $\catL$ has biproducts/has finite products and is $\CMon$-enriched).
Then, these are $\Sigma$-tractable coproducts, and, by duality, $\Sigma$-cotractable products, by Example \ref{ex:prodcoprod}.
\end{example}
Concrete examples are the categories $\CMon$ of commutative monoids and homomorphisms and $\Vect$ of vector spaces and linear functions.

\begin{example}[Cartesian closure] \label{ex:ccc}
By Proposition~\ref{ex:monoidal-closure}, products are $\Sigma,\catC$-cotractable in a cartesian closed category $\catL$
(for any $\catC$ with a terminal object).
Further, they are $\Sigma$-cotractable if $\catL$ additionally has an initial object.
\end{example}

\begin{example}[Extensive category] \label{ex:extterm}
Recall that a (finitely) extensive category $\catL$ is a category with finite coproducts such that 
$$
(-)\sqcup(-):\catL/B\times \catL/C\to \catL/(B\sqcup C)
$$
defines an equivalence.
(That is, if the codomain fibration is extensive.)
As a consequence, the equivalence inverse is given by the pullbacks $g\mapsto (\iota_1^*g, \iota_2^*g)$.

Assuming that $\catL $ is extensive, let us write $\partial g\rightarrowtail A\leftarrowtail \compldom g$ for the coproduct diagram that is obtained as the pullback along $g:A\to B\sqcup C$  of the coproduct diagram $B\rightarrowtail B\sqcup C\leftarrowtail C$.

Then, if $\catL$ further has a terminal object, its coproducts are $\Sigma$-tractable:
\begin{itemize}
    \item we have coproducts by assumption;
    \item we take $T=(-)\sqcup \terminal:\catL\to\catL$ to be the functor that takes the coproduct with the terminal object;
    \item we take $\abstrcompldom:\catL\downarrow (-)\sqcup\terminal\to\catL$ to be the  functor that takes $(A, B, f:A\to B\sqcup \terminal)$ to the pullback $\compldom f$;
    \item we have the natural isomorphism 
    \begin{align*}
    \catL(A,B\sqcup C)&&&\cong  &
    \Sigma f\in \catL(A, B\sqcup \terminal).
    \catL(\compldom f, C)\\
    g &&& \mapsto & ((B\sqcup !_C)\circ g, \iota_2^* g)\\
    [\iota_1\circ (\iota_1^* f),\iota_2\circ f'] &&&\mapsfrom& (f,f').
    \end{align*}
    because in the following diagram all commutative rectangles are pullbacks and all horizontal and diagonal arrows are coproduct inclusions\\
\begin{center} 
\begin{tikzcd}
                                                      & \partial g\sqcup \compldom g \arrow[d, no head, equal] &                                                         \\
\partial g \arrow[r, tail] \arrow[d,"\iota_1^*g"] \arrow[ru, tail, "\iota_1"] & A \arrow[d, "g={[\iota_1^*g, \iota_2^* g]} "]                                             & \partial^c g \arrow[l, tail] \arrow[d,"\iota_2^*g"] \arrow[lu, tail, "\iota_2"] \\
B \arrow[r, tail, "\iota_1"] \arrow[d, no head, equal]      & B\sqcup C \arrow[d, "B\sqcup !_C"]                           & C \arrow[l, tail, "\iota_2"]       \arrow[d, "!_C"]                                \\
B \arrow[r, tail, "\iota_1"]                                     & B\sqcup \terminal                                                    &    \arrow[l, tail, "\iota_2"]         \terminal                                             
\end{tikzcd}
\end{center} 
\end{itemize}
\end{example}
Some concrete examples are the categories $\Set$ of sets and functions and $\CTop$ of topological spaces and continuous functions.

\begin{example}[Free coproduct completions]\label{ex:fam}
Consider the free coproduct completion $\Fam{\catC}$ of $\catC$.
Recall that $\Fam{\catC}$ has objects that are a pair of a set $I$ and an $I$-indexed family $\sumfam{C_i}{i\in I}$ of $\catC$-objects $C_i$.
The homset $\Fam{\catC}(\sumfam{C_i}{i\in I},\sumfam{C'_{i'}}{i'\in I'})$ is $$\Pi_{i\in I}\Sigma_{i'\in I'}\catC(C_i,C'_{i'}).$$

In case $\catC$ has a terminal object, then $\Fam{\catC}$ is an extensive category with a terminal object, 
so by Example \ref{ex:extterm}, it is has $\Sigma$-tractable coproducts.
However, even if $\catC$ does not have a terminal object, $\Fam{\catC}$ always has $\Sigma,\Set$-tractable coproducts.
Indeed, while we cannot define the monad $(-)\sqcup \terminal$ on $\Fam{\catC}$, unless $\catC$ has a terminal object, we can always define the functor 
$$(-)\multimap (-)\sqcup \terminal: \Fam{\catC}^{op}\times \Fam{\catC}\to \Set $$
by 
$$(\sumfam{C'_{i'}}{i'\in I'},\sumfam{C_i}{i\in I})\mapsto \Pi_{i'\in I'}\Sigma_{i\in I\sqcup \{\bot\}}\catC(C_i, C'_{i'})\text{ if $i\neq\bot$ else } \{\bot\}.$$
Further, we can define 
$$\abstrcompldom( \sumfam{C'_{i'}}{i'\in I'},\sumfam{C_i}{i\in I}, f)=\sumfam{C'_{i'}}{i'\in I', f(i')=\langle \bot,\bot\rangle}.$$
Then,
\small 
\begin{align*}
&\Fam{\catC}(
\sumfam{C''_{i''}}{i''\in I''},\sumfam{C'_{i'}}{i'\in I'}\sqcup\sumfam{C_i}{i\in I} ) \cong\\
& \Fam{\catC}(
\sumfam{C''_{i''}}{i''\in I''},\sumfam{\begin{array}{ll}C'_{k}& \text{if $k\in I'$}\\C_{k}& \text{if $k\in I$}\end{array}}{k\in I'\sqcup I} )=\\
&\Pi_{i''\in I''}\Sigma_{k\in I'\sqcup I} \catC\left(C''_{i''}, \begin{array}{ll}C'_{k}& \text{if $k\in I'$}\\C_{k}& \text{if $k\in I$}\end{array}\right)\cong \\
&\Sigma_{f\in \Pi_{i''\in I''}\Sigma_{k\in I'\sqcup \{\bot\}}\catC(C''_{i''}, C'_{i'})\text{ if $k\neq \bot$ else }\{\bot\}}\Pi_{i''\in I'', f(i'')=\langle \bot,\bot\rangle}\Sigma_{i\in I}\catC(C''_{i''}, C_i)=\\
&\Sigma f\in \Set(\terminal,\sumfam{C''_{i''}}{i''\in I''}\multimap
\sumfam{C'_{i'}}{i'\in I'}\sqcup \terminal 
). \Fam{\catC}( \abstrcompldom(\sumfam{C''_{i''}}{i''\in I''},\sumfam{C'_{i'}}{i'\in I'},f),\sumfam{C_{i}}{i\in I}).
\end{align*}\normalsize 
\end{example}
By duality, products are always $\Sigma,\Set$-cotractable in a free product completion $\Fam{\catC^{op}}^{op}$ of $\catC$.

\begin{example}[Product categories]
Specialising Example \ref{ex:product-categories},
observe that any product of categories with $\Sigma$-tractable coproducts has $\Sigma$-tractable coproducts.
This gives us examples of $\Sigma$-tractable coproducts that do not arise from our Examples \ref{ex:biprod}, \ref{ex:ccc}, and \ref{ex:extterm}, like 
$\Set^{op}\times \Set$, which has $\Sigma$-tractable coproducts, but does not have biproducts (as $\Set$ does not have biproducts), is not (co)-cartesian (co)-closed (as $\Set^{op}$ is not cartesian closed), and is not extensive (as coproducts in $\Set^{op}$ are not disjoint). 
By a similar argument (as a self-dual category) $\Set^{op}\times\Set$ has $\Sigma$-cotractable products.  
\end{example}

\begin{example}[Partial functions] \label{ex:pset}
Consider the category $\pSet$ of sets and partial functions.
The coproduct $S\sqcup S'$ is the usual disjoint union of sets while the product $S\times_p S'$ is given by $S\times S'\sqcup S \sqcup S'$, where we write $\times_p$ for the product in $\pSet$ and $\times $ for the usual product in $\Set$.
Obviously, $\pSet$ does not have biproducts.
Clearly, $\pSet$ is not distributive hence not extensive:
\begin{align*}
X \times_p (Y \sqcup  Z)
&\cong
X \times (Y \sqcup  Z) \sqcup  X \sqcup  (Y \sqcup  Z) 
\\& \cong
X \times Y \sqcup  X \times Z \sqcup  X \sqcup  Y \sqcup  Z
\\&\not \cong
X \times Y \sqcup  X \times Z \sqcup  X \sqcup  Y \sqcup  Z \sqcup  X
\\&\cong 
X \times Y \sqcup  X \sqcup  Y \sqcup  X \times Z \sqcup  X \sqcup  Z
\cong
X \times_p Y \sqcup  X \times_p Z.
\end{align*}
Moreover, $\pSet$ with the cocartesian monoidal structure is not monoidal coclosed as $X\sqcup (-)$ does not preserve products:
\begin{align*}
X \sqcup  (Y \times_p Z)
&\cong
X \sqcup  Y \times  Z \sqcup  Y \sqcup  Z
\\&\not\cong
X \sqcup  Y \times  Z \sqcup  Y \sqcup  Z \sqcup  X \times  X \sqcup  X \times  Z \sqcup  Y \times  X \sqcup  X
\\&\cong
X \times  X \sqcup  X \times  Z \sqcup  Y \times  X \sqcup  Y \times  Z \sqcup  X \sqcup  Y \sqcup  X \sqcup  Z
\\&\cong
(X \sqcup  Y) \times  (X \sqcup  Z) \sqcup  X \sqcup  Y \sqcup  X \sqcup  Z
\\&\cong
(X \sqcup  Y) \times_p (X \sqcup  Z)
\end{align*}
However, $\pSet$ does have $\Sigma$-tractable coproducts, for $T=\id$ and $\abstrcompldom(A, B, f)=A\setminus f^{-1}(B)$. Indeed,
$$\pSet(A, B\sqcup C) \cong \Sigma f \in \pSet(A, B). \pSet(A \setminus f^{-1}(B), C).$$
This shows that the cocartesian structure on $\pSet$ is a $\Sigma$-tractable coproduct that does not arise from our Examples \ref{ex:biprod}, \ref{ex:ccc}, and \ref{ex:extterm}.
\end{example}

\begin{example}[$\Sigma$-tractable posets] \label{ex:posets}
Let $X$ be a poset with $\Sigma$-tractable coproducts.
Observe that we have a natural transformation $X(x, y\vee z)\to X(x, Ty)$.
Therefore, by the full and faithfulness of the Yoneda embedding, we get a morphism 
$y\vee z\leq Ty$ and, in particular, a morphism $z\leq Ty$.
We see that $Ty$ is the terminal object $\top$ of $X$.
Therefore, the condition for $\Sigma$-tractability is that 
\begin{align*}
X(x,y\vee z)&\cong  \Sigma f\in X(x,Ty).X(\abstrcompldom(x,y,f),z)\cong X(x,Ty)\times X(\abstrcompldom(x,y),z) \\
&\cong X(x,\top)\times X(\abstrcompldom(x,y),z)\cong 1\times  X(\abstrcompldom(x,y),z) \cong  X(\abstrcompldom(x,y),z).
\end{align*}
That is, $X$ having finite coproducts that are $\Sigma$-tractable is equivalent to $X^{op}$ being cartesian closed with an initial object (Example \ref{ex:ccc}).

As an aside, note that posets with biproducts are trivial ($a\wedge b\leq a,b\leq a\vee b$ implies that $a\wedge b=a\vee b$ iff $a=b$) and extensive posets are trivial (extensive posets are distributive lattices, by definition, and disjointness of coproducts implies that $a\wedge b=\bot$ if $a\neq b$; in particular if $a<b$, $a=a\wedge b=\bot$).
\end{example}

\begin{counterexample}[Non-distributive lattices]
From Example \ref{ex:posets}, we see that for any non-distributive lattice $X$ (the typical examples being $M_3$ and $N_5$), $X^{op}$ has coproducts that are not $\Sigma$-tractable. 
To make this counterexample very concrete, consider the lattice $M_3$:\\
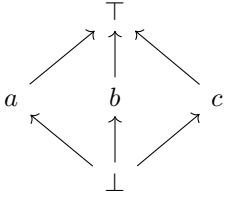
\begin{center} 
\begin{tikzcd}
    & \top & \\
    a \arrow[ru] & b \arrow[u] & c \arrow[lu] \\
    & \bot\arrow[lu]\arrow[u]\arrow[ru] &
\end{tikzcd}\\
\end{center} 
Then, $X \cong X^{op}$ is not distributive as $a\wedge(b\vee c) = a \wedge \top = a$ while $(a\wedge b)\vee (a\wedge c)= \bot\vee \bot=\bot$.
In particular,  $X \cong X^{op}$ is not cartesian closed (not a Heyting algebra), as $a\wedge (-)$ does not preserve coproducts.
Therefore, the coproducts in $X$ are not all $\Sigma$-tractable.
\end{counterexample}

\section{A Dialectica-like Formula for the Monoidal Closure of Grothendieck Constructions}\label{sec:exponentials-in-grothendieck}

We now turn to our main result: sufficient conditions ensuring the existence of a monoidal closed structure on a Grothendieck construction.
The resulting formula for the closed structure generalises G\"odel's Dialectica construction~\cite{godel1958bisher}.
Its formulation requires certain dependent type-theoretic primitives—namely $\Sigma$- and $\Pi$-types—to express.
While these notions are familiar from type theory and proof theory, their categorical expression is necessarily more elaborate.
We therefore begin by recalling the relevant definitions and terminology.

\subsection{Sufficient conditions for the monoidal closure of $\Sigma_\catC \catL$}\label{sec:sufficient-conditions}

In what follows we make systematic use of the language of dependent type theory.
For this reason, we first recall the categorical structure underlying models of dependent type theory with $\Pi$- and strong $\Sigma$-types, following Jacobs~\cite{jacobs1999categorical} and Vákár~\cite{vakar2017search}.

\paragraph{\textbf{A model of (cartesian) dependent type theory $\catC' \colon \catC^{op}\to\CAT$.}}
Let $\catC' \colon \catC^{op}\to\CAT$ be a model of (cartesian) dependent type theory with $\Pi$-types and strong $\Sigma$-types.
That is, $\catC'$ is an indexed category satisfying  comprehension in the sense of~\cite[Definition~2.1.4]{vakar2017search}, or equivalently a cloven  comprehension category (with unit) in the sense of~\cite[Definitions~10.4.2,~10.4.7]{jacobs1999categorical}, equipped with $\Pi$-types, terminal types, and strong $\Sigma$-types (see~\cite[Theorem~2.1.7]{vakar2017search} or~\cite[Definitions~10.5.1,~10.5.2(i)]{jacobs1999categorical}).

For completeness, we spell out the relevant data.

\begin{definition}[Model of dependent type theory with $\Pi$-types and strong $\Sigma$-types]
	A model of dependent type theory consists of the following data:
	\begin{itemize}
		\item An indexed category $\catC' \colon \catC^{op}\to\CAT$ over a base category $\catC$ with a terminal object $\terminal$.
		\item \emph{Indexed terminal objects}, given by a right adjoint $\terminal \colon \catC \to \Sigma_\catC \catC'$ to the projection $\pi_1 \colon \Sigma_\catC \catC' \to \catC$.
		\item Comprehension: a further right adjoint $(\_.\_)\colon \Sigma_\catC \catC' \to \catC$ to $\terminal$; we write
\[
\mathbf{p}_{\_}\colon \Sigma_\catC \catC' \longrightarrow \catC^{\to}, \qquad 
\mathbf{p}_{W,w} \defeq \pi_1(\epsilon_{(W,w)}),
\]
where $\epsilon$ is the counit of $\terminal \dashv (-.-)$..
		\item \emph{Strong $\Sigma$-types}, that is, left adjoints $\Sigma_w\colon \catC'(W.w)\to\catC'(W)$ to $\catC'(\depproj{W}{w})$, satisfying the left Beck–Chevalley condition: the canonical natural transformations
		\[
		\Sigma_{\catC'(f)(w)}\circ \catC'(\mathbf{q}_{f,w})
		\longrightarrow 
		\catC'(f)\circ \Sigma_w
		\]
		are isomorphisms, where $\mathbf{q}_{f,w}$ is the unique morphism making the following square a pullback:
		\begin{center}
			\begin{tikzcd}
				W'.\catC'(f)(w) \arrow[r, "{\mathbf{q}_{f,w}}"] \arrow[d, "{\mathbf{p}_{W',\catC'(f)(w)}}"'] & W.w \arrow[d, "{\mathbf{p}_{W,w}}"] \\
				W' \arrow[r, "f"'] & W
			\end{tikzcd}
		\end{center}
		Moreover, there is a canonical isomorphism $\depproj{W}{w}\circ \depproj{W.w}{w'} \cong \depproj{W}{\Sigma_w w'}$.
		In particular, $\catC'$ has indexed binary products, given by
		\[
		w\times w' \;\defeq\; \Sigma_w \catC'(\mathbf{p}_{W,w})(w')
		\qquad \text{for } w,w'\in\ob\catC'(W).
		\]
		For our purposes, the weaker assumption of \emph{non-dependent} $\Sigma$-types—left adjoints to $\catC'(\pi_1)$ for $\pi_1\colon W\times W'\to W$—suffices.
		\item \emph{$\Pi$-types}, defined as $\Sigma$-types in $\catC'{}^{op}$; that is, right adjoints $\Pi_w\colon \catC'(W.w)\to\catC'(W)$ to $\catC'(\depproj{W}{w})$, satisfying the right Beck–Chevalley condition: the canonical natural transformations
		\[
		\catC'(f)\circ \Pi_w
		\longrightarrow 
		\Pi_{\catC'(f)(w)} \circ \catC'(\mathbf{q}_{f,w})
		\]
		are isomorphisms.
		In particular, $\catC'$ has indexed exponentials given by
\[
w\Rightarrow w' \;\defeq\; \Pi_w \catC'(\mathbf{p}_{W,w})(w')
\qquad \text{for } w,w'\in\ob\catC'(W).
\]
Throughout, we freely use the bijections $\catC'(W)(\terminal,z)\cong \catC/W(\id_W,\mathbf{p}_{W,z})$ coming from the adjunction $\terminal \dashv (-.-)$; no fullness or faithfulness of $\mathbf{p}$ is required.
		For our applications, the weaker assumption of \emph{non-indexed} $\Pi$-types—right adjoints $\Pi_w\colon \catC'(\terminal.w)\to\catC'(\terminal)$ to $\catC'(\depproj{\terminal}{w})$—is sufficient.
	\end{itemize}
\end{definition}

\begin{example}[Families of sets]\label{ex:dtt-fam}
	Let $\catC=\Set$ and $\catC'(S)\defeq \CAT(S,\Set)$.
	Then $\Sigma_\catC \catC' = \Fam{\Set}$, the category of families of sets (the free coproduct completion of $\Set$).
	The comprehension functor $(-.-)\colon \Fam{\Set}\to \Set$ is given by disjoint union.
	Strong $\Sigma$-types correspond to disjoint unions, and $\Pi$-types to products.
	See~\cite{hofmann1997syntax} for details.
\end{example}

\begin{example}[Continuous families of $\omega$-cpos]\label{ex:dtt-wfam}
	Let $\catC=\wCpo$ be the category of $\omega$-cocomplete partial orders and $\omega$-cocontinuous maps, and define
	\[
	\catC'(X)=\mathbf{\omega ContFunc}(X, \wCpo_{ep}),
	\]
	the category of $\omega$-cocontinuous functors from $X$ into the category of $\omega$-cpos and embedding–projection pairs, with lax natural transformations.
	This yields a model of $\omega$-continuous families of $\omega$-cpos; see, for example,~\cite{palmgren1990domain,ahman2016dependent}.
\end{example}

\begin{example}[Locally cartesian closed categories]\label{ex:dtt-lccc}
	Another fundamental source of examples is given by codomain fibrations $\mathrm{cod}\colon \catC^{\to}\to\catC$ (with a chosen cleavage) over locally cartesian closed categories~\cite{seely1984locally,clairambault2014biequivalence}.
	In this case we take $\catC'(C)=\catC/C$.
	  The  comprehension $(-.-)\colon \catC^{\to}\to\catC$ is given by the domain functor $\mathrm{dom}\colon \catC^{\to}\to\catC$.
	Strong $\Sigma$-types correspond to composition, and $\Pi$-types to the right adjoints of pullback functors.
\end{example}

\begin{example}[Product self-indexing] \label{ex:dtt-locally-indexed}
Given a cartesian closed category $\catC$, we can form the locally indexed category $\self(\catC):\catC^{op}\to\CAT$ (see Example \ref{ex:locally-indexed}).
Then, $\self(\catC)$ is a model of dependent type theory.
Indeed, the comprehension $(.-.):\Sigma_\catC \self(\catC)\to \catC$
is defined as $(C, C')\mapsto C\times C'$ on objects and 
$(f:C_1\to C_2, g: C_1\times C'_1\to C'_2)\mapsto (\langle f\circ \pi_1, g \rangle : C_1\times C'_1\to C_2\times C'_2)$ on morphisms.
Strong $\Sigma$-types are defined as $\Sigma_C C'=C\times C$ and $\Pi$-types are defined as $\Pi_C C'=C\Rightarrow C'$.
\end{example}
\begin{example}[Indexed category of indexed categories] \label{ex:dtt-indexed-cats}
    This example categorifies the families construction of Example \ref{ex:dtt-fam} and replaces $\Set$ with $\Cat$ and $\CAT$ with $\twoCAT$.
    We have a model of dependent type theory: $\catC=\Cat$ and $\catC'(\catC)=\twoCAT(\catC^{op},\Cat)_{oplax}$ is the category of (strict) $\catC$-indexed categories and oplax natural transformations.
    This indexed category satisfies the  comprehension axiom with $(-.-):\Sigma_{\Cat}\twoCAT(\catC^{op},\Cat)_{oplax}\to \Cat$ given by the Grothendieck construction \cite{north2019towards}.
    It has strong $\Sigma$-types $\Sigma_\catC \catL$ given by the oplax colimit, which exists as a functor $\Sigma_\catD:\twoCAT((\catC.\catD)^{op},\Cat)_{oplax}\to \twoCAT(\catC^{op},\Cat)_{oplax} $ and is precisely the Grothendieck construction (as should be clear from Proposition \ref{prop:2equivalence}; see \cite{gray1974quasi} for the original reference and details -- note that Gray calls these quasi-(co)limits).
    It also has non-parameterised $\Pi$-types $\Pi_\catC \catL$ given by the oplax limit, which exists as a functor $\Pi_\catD:\twoCAT(\catD^{op},\Cat)_{oplax}\cong\twoCAT((\terminal.\catD)^{op},\Cat)_{oplax}\to \twoCAT(\terminal^{op},\Cat)_{oplax} \cong \Cat$ and is given by the category of sections of the Grothendieck construction 
    (i.e., functors $F:\catC\to\Sigma_\catC \catL$ such that $\pi_1 F=\id[\catC]$ and natural transformations $\alpha:F\to G$ such that $\pi_1 \alpha = \id[{\id[\catC]}]$) \cite{gray1974quasi}.
\end{example}

\begin{example}[Indexed groupoids]
\cite{hofmann1998groupoid} restricts Example \ref{ex:dtt-indexed-cats} to 
indexed groupoids indexed by another groupoid.
That gives us another model of dependent type theory and it is the starting point for homotopy type theory, where people consider variants of this model based on $\infty$-groupoids rather than 1-groupoids \cite{kapulkin2021simplicial}.
\end{example}
Needless to say, many other models exist, such as ones based on polynomials \cite{glehnmoss2018}.\\

\paragraph{\textbf{A $\Sigma$-cotractable indexed monoidal category $\catL:\catC^{op}\to\CAT$ with $\Pi$-types}}
Further, assume that we have a model $\catL$ of linear dependent type theory \cite[Chapter 2]{vakar2017search} over the same base category $\catC$, with a $\Sigma$-tractable monoidal structure, in the following sense:
\begin{itemize}
    \item an indexed category $\catL:\catC^{op}\to\CAT$;
    \item $\catL$ has $\Pi$-types in the sense of right adjoints $\Pi_w:\catL(W.w)\to \catL(W)$ to $\catL(\depproj{W}{w})$ that satisfy the right Beck--Chevalley condition, i.e., the canonical natural transformations $\catL(f)\circ \Pi_w\to\Pi_{\catC'(f)(w)} \circ\catL(\mathbf{q}_{f,w})$ are an isomorphism;
    in fact, for our purposes, the weaker assumption of \emph{non-dependent} $\Pi$-types in the sense of right adjoint functors to $\catL(\pi_1)$ for (non-dependent) product projections $\pi_1: W\times W'\to W$ suffice;
    \item $\catL$ has an indexed $\Sigma,\catC'$-cotractable monoidal structure\footnote{Observe that this last condition is, in particular, implied by the following pair of conditions that often holds in practice:
    \begin{itemize}
    \item $\catL$ is enriched over $\catC'$, or more weakly, we have $\catL-\multimap$-types in $\catC'$
    in the sense of that we have an indexed functor $(-)\multimap (-):\catL^{op}\times \catL\to \catC'$ and a natural isomorphism 
    $$
    \catL(W)(A, B)\cong \catC'(W)(\terminal, A\multimap B);
    $$
    \item $\catL$ has an indexed $\Sigma$-cotractable monoidal structure in the sense of an indexed monoidal structure on $\catL$ such that on each fibre $\catL(C)$ the monoidal structure is $\Sigma$-cotractable
    and $T$ and $\abstrcompldom$ are $\catC$-indexed functors.
    \end{itemize}} in the sense of an indexed monoidal structure on $\catL$ such that on each fibre $\catL(C)$ the monoidal structure is $\Sigma,\catC'(C)$-cotractable
    and $T(-)\multimap (-)$ and $\abstrcompldom$ are $\catC$-indexed functors.
    \end{itemize}
    For example, the fibre categories of $\catL$ could have $\Sigma$-cotractable monoidal structure because they are monoidal closed with an initial object, because they have biproducts, or because they are co-extensive with an initial object).
\begin{example}[$\catL=\catC'$] Take $\catL=\catC':\catC^{op}\to\CAT$ to be any model of dependent type theory with $\Pi$-types and strong $\Sigma$-types.
Then, $\catL$ is an indexed cartesian closed category and $\catC$ has a terminal object.
By Proposition \ref{ex:monoidal-closure}, products in $\catL$ are $\Sigma,\catC'$-cotractable.
Observe that $\Sigma_\catC\catL=\Sigma_\catC\catC'$.
\end{example}
\begin{example}[$\catL=\catC'{}^{op}$, extensive] \label{ex:ldttt-op}
Take $\catC':\catC^{op}\to\CAT$ to be any model of dependent type theory with $\Pi$-types and strong $\Sigma$-types.
Further, assume that $\catC'$ has indexed coproducts and that the categories $\catC'(C)$ are extensive.
Then, by Example \ref{ex:extterm}, $\catL=\catC'{}^{op}$ has indexed $\Sigma$-cotractable products.
Further, it has $\Pi$-types, given by the $\Sigma$-types of $\catC'$.
Observe that $\Sigma_\catC\catL=\Sigma_\catC\catC'{}^{op}$.
\end{example}
\begin{example}[Coproducts/biproducts] \label{ex:ldttt-biproducts}
Let $\catC':\catC^{op}\to\CAT$ be a model of dependent type theory with $\Pi$-types and strong $\Sigma$-types.
Let $\catL:\catC^{op}\to \CAT$ be any indexed category with finite indexed coproducts, such that the hom-functor of $\catL$ factors over $\catC'$.
(For example, we can take $\catL=\catC'{}^{op}$.)
Seeing that coproducts always form a $\Sigma$-cotractable monoidal structure, it follows that they are a $\Sigma, \catC'$-cotractable monoidal structure.
Observe that the case where $\catL$ has indexed biproducts is of particular interest as, in that case, $\catL$ has $\Sigma,\catC'$-cotractable products.
\end{example}
\begin{example}[Locally indexed categories] \label{ex:ldtt-locally-indexed}
  This Example builds on the choice $\catC'=\self(\catC)$ of Example \ref{ex:dtt-locally-indexed}.
  Suppose that $\catD$ is a $\catC$-enriched category.
  Then, it, in particular, defines a locally $\catC$-indexed category $\catL(C)(D, D')=\catC(C, \catD(D, D'))$.
  If $\catD$ is $\catL$ is $\catC$-powered in the sense that $\catC(C, \catD(D, D'))\cong \catD(D, C\Rightarrow D')$, then $\catL$ has $\Pi$-types: $\Pi_C D=C\Rightarrow D$.
  If $\catD$ has a $\Sigma$-cotractable monoidal structure, then it meets our conditions.
\end{example}
\begin{example}[Dual product self-indexed] \label{ex:dual-prod-self-indexed}
  This Example builds on the choice $\catC'=\self(\catC)$ for a cartesian closed category $\catC$ of Example \ref{ex:dtt-locally-indexed}, and it specialises Example \ref{ex:ldttt-biproducts}.
  Observe that $\self(\catC)^{op}$ is a  (locally) $\catC$-indexed category with indexed coproducts (products in $\catC$). 
  Further, it has $\Pi$-types given by $\Pi_C C'= C\times C'$ products in $\catC$.
  Seeing that coproducts are always $\Sigma$-tractable and seeing that $\self(\catC)^{op}$ is $\self(\catC)$ enriched, it follows that  $\self(\catC)^{op}$ has a $\Sigma, \self(\catC)$-cotractable monoidal structure.
\end{example}
\begin{example}[Families]
Building on the choice of $\catC'$ of Example \ref{ex:dtt-fam},
for any category $\catD$ with a $\Sigma,\Set$-cotractable monoidal structure (for example, $\catD$ monoidal closed, a free product completion, co-extensive with an initial object, or a category with biproducts) and small products, $\catL:\Set^{op}\to\CAT$ with $\catL(S)=\CAT(S, \catD)$ meets our conditions.
The $\Pi$-types are given by products in $\catD$ (see \cite{vakar2015categorical}).

For example, we may take $\catD$ to be a product-complete monoidal closed category such as a category of algebras for a commutative algebraic theory on $\Set$.
Observe that $\Sigma_\catC\catL=\Fam{\catD}$.
\end{example}
\begin{example}[$\omega$-Continuous families]
This Example builds on the choice of $\catC'$ of Example \ref{ex:dtt-wfam}.
Given an $\wCpo$-enriched Lawvere theory, we may take $\catD$ to be its category of algebras in $\wCpo$ and $\catL(X)=\mathbf{\omega ContFunc}(X, \catD_{ep})$ to be the $\wCpo$-indexed category of $\omega$-cocontinuous functors into the category of $\catD$-objects and embedding-projection pairs.
Then, $\catL$ is an indexed monoidal closed category, hence an indexed $\Sigma,\catC'$-cotractable monoidal category.
Details are discussed in \cite[Section 6]{ahman2016dependent}.
Observe that $\Sigma_\catC\catL=\mathbf{\omega Cont}\Fam{\catD}$ is the category $\omega$-continuous families of $\catD$-objects.
\end{example}

\begin{example}[Lextensive locally cartesian closed categories]
This example specializes Examples \ref{ex:dtt-lccc} and \ref{ex:ldttt-op}.
Assume that $\catC$ is a lextensive locally cartesian closed category category (for example, $\catC$ an elementary topos).
Consider the codomain fibration $\catC'(C)=\catC/C$, which we can turn into an indexed category by making use of the axiom of choice to choose pullbacks.
Observe that $\catC/C$ is also lextensive (lextensive categories are locally lextensive \cite[Proposition 4.8]{MR1201048}) with a terminal object hence has $\Sigma$-tractable coproducts.
Define $\catL=\catC'{}^{op}$.
Then, $\catL$ has $\Sigma$-cotractable products and $\Pi$-types ($\Sigma$-types in $\catC'$).
Observe that $\Sigma_\catC\catL=\Sigma_\catC (\catC/-)^{op}$ is a kind of generalised category of polynomials (or containers).
\end{example}

\begin{example}[Lax comma] \label{ex:ldttt-lax-comma}
This Example takes $\catC'$ to be defined as in Example \ref{ex:dtt-indexed-cats}.
Let $\catD$ be some 2-category with oplax limits.
(For example, we already obtain many interesting examples for $\catD$ a 1-category with limits.)
We have a $\Cat$-indexed category $\catL(\catC)=\twoCAT(\catC^{op},\catD)_{oplax}$.
Its non-dependent $\Pi$-types are simply given by oplax limits (ordinary limits, if $\catD$ is a 1-category).
If $\catD$ has a $\Sigma,\Cat$-cotractable monoidal structure (such as a $\Sigma,\Set$-cotractable one), then $\catL$ meets our conditions.
Observe that $\Sigma_\catC\catL=\Cat // \catD$ is the lax comma category of $\catD$ in $\Cat$.
Some important subcases of this example are worked out in more detail in \cite{clementino2024lax}, where the structure of exponentials is presented in terms of ends. 
\end{example}

\subsection{Monoidal Closure of $\Sigma_\catC \catL$}
We can now phrase our main theorem.
\begin{theorem}[Monoidal closure of $\Sigma_\catC\catL$ via a Dialectica formula]\label{theo:grothendieck-closed}
Assuming the conditions of Section \ref{sec:sufficient-conditions},
$\Sigma_\catC\catL$ is monoidal left-closed with 
$$
(X,x)\leftmultimap(Y,y)\defeq 
\big(\Pi_X\Sigma_Y (T x\multimap y),\,\Pi_X \catL(\zeta)(\abstrcompldom{v})\big)
$$
for two morphisms ${v}$ and $\zeta $ that we define below.
By co-duality, we obtain monoidal right-closure if $\catL^{co}$ is $\Sigma,\catC'$-cotractable instead.
\end{theorem}
\begin{proof}
By Theorem \ref{thm:fibred-monoidal}, $(\terminal_\catC,\, I_{\catL(\terminal_\catC)})$ is the monoidal unit of $\Sigma_\catC\catL$ and $(X\times Y,\catL(\pi_1)(x)\otimes\catL(\pi_2)(y))$ is the monoidal product of $(X,x)$ and $(Y,y)$ in $\Sigma_\catC\catL$.

The novel part is the existence of exponentials, which we turn to next.
We have (natural) bijections 
 (where, to aid legibility, we abuse notations a bit by leaving implicit: (1)  some weakening functors $\catL(\depproj{W}{w})$ and $\catC'(\depproj{W}{w})$, (2) the equivalence $\catC'(\terminal)\simeq \catC$, and (3) isomorphisms $X.\Sigma_Y Z\cong X.Y.Z$ where they are obvious from the context):
	\vspace{-2pt}\\
	\resizebox{\linewidth}{!}{\parbox{\linewidth}{
			\begin{align*}
				&\Sigma_{\catC}\catL((X,x)\otimes (W,w), (Y,y)) =\\
				&=\Sigma_{\catC}\catL((X\times W,\catL(\pi_1)(x)\otimes \catL(\pi_2)(w)), (Y,y))\\
				&=\Sigma_{f\in \catC(X\times W, Y)}\catL(X\times W)(\catL(\pi_1)(x)\otimes \catL(\pi_2)(w), \catL(f)(y))\\
      &\cong \Sigma_{f\in \catC(X\times W, Y)}\Sigma_{g\in\catC'(X\times W)(\terminal, T\catL(\pi_1)(x)\multimap \catL(f)(y))}\catL(X\times W)(\catL(\pi_2)(w),\abstrcompldom g)\explainr{$\otimes$ $\Sigma,\catC'$-cotractable}\\
    &\cong \Sigma_{f\in \catC(X\times W, Y)}\Sigma_{g\in\catC'(X\times W)(\terminal,T\catL(\pi_1)(x)\multimap\catL(f)(y))}\catL(X\times W)(\catL(\pi_2)(w),\abstrcompldom  g)\explainr{$\terminal \dashv (-.-)$ (implicit)}\\
 &\cong \Sigma_{(f,g)\in \Sigma_{f\in \catC(X\times W, Y)}\catC'(X\times W)(\terminal,T\catL(\pi_1)(x)\multimap\catL(f)(y))}\catL(X\times W)(\catL(\pi_2)(w),\abstrcompldom  g)  \explainr{$\Sigma$-types in $\Set$}
 \\
 &\cong \Sigma_{(f,g)\in \Sigma_{f\in \catC(X\times W, Y)}\catC'(X\times W)(\terminal,T\catL(\pi_1)(x)\multimap\catL(f)(y))}\catL(W )(w,\Pi_X\abstrcompldom  g)  \explainr{$\Pi$-types in $\catL$} \\
 &=\Sigma_{(f,g)\in \Sigma_{f\in \catC(X\times W, Y)}\catC'(X\times W)(\terminal,Tx\multimap\catL(f)(y))}\catL(W )(w,\Pi_X\abstrcompldom  g)  \explainr{implicit $\catL(\pi_1)$ for legibility}\\
 &=  \Sigma_{(f,g)\in \Sigma_{f\in \catC(X\times W, Y)}\catC'(X\times W)(\terminal ,Tx\multimap\catL(f)(y))}\catL(W)(w,\Pi_X \abstrcompldom\big(\catC'((\pi_1,f,g))(v)\big)) \explainr{definition ${v}$}\\
   &\cong\Sigma_{(f,g)\in \Sigma_{f\in \catC(X\times W, Y)}\catC'(X\times W)(\terminal ,Tx\multimap\catL(f)(y))}\catL(W)(w,\Pi_X \catL((\pi_1,f,g))(\abstrcompldom  {v})) \explainr{$\abstrcompldom$ indexed functor}\\
   &= \Sigma_{(f,g)\in \Sigma_{f\in \catC(X\times W, Y)}\catC'(X\times W)(\terminal ,Tx\multimap\catL(f)(y))}\catL(W)(w,\Pi_X \catL(\zeta \circ (\Lambda(f,g),\pi_1))(\abstrcompldom  {v})) \explainr{definition $\zeta $}\\
   &\cong \Sigma_{(f,g)\in \Sigma_{f\in \catC(X\times W, Y)}\catC'(X\times W)(\terminal ,Tx\multimap\catL(f)(y))}\catL(W)(w,\Pi_X \catL((\Lambda(f,g),\pi_1))(\catL(\zeta )(\abstrcompldom  {v})) )  \explainr{$\catL$ pseudofunctor }
    \\
&\cong \Sigma_{(f,g)\in \Sigma_{f\in \catC(X\times W, Y)}\catC'(X\times W)(\terminal ,Tx\multimap\catL(f)(y))}\catL(W)(w,\catL(\Lambda(f,g))(\Pi_X \catL(\zeta )(\abstrcompldom  {v})) )  \explainr{Beck--Chevalley for $\Pi$}
\\
&\cong \Sigma_{h\in\catC'(X\times W)(\terminal, \Sigma_Y Tx\multimap y)}\catL(W)(w,\catL(\Lambda(h))(\Pi_X \catL(\zeta )(\abstrcompldom  {v})) )  \explainr{strong $\Sigma$-types in $\catC'$}
\\
&\cong \Sigma_{h\in\catC'(X\times W)\big(\catC'(\pi_1)(\terminal), \Sigma_Y Tx\multimap y\big)}\catL(W)(w,\catL(\Lambda(h))(\Pi_X \catL(\zeta )(\abstrcompldom  {v})) )  \explainr{indexed $\terminal$ in $\catC'$}
\\
&\cong \Sigma_{k\in\catC'(W)(\terminal, \Pi_X\Sigma_Y Tx\multimap y)}\catL(W)(w,\catL(k)(\Pi_X \catL(\zeta )(\abstrcompldom  {v})) )  \explainr{$\Pi$-types in $\catC'$}
\\
& \cong  \Sigma_{k\in\catC(W, \Pi_X\Sigma_Y Tx\multimap y)}\catL(W)(w,\catL(k)(\Pi_X \catL(\zeta )(\abstrcompldom  {v})) )   \explainr{$\terminal \dashv (-.-)$}
    \\
& =  \Sigma_\catC\catL\big((W,w),(\Pi_X\Sigma_Y Tx\multimap y,\Pi_X \catL(\zeta )(\abstrcompldom  {v}) )\big).
\end{align*}
}}\\
Here, we have used the obvious morphisms (again leaving weakening / change of base along projections implicit, for legibility):
\begin{align*}
   & v \in \catC'(X.Y.T x\multimap y)(\terminal,Tx \multimap y) \explainr{representing element of the comprehension}\\
   & \mapsto \abstrcompldom {v} \in \catL(X.Y.T x\multimap y) \explainr{$\abstrcompldom:\terminal \downarrow T(-)\multimap (-)\to \catL$}\\ 
   & \mapsto \catL(\zeta )(\abstrcompldom  {v})\in \catL(\Pi_X\Sigma_Y Tx\multimap y.X)\explainr{change of base along $\zeta $}\\
   &  \mapsto \Pi_X \catL(\zeta )(\abstrcompldom  {v})\in \catL(\Pi_X\Sigma_Y Tx\multimap y) \explainr{$\Pi$-types in $\catL$}\\
\end{align*}
and
$$
\zeta \;=\; \bigl(\,\pi_2,\; \ev\circ(\pi_1\times \id),\; \ev\circ(\pi_2\times \id)\,\bigr)
\;:\;
\Pi_X\Sigma_Y Z\,.\,X \longrightarrow X.Y.Z.
$$
\end{proof}

\begin{remark}
It is useful to note that the first component of $(X,x)\multimap (Y,y)$
is isomorphic to 
\[\Sigma_{X\Rightarrow Y}\Pi_X Tx\multimap \catL(\ev)(y).\]
\end{remark}

Observe that any of the Examples from Section \ref{sec:sufficient-conditions} now give us monoidal closed Grothendieck constructions.
We would like to highlight just a few concrete Examples, because they show up a lot in practice.

\begin{example}[Monoidal closure of $\Fam{-}$-constructions]\label{ex:fam-monoidal-closed}
By Proposition \ref{ex:monoidal-closure}, and Examples \ref{ex:biprod}, \ref{ex:extterm}, \ref{ex:fam}, and \ref{ex:pset},
we have that
\begin{itemize}
    \item $\Fam{\catD}$ is monoidal left-closed (resp., right-closed) for any monoidal left-closed (resp., right-closed) category $\catD$; in this case, the monoidal-closed structure on $\Fam{\catD}$ is fibred over the cartesian closed structure on $\Set$:
    $$\sumfam{D_i}{i\in I}\Rightarrow \sumfam{D'_{i'}}{i'\in I'}=\left\sumfam{\bigsqcap_{i\in I}D_i \Rightarrow D_{f(i)}}{f\in I\Rightarrow I'\right};$$
    \item $\Fam{\catD}$ is (non-fibred) cartesian closed for a category $\catD$ with biproducts and small products (such as $\catD=\CMon$ or $\catD=\CMon^{op}$):
    $$\sumfam{D_i}{i\in I}\Rightarrow \sumfam{D'_{i'}}{i'\in I'}=
    \left\sumfam{{\bigsqcap_{i\in I}D_{\pi_1(f(i))}}}{f\in \Pi_{i\in I}\Sigma_{i'\in I'}\catD(D_i, D'_{i'})\right};$$
    \item $\Fam{\catD^{op}}$ is (non-fibred) cartesian closed for an extensive category $\catD$ with a small coproducts and a terminal object (such as $\catD=\Set$ or $\catD=\Top$):
    $$\sumfam{D_i}{i\in I}\Rightarrow \sumfam{D'_{i'}}{i'\in I'}=
    \left\sumfam{{ \bigsqcup_{i\in I}\abstrcompldom{(\pi_2(f(i)))}}}{f\in \Pi_{i\in I}\Sigma_{i'\in I'}\catD(D'_{i'}, D_i\sqcup \terminal)\right}$$
    here, $\abstrcompldom{(g)}$ should be thought of as the complement of the domain of $g$; in particular, for $\catD=\Set$ that is precisely what it is;
    \item free doubly-infinitary distributive categories $\Dist{\catC}=\Fam{\Fam{\catC^{op}}^{op}}$ are always (non-fibred) cartesian closed (see also \cite{nunes2024free}):
    \begin{align*}
&\sumfam{\prodfam{ C_{ji}}{ i\in I_{j}} }{ j\in J} \Rightarrow 
\sumfam{\prodfam{ C'_{j'i'}}{ i'\in I'_{j}} }{ j'\in J'}
= \\
&\quad\sumfam{\prodfam{ C'_{j'i'} }{  j\in J, \langle j', g\rangle=f(j), i'\in I'_{j'}, g(i')= \langle \bot,\bot\rangle  } }{\\
&\hspace{120pt}f \in \Pi_{j\in J} \Sigma_{j'\in J'}
\Pi_{i'\in I'_{j'}}\Sigma_{i \in I_j\sqcup \{ \bot \} }\catC(C_{ji}, C'_{j'i'}) \text{ if } i\neq \bot\text{ else } \{ \bot \}
};
\end{align*}
using the same formula for the exponentials, we see that finite coproducts of products of $\catC$-objects are exponentiable in the free infinitary distributive category on $\catC$;
further, \cite{prezado2024extensive} 
shows that a similar formula for exponentials also exists in free lextensive categories;
    \item $\Fam{\pSet^{op}}$ is (non-fibred) cartesian closed:
   \small  $$\sumfam{D_i}{i\in I}\Rightarrow \sumfam{D'_{i'}}{i'\in I'}=\left\sumfam{\bigsqcup_{i\in I} D'_{\pi_1(f(i))} \setminus (\pi_2(f(i)))^{-1}(D_i)}{ f\in \Pi_{i\in I}\Sigma_{i'\in I'} \pSet(D'_{i'},D_i)\right};$$
   \normalsize  this example is reminiscent of the variant of the Dialectica interpretation discussed by \cite{biering2008cartesian}.
\end{itemize}
\end{example}

\begin{example}[Monoidal closure of lax comma categories]\label{lax-comma-example}
We can categorify Example \ref{ex:fam-monoidal-closed} by building on Example \ref{ex:ldttt-lax-comma}.
Theorem \ref{theo:grothendieck-closed} tells us that the lax comma category $\Cat//\catD$ is monoidal closed for any small complete category $\catD$ with a $\Sigma,\Cat$-cotractable monoidal structure (including any $\Sigma,\Set$-cotractable one).
In particular, this is true for any small complete category $\catD$ that is monoidal closed, has finite biproducts, or is co-extensive with an initial object.
In the last two cases, $\Cat//\catD$ is cartesian closed.
Similarly, $\Cat//\pSet^{op}$ is cartesian closed.
For example, $\Cat//\CMon$ and $\Cat//\CMon^{op}$ are cartesian closed with exponentials given, respectively, by 
\[
(X,x)\Rightarrow (Y, y)=\left(\Sigma_{X\Rightarrow Y}\Pi_X x\multimap (y\circ\ev), \lim_X y\circ (\pi_1(-)) \right)
\]
and 
\[
(X,x)\Rightarrow (Y, y)=\left(\Sigma_{X\Rightarrow Y}\Pi_X (y\circ\ev)\multimap x, \colim_X y\circ (\pi_1(-)) \right). 
\]
That is, the first component consists of the Grothendieck construction of categories of natural transformations between $x$ and $y\circ \ev$ and the second component consists of a (co)limit in $\CMon$ of $y$ considered as a diagram indexed by $X$ (via the functor $X\to Y$ from the first component).
Note that these exponentials are not fibred.
\end{example}

\begin{example}[Predicate-free Dialectica] \label{ex:dialectica-simple}
    Building on Example \ref{ex:dual-prod-self-indexed},
    we have a symmetric monoidal structure $(U, X)\otimes (V, Y)=(U\times V, X\times Y)$ on $\Dial_{pf}=\Sigma_\catC \self(\catC)^{op}$.
    This has a corresponding closed structure: $(U,X)\multimap (V, Y)=(U\Rightarrow V\times (Y\Rightarrow X),U\times Y)$.
    This is a predicate-free version of the Dialectica interpretation \cite{godel1958bisher}.
    The original Dialectica interpretation has a further fibration of predicates over this category, which we omit as it would distract from the main point of this paper.
    See Section \ref{sec:related-work} and \cite{zbMATH01719167} for details.
\end{example}

\begin{example}[Predicate-free Diller-Nahm]
Building on Example \ref{ex:ldtt-locally-indexed},
assume that $\catD$ is a $\catC$-enriched category with biproducts and $\catC$-copowers $C\otimes D$.
Then, $\catL(C)(D, D')=\catC(C, \catD^{op}(D, D'))$ defines a locally $\catC$-indexed category with $\Pi$-types given by $\Pi_C D=C\otimes D$.
It then follows that $\Dill_{pf}=\Sigma_\catC \catL$ is cartesian closed with products given by $(U, X)\times (V, Y)=(U\times V, X\times Y)$ and exponentials given by $(U\Rightarrow V\times \catD(Y, X), U\otimes Y)$.
This is a predicate-free version of the Diller-Nahm interpretation, where one classically considers the case where $\catD$ is the Kleisli category for an additive monad on $\catC$.
Like the Dialectica interpretation, the Diller-Nahm variant can also be extended with a further fibration of predicates over this category.
See Section \ref{sec:related-work} and \cite{zbMATH01719167} for details.
\end{example}

\begin{example}[Fibred closed structures] \label{ex:diller-nahm-simple}
From the formula given in Theorem \ref{theo:grothendieck-closed}, it is immediately clear that the monoidal left-closed structure on $\Sigma_\catC\catL$ will be fibred, if $\catL$ is an $\catC$-indexed left-closed monoidal category (Proposition \ref{ex:monoidal-closure}), as we can then choose $Tx\multimap y=\terminal\in \catC'(W)$ for all $x,y\in \catL(W)$, resulting in the formula 
\begin{align*}
 (X, x)\leftmultimap  (Y,y)
&\cong (X\Rightarrow Y, \Pi_X x\leftmultimap \catL(\ev)(y)),
\end{align*}
for the left-exponential in $\Sigma_\catC \catL$.
The converse also holds: if the monoidal left-closed structure resulting from Theorem \ref{theo:grothendieck-closed} 
is fibred, then $Tx\multimap y\cong\terminal\in \catC'(W)$.
Then, $\Sigma,\catC$-cotractability of $\otimes$ tells us that
$$
\catL(W)(y\otimes z, x)\cong \Sigma f\in \catC'(W)(\terminal,Tx\multimap y).\catL(W)(z,\abstrcompldom (x,y,f)) \cong\catL(W)(z,\abstrcompldom (x,y,!_\terminal)).
$$
We see that $\catL(W)$ is monoidal left-closed with exponential $y\leftmultimap x=\abstrcompldom (x,y,!_\terminal)$, which is an indexed functor, as $\abstrcompldom$ and $\terminal$ are.
Co-dually, we get fibred right-exponentials in $\Sigma_\catC\catL$ from our Theorem \ref{theo:grothendieck-closed} if and only if $\catL$ is an indexed monoidal right-closed category.

These results are a special case of those of Theorem \ref{thm:fibred-monoidal-closed}.
\end{example}

\begin{example}[Cartesian closure for indexed co-extensive categories]
We build on Example \ref{ex:ldttt-op}.
In the special case that $\catL$ is an indexed extensive category with an indexed terminal object $\terminal$, $\Sigma_\catC \catL^{op}$ is cartesian closed and we have the following formula for exponentials:
\begin{align*}
(X,x)\Rightarrow (Y,y) &=(\Pi_X \Sigma_Y \catL(\pi_2)(y)\multimap (\catL(\pi_1)(x)\sqcup \terminal), \Sigma_X \catL(\zeta )(\compldom{v} )),
\end{align*}
i.e. the second component is the $\Sigma$-type (sum, in $\catL$, so product in $\catL^{op}$) of all complements of the domains $\compldom(g)$ of definition of the morphisms $g:\catL(\pi_2)(y)\multimap \catL(\pi_1)(x)\sqcup \terminal$ (which we think of as partial functions) in the first component.
This special case can be seen as a generalisation of the results of \cite{altenkirch2010higher} on higher-order containers.
\end{example}

\begin{example}[Cartesian closure for indexed coproduct/biproduct categories]
We build on Example \ref{ex:ldttt-biproducts}.
In the special case that $\catL$ is a  $\catC'$-enriched indexed category with indexed coproducts and $\Pi$-types, $\Sigma_\catC \catL$ is symmetric monoidal closed and we have the following formula for exponentials:
$$
(X,x)\multimap (Y, y) =(\Pi_X \Sigma_Y \catL(\pi_1)(x)\multimap \catL(\pi_2)(y),\Pi_X \catL(\evf)(y)),
$$
where we use the obvious morphism 
$$
\evf:\Pi_{X} \Sigma_{Y}Z.X\to Y,
$$
that is, the morphism obtained as the composition (where we write $\pi_1$ for the projection $\Sigma_{Y}Z\to Y$)
$$
\Pi_{X}\Sigma_{Y}Z.X\cong (\Pi_{X}\Sigma_{Y}Z)\times X\xto{(\Pi_{X}\pi_1)\times X}(\Pi_{X}Y)\times X\cong (X\Rightarrow Y)\times X\xto{\mathrm{ev}}Y.
$$
Of particular interest are the cases that 
\begin{itemize}
\item $\catL=\catC'{}^{op}$: in this case, the required $\Pi$-types and coproducts always exist (as $\catC'$ has $\Sigma$-types)
and the $\catC'$-enrichment exists as $\catC'$ has $\Pi$-types so its fibres are cartesian closed; we conclude that for any model $\catC':\catC^{op}\to\CAT$ of dependent type theory with $\Pi$-types and strong $\Sigma$-types, $\Sigma_\catC \catC'{}^{op}$ is symmetric monoidal closed;
\item $\catL$ has biproducts: in this case, the monoidal structure, if it exists is a cartesian one, giving us a cartesian closed structure on $\Sigma_\catC \catL$, assuming that the required $\catC'$-enrichment and $\Pi$-types exist; further, observe that $\Sigma_\catC \catL^{op}$ is then also cartesian closed as long as the required $\Sigma$-types exist in $\catL$:
$$
(X,x)\Rightarrow (Y, y) =(\Pi_X \Sigma_Y  \catL(\pi_2)(y)\multimap \catL(\pi_1)(x),\Sigma_X \catL(\evf)(y))
$$
\end{itemize}
This shows that we reproduce the results of \cite[Proposition 4.6.1]{moss_2018} and \cite[Section 6.4]{nunes2023chad}, as a special case of Theorem \ref{theo:grothendieck-closed}.
\end{example}

Finally, we can also use our result as a tool to show that a monoidal structure is not $\Sigma$-cotractable.
\begin{counterexample}\label{ex:fam-mon-not-ccc}
Let $\catD$ be a infinitary distributive category, i.e. a category with small coproducts and finite products such that $\bigsqcup: \Fam{\catD}\to\catD$
preserves finite products.
Suppose further that $\catD$ is not cartesian closed.
For example, $\catD$ could be the category of locally connected topological spaces and continuous functions \cite[Example 8]{nunes2024free} or the category of finite dimensional smooth manifolds of varying dimension and smooth functions \cite[Appendix A]{hsv-fossacs2020}.
Then, by \cite[Theorem 4.2]{nunes2024free},
$\Fam{\catD}$ is not cartesian closed.
As a consequence, by Example \ref{ex:fam-monoidal-closed}, it follows that the products in $\catD$ are not $\Sigma$-cotractable.
\end{counterexample}

\section{Related Work and Outlook}\label{sec:related-work}

We situate our results within the broader literature, tracing their antecedents in the Dialectica and Diller–Nahm interpretations, in lax comma $2$-categories, and in connections with higher order containers and CHAD.  We then relate them to freely generated categorical structures and their dependently typed extensions, before concluding with prospects for further $\Sigma$-cotractable monoidal structures and observations on efficient implementation.

\subsection{Dialectica and Diller–Nahm interpretations (with predicates)}

The categorical study of Dialectica constructions originates in Valeria de Paiva’s Cambridge thesis and her paper in the volume \emph{Categories in Computer Science and Logic}~\cite{zbMATH04104952, zbMATH07648691}.

The earliest examples of similar techniques for constructing exponentials on Grothendieck constructions that we are aware of arose in proof theory when demonstrating the relative consistency of Heyting arithmetic: G\"odel's Dialectica interpretation~\cite{zbMATH04104952, zbMATH07648691, zbMATH01719167, godel1958bisher} and Diller and Nahm's $\CMon$-enriched variant of that interpretation~\cite{diller1974variante} also show a similar presentation.
In Examples~\ref{ex:dialectica-simple} and~\ref{ex:diller-nahm-simple} we give simplified (predicate-free) presentations $\Dial_{pf}$ and $\Dill_{pf}$ of these constructions.
Here, we briefly point out how to extend them with predicates, following \cite{zbMATH01719167}'s categorical presentation of these interpretations.
We first quote the presentation in \cite{zbMATH01719167} for definitions, and next briefly explain how the closed structures are obtained from Theorem~\ref{theo:grothendieck-closed} by building on the closed structures described in Examples~\ref{ex:dialectica-simple} and~\ref{ex:diller-nahm-simple}.

{\color{gray}\begin{quote}[\cite{zbMATH01719167}, Dialectica]
		Suppose that we have a category $T$ which we can think of as interpreting some type
		theory; and suppose that over the category $T$ we have a pre-ordered fibration $p : P \to T$,
		which we can regard as providing for each $I \in T$ a pre-ordered collection of (possibly
		non-standard) predicates $P(I )=(P(I ), \vdash)$. Starting with this data we construct a new
		category $\Dial = \Dial(p)$ which we regard as a category of propositions and proofs.
		We do this as follows.
		\begin{itemize}
			\item The objects $A$ of $\Dial$ are $U, X \in T$ together with $ \alpha\in P(U \times X )$. $(\cdots)$
			Our understanding of the predicate $\alpha$ is not symmetric as regards $U$ and $X$ : we read
			$\alpha$ as $\exists u \in U.\forall x\in X.\alpha(u, x)$, in accord with the form of propositions in the image of the
			Dialectica interpretation.
			\item  Maps of $\Dial$ from $A = (U,X,\alpha)$ to $B = (V, Y,\beta )$ are $(\cdots)$ of the form $f:U\to V$, $F:U\times Y\to X$
			with $\alpha(u, F(u, y)) \vdash \beta(f(u), y)$ in $P(U \times Y )$.
		\end{itemize}
\end{quote}}

We can observe that\footnote{
	We use the locally indexed category $\self(T)$ for the category with products $T$ here.
	See Example~\ref{ex:locally-indexed}.}
\[
\Dial \;\text{is precisely the category}\; \Sigma_{U\in T} \Dial(U)
\qquad \text{for} \qquad
\Dial(U)=\bigl(\Sigma_{X\in \self(T)(U)} P(U \times X)^{op}\bigr)^{op}.
\]
It is a more involved version of the category discussed in Example~\ref{ex:dialectica-simple}, where we additionally endow all objects with predicates.

If we assume that $P\to T$ is fibred cartesian closed, it follows from our Theorem~\ref{theo:grothendieck-closed} that $\Dial$ is monoidal closed for the (symmetric) monoidal structure $(U,X,\alpha)\otimes (V, Y,\beta)=(U\times V, X\times Y, \alpha\wedge \beta)$.
Then
$$(V, Y,\beta)\multimap (W, Z, \gamma)
=
\bigl( (V\Rightarrow W ) \times (V \times Z \Rightarrow Y ),\; V\times Z,\; \rho \bigr),$$
$$\rho((g, G),(v, z)) = \beta(v, G(v, z)) \Rightarrow \gamma(g(v), z).$$
Indeed, the indexed monoidal structure on $\Dial(U)$ with unit $(\terminal, \top)$ and product
\[
(X,\alpha)\otimes (X',\alpha')=\bigl(X\times X',\, P(\id\times \pi_1)(\alpha)\wedge P(\id\times\pi_2)(\alpha')\bigr)
\]
is $\Sigma,\self(T)(U)$-cotractable because
\begin{align*}
	&\Dial(U)\bigl((X, \alpha) \otimes (X', \alpha'), (X'', \alpha'')\bigr)=\\
	&\{F: U \times X''\to X \times X' \mid \alpha(u, \pi_1 (F(u, x''))) \wedge \alpha'(u, \pi_2 (F(u, x''))) \vdash \alpha''(u, x'') \} \\
	&=\Sigma\, F_1 \in \self(T)(U)(\terminal, X'' \Rightarrow X)\,.\, \Dial(U)\bigl((X', \alpha'), (X'', \rho)\bigr),
\end{align*}
where $\rho(u, x'')=\alpha\bigl(u, F_1(u)(*)(x'')\bigr) \Rightarrow \alpha''(u, x'')$.
Therefore,
\begin{align*}
	&\Dial(U)\bigl((X', \alpha'), (X'', \rho)\bigr) \\
	&\quad = \{F_2 : U \times X'' \to X' \mid \alpha'(u, F_2(u, x'')) \vdash \alpha\bigl(u, F_1(u)(*)(x'')\bigr) \Rightarrow \alpha''(u, x'') \},
\end{align*}
showing that we can choose $T(X'', \alpha'') \multimap (X, \alpha) = X'' \Rightarrow X$ and $\abstrcompldom\bigl((X'', \alpha''), (X, \alpha), F_1\bigr) = (X'', \rho)$, where $\rho(u, x'') = \alpha\bigl(u, F_1(u)(*)(x'')\bigr) \Rightarrow \alpha''(u, x'')$.
Further, the indexed category $U\mapsto \Dial(U)$ has $\Pi$-types, given by $\Pi_{V} (X, \alpha)= (V\times X,P(\pi_2)(\alpha))$, meaning that the assumptions of Theorem~\ref{theo:grothendieck-closed} are met.

{\color{gray}\begin{quote}[\cite{zbMATH01719167}, Diller–Nahm]
		Suppose again that we have a pre-ordered set fibration $p : P \to T$, providing for each
		type $I \in T$ a collection of (possibly non-standard) predicates $P(I )$ over $I$. We need
		some additional structure. We suppose that $p : P \to T$ is equipped with a commutative
		monoid $(-)^\bullet$ in the following sense.
		\begin{itemize}
			\item Firstly, $T$ is a category with products and $(-)^\bullet$ is a strong monad on $T$ such that
			each algebra is equipped naturally with the structure of a commutative monoid.
			\item Secondly, we suppose that we have an indexed extension of $(-)^\bullet$ to $P$. For
			$\phi \in  P(I\times A) $ we have $\phi^\bullet \in P(I \times X^\bullet)$. For each $I \in T$, the strength gives an action
			of $(-)^\bullet$ on the (simple slice) category $T/I$ . And the operation $\phi \to \phi^\bullet$ just described
			is an extension of this to the global category $P/I \to T/I$.
		\end{itemize}
		The example to have in mind here is the finite multiset monad on the category of
		sets; of course, that is exactly the monad whose algebras are commutative monoids.
		This monad extends naturally to the subset lattices: if $\phi \subseteq I \times X$ then $\phi^\bullet \subseteq I \times X^\bullet$ is
		defined by
		$\phi^\bullet(i, \xi)$ if and only if $\forall x \in \xi.\phi(i,x)$.
		From the data just described we construct a new
		category $\Dill = \Dill(p)$ which we regard again as a category of propositions and proofs.
		\begin{itemize}
			\item The objects of $\Dill$ are still pairs $U, X \in T$ together with $ \alpha\in P(U \times X )$. $(\cdots)$
			\item  Maps of $\Dill$ from $A = (U,X,\alpha)$ to $B = (V, Y,\beta )$ are $(\cdots)$ of the form $f:U\to V$, $F:U\times Y\to X^\bullet$
			with $\alpha^\bullet(u, F(u, y)) \vdash \beta(f(u), y)$ in $P(U \times Y )$.
		\end{itemize}
\end{quote}}

That is,
\[
\Dill \;\text{is precisely the category}\;
\Sigma_{U\in T}\Dill(U)
\qquad\text{for} \] 
\[\Dill(U)=\Bigl(\mathrm{Kleisli}((-)^\bullet)\bigl(\Sigma_{X\in \self(T)(U)} P(U \times X)^{op}\bigr)\Bigr)^{op},
\]
for the lifted monad $(-)^\bullet$ on $\Sigma_{X\in \self(T)(U)} P(U \times X)^{op}$.
It is a more involved version of the category discussed in Example~\ref{ex:diller-nahm-simple}, where we additionally endow all objects with predicates.

If we assume that $P\to T$ is a fibred cartesian closed category over a bicartesian closed category $T$, and that $P$ is an extensive indexed category in the sense that we have a natural isomorphism $[-]:\bigsqcap_{i=1}^N P(X_i)\cong P(\bigsqcup_{i=1}^N X_i)$, and that $(-)^\bullet$ is an additive monad in the sense that $(\bigsqcup_{i=1}^N X_i)^\bullet \cong \bigsqcap_{i=1}^N X_i^\bullet$ (we will abuse notation slightly and leave these two isomorphisms implicit), then it follows from our Theorem~\ref{theo:grothendieck-closed} that $\Dial$ has the cartesian closed structure $(U,X,\alpha)\times (V, Y,\beta)=(U\times V, X\sqcup Y, [\alpha, \beta])$ and
\[
(V, Y,\beta)\Rightarrow (W, Z, \gamma)
=
\bigl( (V\Rightarrow W ) \times (V \times Z \Rightarrow Y^\bullet ),\; V\times Z,\; \rho \bigr),
\]
where $\rho((g, G),(v, z)) = \beta^\bullet(v, G(v, z)) \Rightarrow \gamma(g(v), z)$.

Indeed, the indexed product structure on $\Dill(U)$ with unit $(\initial, [])$ and product
\[
(X,\alpha)\times (X',\alpha')=(X\sqcup X', [\alpha,\alpha'])
\]
is $\Sigma,\self(T)(U)$-cotractable because \small 
\begin{align*}
	&\Dill(U)\bigl((X, \alpha) \times (X', \alpha'), (X'', \alpha'')\bigr)=\\
	&\{F: U \times X''\to X^\bullet \times X'{}^\bullet\cong (X\sqcup X')^\bullet \mid \alpha^\bullet(u, \pi_1 (F(u, x''))) \wedge \alpha'{}^\bullet(u, \pi_2 (F(u, x''))) \vdash \alpha''(u, x'') \} \\
	&=\Sigma\, F_1 \in \self(T)(U)(\terminal, X'' \Rightarrow X)\,.\, \Dill(U)\bigl((X', \alpha'), (X'', \rho)\bigr),
\end{align*} \normalsize 
where $\rho(u, x'')=\alpha^\bullet\bigl(u, F_1(u)(*)(x'')\bigr) \Rightarrow \alpha''(u, x'')$.
Therefore,
\begin{align*}
	&\Dill(U)\bigl((X', \alpha'), (X'', \rho)\bigr) \\
	&\quad= \{F_2 : U \times X'' \to X'{}^\bullet \mid \alpha'{}^\bullet(u, F_2(u, x'')) \vdash \alpha^\bullet\bigl(u, F_1(u)(*)(x'')\bigr) \Rightarrow \alpha''(u, x'') \},
\end{align*}
showing that we can choose $T(X'', \alpha'') \multimap (X, \alpha) = X'' \Rightarrow X^\bullet$ and $\abstrcompldom\bigl((X'', \alpha''), (X, \alpha), F_1\bigr) = (X'', \rho)$, where $\rho(u, x'') = \alpha^\bullet\bigl(u, F_1(u)(*)(x'')\bigr) \Rightarrow \alpha''(u, x'')$.
Further, the indexed category $U\mapsto \Dill(U)$ has $\Pi$-types, given by $\Pi_{V} (X, \alpha)= (V\times X,P(\pi_2)(\alpha))$, meaning that the assumptions of Theorem~\ref{theo:grothendieck-closed} are met.

These Examples raise the more general question under what circumstances, given two fibrations $p\colon P\to Q$ and $q\colon Q\to R$, the fibration $(q\circ p^{op})^{op}$ (using the fibrewise opposite fibration and composition of fibrations) has a monoidal closed total space.
Assuming that $q$ is a model of dependent type theory with $\Pi$-types and strong $\Sigma$-types, that amounts, in the light of our Theorem~\ref{theo:grothendieck-closed}, to characterising when an indexed monoidal structure on the fibres of $(q\circ p^{op})^{op}$ is $\Sigma, Q(-)$-cotractable and when $(q\circ p^{op})^{op}$ has $\Pi$-types.

\subsection{Lax comma $2$-categories}
There has recently been a renewed interest in the study of lax comma $2$-categories in the literature; see, for example, \cite{arXiv:2504.12965, zbMATH08028627, zbMATH08058043, zbMATH07844805, zbMATH07766161, clementino2024lax, zbMATH08084496, zbMATH07558575}.
As illustrated in Examples~\ref{ex:ldttt-lax-comma} and~\ref{lax-comma-example}, lax comma categories fall within the scope of our results.
In particular, we have established that
\[
\CAT \sslash X
\]
is complete and cartesian closed whenever $X$ is complete and cartesian closed.
Moreover, we have shown that
\[
\CAT \sslash \CMon
\]
is cartesian closed.
This strengthens existing results in the literature, such as~\cite{zbMATH07766161}, which considered only the case of fibred exponentials.

We also observe that lax comma $2$-categories, in full generality, extend the Grothendieck construction, which they subsume as a special case.
Thus, the present work points towards a more general theory of closed structures on lax comma $2$-categories.
Developing such a theory is the subject of ongoing research.

\subsection{Higher order containers}
\cite{altenkirch2010higher} previously gave the special case of our formula for exponentials in $\Fam{\Set^{op}}=\Sigma_\Set\CAT(-,\Set^{op})$, which they interpret as a category of containers (or polynomial endofunctors).
Such containers are useful in programming as they give a certain, concrete representation of datatypes.
As such, the authors use it to give a notion of “higher-order container”.
Our construction shows that the same construction can be carried out for more general notions of containers valued in a category with a $\Sigma$-tractable monoidal structure, such as an extensive category with its coproduct structure or a category with biproducts.
Some examples of such containers (such as additive containers, as in CHAD, see below) have already found useful programming applications.
However, we believe there might be potential for many more notions of container and lens to find use in programming.
We hope that the formulas given in this work can contribute to principled programming idioms for such data representations.

\subsection{CHAD: Combinatory Homomorphic Automatic Differentiation}
Recent work~\cite{vakar2020reverse,vakar2021chad,nunes2023chad} has analysed the special case of our formula for exponentials in the case that the fibre categories $\catL$ have biproducts.
They show that this case can be used to prove correctness (see loc.\ cit.) and to give an efficient implementation~\cite{smeding2024efficient} of a programming technique called Automatic Differentiation (AD), typically the method of choice for efficiently computing derivatives of numerical programs.
It is tempting to give a similar analysis, based on Grothendieck constructions, for reverse-mode AD methods for calculating higher derivatives~\cite{betancourt2018geometric,huot2021higher}.

\subsection{\textbf{Freely generated categorical structures}}
The case of our formula for exponentials in Grothendieck constructions $\Sigma_\catC \catL$ indexed cartesian closed categories indexed by a cartesian closed category seems to be well known.
It is used, in particular, for the case of families $\Fam{\catD}=\Sigma_\Set \CAT(-,\catD)$ valued in a cartesian closed category $\catD$ (that is, the freely generated category with small coproducts on $\catD$) by~\cite{zbMATH07186728}.

Recently, \cite{nunes2024free} and \cite{prezado2024extensive} analysed exponentiability in freely generated distributive and lextensive categories generated from an arbitrary locally small category $\catD$ (which need not be cartesian closed), respectively.
The formula used for the exponentials arises as a special case of the present work.
These works raise the question whether our method is suitable for a study of exponentiability in further kinds of freely generated categorical structures.

\subsection{Dependently typed Dialectica}
In a recent tour de force, \cite{von_glehn_2015,glehnmoss2018,moss_2018} showed that the Dialectica and Diller–Nahm interpretations can be extended to dependently typed languages.
In particular, they show the following two results, which are in a sense dependently typed variants of two of our examples:
\begin{itemize}
	\item starting from a model of dependent type theory with strong $\Sigma$-types, $\Pi$-types and identity types that is \emph{extensive} in a suitable sense, they construct another model of dependent type theory with $\Sigma$-types, $\Pi$-types and identity types, generalizing our Example~\ref{ex:ldttt-op}, in a sense;
	\item starting from a model of dependent type theory with strong $\Sigma$-types, $\Pi$-types and identity types with an \emph{additive monad, using a Kleisli construction}, they construct another model of dependent type theory with $\Sigma$-types, $\Pi$-types and identity types; this is closely related to but not quite a generalisation of our Example~\ref{ex:ldttt-biproducts}.
\end{itemize}
Compared to their work, on the one hand, we do not consider the considerable amount of structure needed to interpret dependent types in a Grothendieck construction, so in this sense our work is more limited.
On the other hand, we generalise from two examples of products in extensive categories and Kleisli categories of additive monads to $\Sigma$-cotractable monoidal structures.
The latter also give rise to various new examples of cartesian and non-cartesian (even non-symmetric) monoidal closed structures on Grothendieck constructions.
In that sense, our work is more general.

Dependently typed ($\Sigma$-type) equivalents of non-cartesian monoidal structures are surprisingly subtle~\cite{vakar2017search}, so it is not clear if a common generalisation of both approaches is possible.
The most promising avenue might be to limit oneself to cartesian type theories and to pursue a notion of $\Sigma$-cotractable $\Sigma$-type to generalise $\Sigma$-cotractable binary products as well as the examples in~\cite{glehnmoss2018}.

\subsection{Other $\Sigma$-(co)tractable monoidal structures}
So far, we have shown that typical examples of $\Sigma$-cotractable monoidal structures are:
\begin{itemize}
	\item coproducts in any category;
	\item products in a coextensive category with an initial object;
	\item a monoidal left-closed structure on a category with an initial object;
	\item products in $\pSet^{op}$.
\end{itemize}
In fact, we have seen that for posets (and, more generally, preorders) $\Sigma$-cotractability of the product is equivalent to cartesian closure plus an initial object.
For non-thin categories, we have no such characterisation of $\Sigma$-cotractability.
This raises the question whether there are other interesting, naturally occurring examples of $\Sigma$-cotractable products and non-cartesian monoidal structures for non-thin categories.

\subsection{Efficient implementation}
\cite{smeding2024efficient} shows that, when \cite{nunes2023chad} is  interpreted as a recipe for generating code in a functional programming language, programs that make use of the Dialectica-like monoidal closed structure presented in this paper can be inefficient.
Interestingly, the non-fibred nature of the exponentials can result in recomputation.
In the particular example of CHAD-style automatic differentiation, a workaround is possible by closure converting the code, essentially by using a representation for the exponential as a coend via the co-Yoneda lemma.

Containers and lenses are an increasingly important data representation, particularly in machine learning applications where data needs to flow in both directions~\cite{cruttwell2022categorical}.
Therefore, it would be interesting to have a better understanding of the precise nature of these efficiency pathologies arising for higher-order containers, as well as generally applicable solutions.

\section*{Acknowledgements}

The authors thank the anonymous referee for a careful reading of the manuscript and for several valuable suggestions that have improved its presentation.

The first author gratefully acknowledges support from the Fields Institute for Research in Mathematical Sciences through a Fields Research Fellowship (2023), as well as ongoing support from the Centre for Mathematics of the University of Coimbra (CMUC) — UIDB/00324/2020, funded by the Portuguese Government through FCT/MCTES.  
The author also expresses sincere gratitude to Henrique Bursztyn and the Instituto Nacional de Matemática Pura e Aplicada (IMPA) for their generous hospitality.

This project has further benefited from support by the NWO Veni grant VI.Veni.201.124 and the ERC project \textsc{FoRECAST}.


\end{document}